\documentclass[a4paper,reqno]{amsart}

\usepackage{lmodern}
\usepackage{slantsc}
\newcommand{\slkz}{{\normalfont\textsl{\textsc{kz}}}}
\newcommand{\kz}{{\normalfont\textsc{kz}}}
\newcommand{\mathkz}{{\normalfont\kz}}

\usepackage[all,2cell]{xy}
\SelectTips{cm}{10}
\usepackage{mathrsfs,amssymb,stmaryrd,bbm,mathabx,mathtools}
\xyoption{v2}
\UseAllTwocells
\usepackage{amsmath}
\numberwithin{equation}{section}

\usepackage[bookmarks=true]{hyperref}
\usepackage{autonum}
\usepackage[textsize=scriptsize,color=orange]{todonotes}
\usepackage{gitinfo}
\usepackage[inline]{enumitem}
\usepackage[pagewise,modulo]{lineno}

\theoremstyle{plain}
\newtheorem{thm}[equation]{Theorem}
\newtheorem{cor}[equation]{Corollary}
\newtheorem{lemma}[equation]{Lemma}
\newtheorem{prop}[equation]{Proposition}
\newtheorem*{thm*}{Theorem}

\theoremstyle{remark}

\newtheorem{rmk}[equation]{Remark}

\newtheorem{ex}[equation]{Example}
\newtheorem{notation}[equation]{Notation}

\theoremstyle{definition}

\newtheorem{df}[equation]{Definition}

\newcommand{\V}{\ensuremath{{V}}}
\newcommand{\VCat}{\ensuremath{\V\text-\mathbf{Cat}}}
\newcommand{\Cat}{\ensuremath{\mathbf{Cat}}}
\newcommand{\C}{\ensuremath{\mathcal{C}}}
\newcommand{\id}{\ensuremath{\mathrm{id}}}
\DeclareMathOperator{\dom}{dom}
\DeclareMathOperator{\cod}{cod}

\newcommand{\K}{\ensuremath{\mathscr{K}}}
\newcommand{\two}{{\mathbf{2}}}

\newcommand{\lperp}{{\pitchfork_{\mathkz}}}
\newcommand{\lperpleft}{\lperp}

\newcommand{\Sq}{\mathbb{S}\mathrm{q}}

\newcommand{\col}{\operatorname{colim}}

\newcommand{\dd}{\mathsf{Diag}}
\newcommand{\tor}{\ensuremath{\relbar\joinrel\mapstochar\joinrel\rightarrow}}

\begin{document}
% \linenumbers
\title{Lax orthogonal factorisation systems}
\author[M~M~Clementino]{Maria Manuel Clementino}
\address{CMUC, Department of Mathematics\\ Univ.~Coimbra\\3001-501 Coimbra\\ Portugal}
\thanks{
  Research partially supported by Centro de Matem\'{a}tica da Universidade de
  Coimbra -- UID/MAT/00324/2013, funded by the Portuguese Government
  through FCT/MCTES and co-funded by the European Regional Development Fund
  through the Partnership Agreement PT2020.
}
\email{mmc@mat.uc.pt}
\author[I~L\'{o}pez~Franco]{Ignacio L\'{o}pez Franco}
\address{Department of Pure Mathematics and Mathematical Statistics\\Centre for
  Mathematical Sciences\\University of Cambridge\\Wilberforce
  Road\\Cambridge\\CB2 0WB\\UK}
\thanks{The second author acknowledges the support of a Research Fellowship of
  Gonville and Caius College, Cambridge, and of the Australian Research Council Discovery Project DP1094883.}
\email{ill20@cam.ac.uk}
%\date\today
%\date{Compiled on \today. \gitAuthorIsoDate. Revision \gitAbbrevHash, \gitReferences}
\begin{abstract}
  This paper introduces \emph{lax orthogonal algebraic weak factorisation
    systems} on 2-categories and describes a method of constructing them.  This
  method rests in the notion of \emph{simple 2-monad}, that is a generalisation
  of the simple reflections studied by Cassidy, H\'{e}bert and Kelly. Each
  simple 2-monad on a finitely complete 2-category gives rise to a lax
  orthogonal algebraic weak factorisation system, and an example of a simple
  2-monad is given by completion under a class of colimits. The notions
  of \textsl{\textsc{kz}} \emph{lifting operation}, \emph{lax natural lifting
    operation} and \emph{lax orthogonality} between morphisms are studied.
\end{abstract}

\maketitle

% \tableofcontents{}

\section{Introduction}
\label{sec:introduction}

This paper contains four main contributions: the introduction of \emph{lax
  orthogonal algebraic weak factorisation systems} (\textsc{awfs}s); the introduction of the
concept of \textsl{\textsc{kz}} \emph{diagonal fillers} and the study of their relationship to lax
orthogonal \textsc{awfs}s; the introduction of \emph{simple 2-monads}, and the proof that
each such
induces an \textsc{awfs}s; the proof that 2-monads given by completion under a
class of colimits are simple and the study of the induced factorisations.

Weak factorisation systems form the basic ingredient of \emph{Quillen model
  structures}~\cite{MR0223432}, and, as the name indicates, are a weakening of
the ubiquitous orthogonal factorisation systems. A weak factorisation
system (\textsc{wfs}) on a category consists of two classes of morphisms $\mathcal{L}$ and
$\mathcal{R}$ satisfying two properties: every morphism can be written as a
composition of a morphism in $\mathcal{L}$ followed by one in $\mathcal{R}$, and
for any commutative square, with vertical morphisms in $\mathcal{L}$ and
$\mathcal{R}$ as depicted in~\eqref{eq:16}, there exists a diagonal filler. One says that $r$ has
the right lifting property with respect to $\ell$ and that $\ell$ has the left
lifting property with respect to $r$.
\begin{equation}
  \label{eq:16}
  \xymatrixrowsep{.6cm}
  \diagram
  \cdot\ar[r]\ar[d]_{\mathcal{L}\ni\ell}&
  \cdot\ar[d]^{r\in\mathcal{R}}\\
  \cdot\ar@{..>}[ur]^\exists\ar[r]&
  \cdot
  \enddiagram
\end{equation}
When $r$ is the unique map to the terminal object, one usually says that $C$ is injective with
respect to $\ell$.

In order to unify the study of injectivity with respect to
different classes of continuous maps between $T_0$ topological spaces,
Escard\'{o}~\cite{MR1641443} employed lax idempotent 2-monads, also known as
\textsc{kz}~2-monads, on poset-enriched categories --~these are the same as 2-categories whose
hom-categories are posets. For example, if $\mathsf{T}$ is such a lax idempotent 2-monad,
the $\mathsf{T}$-algebras can be described as the objects $A$ that are injective to all the
morphisms $\ell$ such that $T\ell$ is a coretract left adjoint --~also known as a $\mathsf{T}$-embedding. A central point
is that not only each morphism $\dom(\ell)\to A$ has an extension along
$\ell$, but moreover it has a \emph{least} extension: one that is smallest
amongst all extensions.
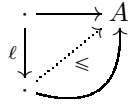
\begin{equation}
  \label{eq:17}
  \xymatrixrowsep{.6cm}
  \diagram
  \cdot\ar[r]\ar[d]_\ell&A\\
  \cdot\ar@{..>}[ur]\urlowertwocell{\omit}\ar@{}@<-6pt>[ur]|\leq&
  \enddiagram
\end{equation}
The assignment that sends a morphism to its least extension can be
described in terms of the 2-monad $\mathsf{T}$, so one no longer has the \emph{property}
of the existence of at least one extension, but the \emph{algebraic structure}
that constructs the extension. If one wants to describe \textsc{wfs}s in this context,
instead of just injectivity, one is led to consider algebraic weak
factorisation systems (\textsc{awfs}s), to which we shall return in this introduction.

\subsection*{Injective continuous maps }
\label{sec:inject-cont-maps}
One of the basic examples that fit in the framework of~\cite{MR1641443} is that
of the filter monad on the category of $T_0$ spaces, that assigns to each space
its space of filters of open sets. It was shown in~\cite{MR0367013} that the
algebras for this monad are the topological spaces that arise as continuous
lattices with the Scott topology. These spaces were known to be precisely those
injective with respect to subspace embeddings~\cite{MR0404073}.
In~\cite{MR2927175} this and other related results are generalised,
characterising those continuous maps of $T_0$ spaces that have the right lifting
property with respect to different classes of embeddings, and exhibiting for
each a \textsc{wfs} in the category of $T_0$ spaces. A morphism
$f\colon X\to Y$ is factorised through the subspace $Kf\subseteq TX\times Y$ of
those $(\varphi,y)$ such that $Tf(\varphi)\leq \{U\in\mathcal{O}(Y):y\in
U\}$. The space $TX$ can be the topological space of filters of open sets of $X$
or a variant of it, and $i_X\colon X\to TX$ the inclusion of $X$ as the set of
principal filters, $i_X(x)=\{U\in \mathcal{O}(X): x\in U\}$. The space $Kf$ fits
in a diagram as displayed. The maps $q_f$ and
$Rf$ send $(\varphi,y)\in Kf$ to $\varphi\in TX$ and $y\in Y$
respectively. The inequality symbol inside the square denotes the fact that
$Tf\cdot q_f\leq i_Y\cdot Rf$.
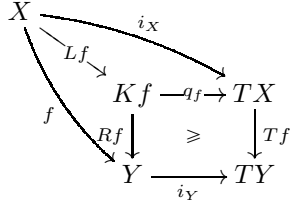
\begin{equation}
  \label{eq:21}
  \xymatrixrowsep{.6cm}
  \diagram
  X\ar[dr]|{Lf}\ar@/^/[drr]^{i_X}\ar@/_/[ddr]_f&&\\
  &Kf\ar[d]_{{R}f}\ar[r]|{q_f}\ar@{}[dr]|\geq &
  TX\ar[d]^{Tf}\\
  &Y\ar[r]_-{i_Y}&
  TY
  \enddiagram
\end{equation}
Central to the arguments in~\cite{MR2927175} is the fact that the monad
$f\mapsto{R}f$ is lax idempotent or \textsc{kz}. This property is intimately linked
with the fact that $Lf$ is always an embedding of the
appropriate variant --~eg, when $TX$ is the space of all filters of open sets,
then $Lf$ is a topological embedding. % --~ie $\lambda_f$ is a $\mathsf{T}$-embedding.

The construction of the factorisation of maps just described resembles the
classical case of simple reflections~\cite{MR779198}. One of the aims of the present paper is to
show that both constructions are particular instances of a more general one.

\subsection*{Algebraic weak factorisation systems}
\label{sec:algebr-weak-fact}
Algebraic weak factorisation systems (\textsc{awfs}s) were introduced in~\cite{MR2283020}
with the name \emph{natural weak factorisation systems}, with a distributive
axiom later added in~\cite{MR2506256}. The theory of \textsc{awfs}s has been
developed in~\cite{MR2506256} and~\cite{MR3393453}, especially with respect to
their relationship to double categories and to cofibrant generation. The present
paper takes the theory in a new direction, that of \textsc{awfs}s on
2-categories whose lifting operations, or diagonal fillers, have a universal
property with respect to 2-cells.

Many of the factorisation systems that
occur in practice provide a construction for the factorisation of an arbitrary
morphism. Such a
structure on a category $\mathcal{C}$ is called a \emph{functorial
  factorisation} and can be described in several equivalent ways: as a functor
$\mathcal{C}^\two\to\mathcal{C}^{\mathbf{3}}$ that is compatible with domain and
codomain; as a codomain-preserving --~ie with identity codomain component~--
pointed endofunctor $\Lambda\colon 1\Rightarrow R$ of $\mathcal{C}^\two$; as a
domain-preserving copointed endofunctor $\Phi\colon L\Rightarrow 1$ of
$\mathcal{C}^\two$. Then, a morphism $f$ factors as $f=Rf\cdot{}Lf$. Any such
functorial factorisation has an underlying \textsc{wfs} $(\mathcal{L},\mathcal{R})$ where
$\mathcal{L}$ consists of those morphisms that admit an $(L,\Phi)$-coalgebra
structure and $\mathcal{R}$ of those that admit an $(R,\Lambda)$-algebra
structure. One usually wants, however, to guarantee that $Lf\in\mathcal{L}$ and
$Rf\in\mathcal{R}$, for which one requires extra data in the form of a
comultiplication that makes $(L,\Phi)$ into a comonad $\mathsf{L}$ and a
multiplication that makes $(R,\Lambda)$ into a monad $\mathsf{R}$.
The pair $(\mathsf{L},\mathsf{R})$ together with an extra distributivity
condition is called an \textsc{awfs}.

The underlying \textsc{wfs} of an \textsc{awfs} $(\mathsf{L},\mathsf{R})$ is an orthogonal factorisation
system precisely when $\mathsf{L}$ and $\mathsf{R}$ are
idempotent~\cite{MR2283020}; for this, it is enough if either is idempotent~\cite{MR3393453}.

All the above constructions can be performed on 2-categories instead of
categories. Two morphisms $\ell\colon A\to B$ and $r\colon C\to D$ in a 2-category
$\K$  are \emph{lax orthogonal} when the comparison morphism
\begin{equation}
  \K(B,C)\longrightarrow\K(A,C)\times_{\K(A,D)}\K(B,D)\label{eq:401}
\end{equation}
has a left adjoint coretract. --~In the
the usual definitions of weak orthogonality and orthogonality this morphism must
be an epimorphism and, respectively, an isomorphism.~-- This left adjoint
provides diagonal fillers that moreover satisfy a universal property with
respect to 2-cells. A choice of diagonal fillers like these that is in
addition natural with respect to $\ell$ and $r$ we call a \textsl{\textsc{kz}}~\emph{lifting operation}.

When the 2-category \K\ is locally a preorder, the lax orthogonality of $\ell$
and $r$ reduces to the statement, encountered before in this introduction, that
for each commutative square~\eqref{eq:16} there exists a least diagonal
filler.

The notion of \textsc{awfs} on a 2-category we choose is the straightforward
generalisation of the usual notion of \textsc{awfs} on a category. If $\K$ is a 2-category, an
\textsc{awfs} on \K\ consists of a 2-comonad $\mathsf{L}$ and a 2-monad
$\mathsf{R}$ on $\K^\two$ that form an \textsc{awfs} on the underlying category
of $\K^\two$, and that satisfy $\cod L=\dom R$ as 2-functors; the definition can
be found in Section~\ref{sec:definition}.

An interesting question is what is the property on an \textsc{awfs} that
corresponds to the existence of \textsc{kz}~lifting operations. The answer is
that both the 2-comonad and the 2-monad of the \textsc{awfs} must be \emph{lax
  idempotent} --~proved in Theorem~\ref{thm:5}. Equivalently, either the 2-comonad \emph{or} the
2-monad must be lax idempotent --~proved in Section~\ref{sec:2-comonad-lax}. This last
statement mirrors the case of \textsc{awfs}s whose underlying \textsc{wfs} is orthogonal, for
which, as mentioned earlier, it is enough that either the comonad \emph{or} the monad
be idempotent.

A basic example of a lax idempotent \textsc{awfs} is the one that factors a functor
$f\colon A\to B$ as a left adjoint coretract $A\to f\downarrow B$ followed by
the split opfibration $f\downarrow B\to B$. We refer to this \textsc{awfs} as the
coreflection--opfibration \textsc{awfs}.

\subsection*{Simple reflections}
\label{sec:simple-reflections}
The paper~\cite{MR779198} studies the relationship between orthogonal
factorisation systems, abbreviated \textsc{ofs}s, and reflections. Every \textsc{ofs}
$(\mathcal{E},\mathcal{M})$ on a category $\mathcal{C}$ induces a reflection
on $\mathcal{C}$ as long as $\mathcal{C}$ has a terminal object $1$; the reflective
subcategory is $\mathcal{M}/1$, the full subcategory of those objects
$X$ such that $X\to 1$ belongs to $\mathcal{M}$. Under certain hypotheses,
a reflection, or an idempotent monad $\mathsf{T}$ on $\mathcal{C}$, induces an \textsc{ofs}. One of the possible
hypotheses is that $\mathsf{T}$ be \emph{simple}, which means that for any morphism $f$
the dashed morphism into the pullback depicted below is inverted by $\mathsf{T}$. The
factorisation of $f$ is then given by $f={R}f\cdot{}Lf$, and the left class
of morphisms consists of those which are inverted by $\mathsf{T}$.
\begin{equation}
  \label{eq:354}
  \xymatrixrowsep{.6cm}
  \diagram
  A\ar@{-->}[dr]^{Lf}\ar@/^/[drr]\ar@/_/[ddr]_f&&\\
  &Kf\ar[r]\ar@{}[dr]|{\mathrm{p.b.}}\ar[d]_{{R}f}&
  TA\ar[d]^{Tf}\\
  &B\ar[r]&
  TB
  \enddiagram
\end{equation}

There is an alternative way of describing simple reflections which seems to be
absent from the literature. Suppose that $\mathsf{T}$ is an idempotent monad on \C\ and
denote by $\mathsf{T}\text-\mathrm{Iso}$ the category of morphisms in \C\ that
are inverted by $T$. This category fits in a pullback square
\begin{equation}
  \label{eq:32}
  \diagram
  \mathsf{T}\text-\mathrm{Iso}\ar[r]\ar[d]_U&\mathrm{Iso}\ar[d]\\
  \C^\two\ar[r]^-{T^\two}&\C^\two
  \enddiagram
\end{equation}
where both vertical functors are full inclusions.
\begin{prop}
  \label{prop:3}
  The reflection $\mathsf{T}$ is simple if and only if
  $U\colon\mathsf{T}\text-\mathrm{Iso}\hookrightarrow{} \C^\two$ is a
  coreflective subcategory.
\end{prop}

One way of expressing the construction of the \textsc{ofs} from $\mathsf{T}$ is the
following. On any category $\mathcal{A}$ we have the \textsc{ofs}
$(\mathrm{Iso},\mathrm{Mor})$, with left class the isomorphisms and right class
all morphisms. Isomorphisms are the coalgebras for the idempotent comonad
$\mathsf{L}'$ on $\mathcal{A}^\two$ given by $L'(f)=1_{\dom(f)}$.
If $F\dashv U\colon \mathcal{A}\hookrightarrow{}\mathcal{C}$ is the
adjunction induced by the reflection $\mathsf{T}$,
the copointed endofunctor $(L,\Phi)$ defined by
pullback along the unit of the adjunction
satisfies the property that the rectangle on the right hand side below is a
pullback.
In other words, $(L,\Phi)$-coalgebras are those morphisms that are
inverted by $F$; equivalently $(L,\Phi)\text-\mathrm{Coalg}\cong\mathsf{T}\text-\mathrm{Iso}$.
\begin{equation}
  \label{eq:355}
  \xymatrixrowsep{.6cm}
  \diagram
  L\ar[d]_\Phi\ar[r]&
  U^\two L'F^\two\ar[d]^{\Phi'}\\
  1\ar[r]&
  U^\two F^\two
  \enddiagram
  \qquad
  \diagram
  (L,\Phi)\text-\mathrm{Coalg}\ar[d]\ar[r]\ar@{}[dr]|{\mathrm{p.b.}}&
  (L',\Phi')\text-\mathrm{Coalg}\ar[d]\\
  \mathcal{C}^\two\ar[r]^-{F^\two}&
  \mathcal{A}^\two
  \enddiagram
\end{equation}
Any morphism that is inverted by ${T}$ is orthogonal to $Tf$ and therefore to its
pullback ${R}f$; in particular, $Lf$ satisfies this if the reflection is
simple. Therefore, we obtain an \textsc{ofs} when $\mathsf{T}$ is simple, with left class
those morphisms that are inverted by ${T}$.

% An alternative way to
% prove that we obtain an \textsc{ofs} is to show that $(L,\Phi)$ has
% an extension to an idempotent comonad. The comultiplication $\Sigma\colon
% L\Rightarrow L^2$ is the morphism that corresponds to the pair of morphisms
% $\Sigma_0\colon L\Rightarrow U^\two L'F^\two L$ and $1\colon L\to L$, where
% $\Sigma_0$ is the transpose of the transformation
% $F^\two L\Rightarrow L'F^\two L$ with component at $f$ the
% $(L',\Phi')$-coalgebra structure of $F\lambda_f$, ie
% $(1,(F\lambda_f)^{-1})\colon 1_{\dom(f)} \longrightarrow F\lambda_f$.
% The pointed endofunctor $(R,\Lambda)$ given by $f\mapsto\rho_f$ underlies a
% monad by construction.

\subsection*{Simple 2-adjunctions and \normalfont{\textsc{awfs}s}.}
\label{sec:simple-2-adjunctions}

The above analysis can be adapted to the case where categories are substituted by
2-categories and \textsc{ofs}s by lax orthogonal \textsc{awfs}s. Reflections are substituted by lax
idempotent 2-monads, idempotent (co)monads by lax
idempotent 2-(co)monads, the simple reflections by appropriately defined simple
2-adjunctions or simple 2-monads. The reflective subcategory $\mathrm{Iso}$ of
the arrow category is substituted by the lax idempotent 2-comonad whose algebras
are coretract left adjoints, while $\mathrm{Mor}$ is substituted by the free
split opfibration 2-monad. A version of the main theorem of
Section~\ref{sec:transf-along-left}, appropriately modified for this
introduction, states:
\begin{thm*}
  If the 2-adjunction $F\dashv U\colon\K\to\mathscr{A}$ is simple, and
  $\K,\mathscr A$ are 2-categories with enough finite limits, then there is a lax
  orthogonal {\normalfont{\textsl{\textsc{awfs}}}} $(\mathsf{L},\mathsf{R})$ on
  $\K$ whose $\mathsf{L}$-coalgebras are morphisms $f$ of $\K$ with a coretract
  adjunction $Ff\dashv r$ in $\mathscr A$.
\end{thm*}
The notion of simple 2-adjunction is central to the theorem, and occupies
most of Sections~\ref{sec:extens-transf-a.f} and~\ref{sec:transf-along-left}.

A \emph{simple 2-monad} is one whose associated free algebra 2-adjunction is simple. When
all the 2-categories involved are in fact categories, lax idempotent 2-monads
reduce to reflections and our concept of simple 2-monad to the one of simple
reflection. Therefore, we know that there are lax idempotent 2-monads that are
not simple, as \cite{MR779198}~gives examples of reflections that are not
simple.

\subsection*{Examples and further work}
\label{sec:exampl-furth-work}
The main example treated in the present article arises from categories with colimits.
Given a class of colimits, there exists a 2-monad on \Cat\ whose algebras are
categories with chosen colimits of that class. We show that these 2-monads are
simple, giving rise to lax orthogonal \textsc{awfs}s on \Cat. Even though the left morphisms of this
factorisation system are described in general in the theorem above, the right
class of morphisms is more difficult to pin down. We carefully investigate the right
class of morphisms and show that they do not coincide with the
obvious candidates: the opfibrations whose fibers have chosen colimits of the
given class and whose push-forward functors between fibres preserve them.
% Another example of simple 2-monad is the one given by Cauchy completion on the
% 2-category of Lawvere metric spaces --~ categories enriched in the extended
% non-negative real numbers. To give an idea about the lax orthogonal \textsc{awfs} within
% the space constraints of this introduction, one can look at maps between
% \emph{metric} spaces. Left maps between
% metric spaces are dense isometries. Right maps $f\colon A\to B$ between metric
% spaces are those distance decreasing maps with the property that each Cauchy
% sequence in $A$, such that its image under $f$ converges to a point $b\in B$,
% converges to a point of $A$ over $b$.

There are a number of examples of lax orthogonal \textsc{awfs}s on locally
preordered 2-categories, including that on the (2-)category of $T_0$ topological
spaces mentioned earlier in this introduction, that we have had to leave out of
this article for reasons of space. These will appear in a companion paper that
will concentrate in the case of locally ordered 2-categories, which is still
rich enough to encompass a large number of examples and relates to a rich
literature on the subject of injectivity in order-enriched categories.

\subsection*{Description of sections}
\label{sec:description-sections}

Sections~\ref{sec:backgr-natur-weak} and~\ref{sec:double-categ-aspects} can be regarded as a fairly self-contained
recount of the basic definitions and properties of \textsc{awfs}s.

We put together at the beginning of Section~\ref{sec:kz-a.f.s} some facts about
lax idempotent 2-(co)monads, one of our main tools, before introducing lax
orthogonal \textsc{awfs}, our main subject of study.

Section~\ref{sec:2-comonad-lax} proves that in order for an \textsc{awfs} to be lax
orthogonal it suffices that either the 2-monad or the 2-comonad be lax idempotent.

Sections~\ref{sec:lifting-operations} and~\ref{sec:univ-categ-with} recount the
notions of lifting operations and diagonal fillers, with their relationship to
\textsc{awfs}s. Our approach uses modules or profunctors and appears to
be novel.
In a 2-category one can
consider the usual lifting operations, but also lax natural
ones. We define lax natural and \textsc{kz}~diagonal fillers in
Section~\ref{sec:lax-natur-diag} and prove that lax
orthogonal \textsc{awfs}s give rise to \textsc{kz}~diagonal fillers. Lax orthogonal functorial
factorisations are briefly considered.

In Section~\ref{sec:freeness}
we characterise lax orthogonal \textsc{awfs}s as those \textsc{awfs} $(\mathsf{L},\mathsf{R})$ for
which $\mathsf{R}$-algebras are algebraically \textsc{kz} injective
to all $\mathsf{L}$-coalgebras, or equivalently, for which natural \textsc{kz} diagonal
fillers exist for squares from $\mathsf{L}$-coalgebras to $\mathsf{R}$-algebras.

% Sections~\ref{sec:extend-copo-endof} to \ref{sec:transf-along-left} are perhaps
% more technical, and give conditions that allow the coreflection--opfibration lax
% orthogonal \textsc{awfs} on a 2-category $\mathscr A$ to be transferred along a left
% 2-adjoint $\mathscr B\to\mathscr A$ to a lax idempotent \textsc{awfs} on $\mathscr
% B$.
Section~\ref{sec:extens-transf-a.f} introduces the concept of simple adjunction
of 2-functors. One of our main
results, the construction of a lax idempotent \textsc{awfs} from a simple
adjunction, can be found in Section~\ref{sec:transf-along-left}.

Section~\ref{sec:simple-2-monads} studies the case when the simple 2-adjunction is
the free algebra adjunction induced by a 2-monad, that we call a
\emph{simple} 2-monad, as it generalises the notion of simple
reflection~\cite{MR779198}. Conditions that guarantee that a
lax idempotent 2-monad is simple are provided.

Section~\ref{sec:exampl-compl-v} studies the example of --~enriched~-- categories and completion under
colimits. We show that for a class of colimits $\Phi$, the 2-monad
whose algebras are categories with \emph{chosen} colimits of that class is
simple, whence inducing a lax orthogonal \textsc{awfs} $(\mathsf{L},\mathsf{R})$. We prove
in Section~\ref{sec:monad-split-opfibr} that $\mathsf{R}$-algebras are always
split opfibrations with fibrewise chosen $\Phi$-colimits and that the converse does
not always hold. The article concludes with a short section that comments on
further work and examples.

\section{Background on algebraic weak factorisation systems}
\label{sec:backgr-natur-weak}

In the last few years there has been much interest in algebraic weak
factorisation systems (\textsc{awfs}s) mainly due to their connection to Quillen's
model categories and the small object argument, but also due to the homotopical
approach to type theory (homotopy type theory). The basic theory of
\textsc{awfs}s appeared in~\cite{MR2283020} with the name of \emph{natural weak
  factorisation system}, and was later expanded in~\cite{MR2506256}, especially
with respect to the construction of cofibrantly generated \textsc{awfs}. Further
study appeared recently in~\cite{MR3393453}. From Section~\ref{sec:kz-a.f.s}
onwards, the present paper expands the theory in another direction, that of
\textsc{awfs} in 2-categories whose lifting operations, or diagonal fillers,
satisfy a universal property with respect to 2-cells. Before all that we need to
collect present basics of the theory of \textsc{awfs}, mostly following
\cite{MR2506256,MR3393453}.

\subsection{The definition of AWFS}
\label{sec:definition-awfs}

We denote by $\two$ the category with two objects $0$ and $1$ and only one
non-identity arrow $0\to 1$, and by $\mathbf{3}$ the category with three objects
and three non-identity morphisms $0\to 1\to 2$.
Given a category \C\ consider the functors $d_0$, $d_1$, $d_2\colon
\mathcal{C}^{\mathbf{3}}\to\mathcal{C}^\two$
that send a pair of composable morphisms $(f\colon A\to B,g\colon B\to C)$ in
\C\ to:
$d_0(f,g)=f$, $d_1(f,g)=g\cdot{}f$, $d_2(f,g)=g$.

When displaying diagrams, we shall denote an object $f\in\C^\two$ by a vertical
arrow and a morphism $(h,k)\colon f\to g$ in $\C^\two$ by a commutative square,
as shown.
\begin{equation}
  \label{eq:30}
  \diagram
  \cdot\ar[d]_f\\
  \cdot
  \enddiagram
  \qquad\qquad\qquad
  \diagram
  \cdot\ar[d]_f\ar[r]^h&\cdot\ar[d]^g\\\cdot\ar[r]^k&\cdot
  \enddiagram
\end{equation}

\begin{df}
\label{df:3}
A \emph{functorial factorisation} in \C\ is a section of the composition functor
$d_1\colon\C^{\mathbf{3}}\to\C^\two$.
This means that for each morphism $(h,k)\colon f\to g$ in $\C^{\mathbf2}$ we
have a factorisation,
functorial in $(h,k)$, as depicted.
  \begin{equation}
    \label{eq:4}
    \xymatrixrowsep{.5cm}
    \diagram
    A\ar[d]_f\ar[r]^h&C\ar[d]^g\\
    B\ar[r]_k&D
    \enddiagram
    \quad\longmapsto
    \quad
    \diagram
    A\ar[d]_{Lf}\ar[r]^h&C\ar[d]^{Lg}\\
    Kf\ar[d]_{{R}f}\ar[r]^{K(h,k)}&Kg\ar[d]^{{R}g}\\
    B\ar[r]^k&D
    \enddiagram
  \end{equation}
\end{df}

A functorial factorisation as above induces a pointed endofunctor $\Lambda\colon
1\Rightarrow R$ and a
copointed endofunctor $\Phi\colon L\Rightarrow 1$ on $\C^{\mathbf2}$.
The endofunctor $L$ is given by $f\mapsto Lf$,
and the component of the copoint $\Phi$ at the object $f$ is depicted on the
left hand side of \eqref{eq:6}. Similarly, $f\mapsto{R}f$, and the component of
the point $\Lambda$ at the object $f$ is depicted
on the right hand side of~\eqref{eq:6}.
We note that the domain component of $\Phi$ and the codomain component of
$\Lambda$ are identities, which implies
$\mathrm{dom}L=\mathrm{dom}$ and $\mathrm{cod}R=\mathrm{cod}$, as
functors $\C^{\mathbf2}\to\C$. We say that $(L,\Phi)$ is \emph{domain
  preserving} and that $(R,\Lambda)$ is \emph{codomain preserving}.
\begin{equation}
  \label{eq:6}
  \xymatrixrowsep{0.5cm}
  \diagram
  A\ar[d]_{Lf}\ar@{=}[r]&A\ar[d]^f\\
  Kf\ar[r]_{{R}f}&B
  \enddiagram
  \qquad
  \diagram
  A\ar[d]_f\ar[r]^{Lf}&Kf\ar[d]^{{R}f}\\
  B\ar@{=}[r]&B
  \enddiagram
\end{equation}
Conversely, either a domain preserving copointed endofunctor $(L,\Phi)$ or a codomain
preserving pointed endofunctor $(R,\Lambda)$ on $\C^\two$ define a functorial
factorisation, in the first case by setting ${R}f=\cod(\Phi_f)$,
and in the second case by setting $Lf=\dom(\Lambda_f)$.

\begin{df}
\label{df:10}
An \emph{algebraic weak factorisation system} \cite{MR2283020,MR2506256} is a functorial
factorisation where the copointed endofunctor $\Phi\colon L\Rightarrow 1$ is
equipped with a comultiplication $\Sigma\colon L\Rightarrow L^2$, making it into a comonad $\mathsf L$, and the
pointed endofunctor $\Lambda\colon 1\Rightarrow R$ is equipped with a
multiplication $\Pi\colon R^2\Rightarrow R$, making it into a monad $\mathsf R$,
plus a distributivity condition.
The components of this comultiplication and multiplication will be denoted as
follows.
\begin{equation}
  \label{eq:33}
  \Sigma_f=
  \xymatrixrowsep{0.5cm}
  \diagram
  {A}
  \ar@{=}[r]\ar[d]_{Lf}
  &
  {A}
  \ar[d]^{{L^2f}}
  \\
  {Kf}
  \ar[r]^-{\sigma_f}
  &
  {KLf}
  \enddiagram
  \qquad
  \Pi_f=
  \diagram
  {K{R}f}
  \ar[r]^-{\pi_f}\ar[d]_{R^2f}
  &
  {Kf}
  \ar[d]^{{R}f}
  \\
  {B}
  \ar@{=}[r]
  &
  {B}
  \enddiagram
\end{equation}
Furthermore, the monad and comonad must be related by the distributivity
condition introduced in~\cite{MR2506256} that asserts that the
natural transformation $\Delta\colon LR\Rightarrow RL$ with components
\begin{equation}
  \label{eq:98}
  \Delta_f=\quad
  \diagram
  \cdot\ar[d]_{L{{R}f}}\ar[r]^{\sigma_f}\ar[dr]^1&
  \cdot\ar[d]^{{R}{Lf}}\\
  \cdot\ar[r]_{\pi_f}&
  \cdot
  \enddiagram
\end{equation}
is a distributive law, ie that the
diagrams shown below commute. In fact, the two triangles automatically
commute as a consequence of the comonad and monad axioms for $\mathsf{L}$ and
$\mathsf{R}$.
  \begin{equation}
    \label{eq:99}
    \xymatrixrowsep{.4cm}
    \diagram
    LR\ar[rr]^{\Delta}\ar[dr]_{\Phi R}&&RL\ar[dl]^{R\Phi}\\
    &R&
    \enddiagram
    \quad
    \diagram
    &L\ar[dl]_{L\Lambda}\ar[dr]^{\Lambda L}&\\
    LR\ar[rr]^\Delta&&RL
    \enddiagram
  \end{equation}
  \begin{equation}
    \label{eq:100}
    \xymatrixrowsep{.5cm}
    \diagram
    LR\ar[rr]^\Delta\ar[d]_{\Sigma R}&&RL\ar[d]^{R\Sigma}\\
    L^2R\ar[r]^-{L\Delta}&LRL\ar[r]^-{\Delta L}&RL^2
    \enddiagram
    \diagram
    LR^2\ar[d]_{L\Pi}\ar[r]^-{\Delta R}&RLR\ar[r]^-{R\Delta}&
    R^2L\ar[d]^{\Pi L}\\
    LR\ar[rr]^\Delta && RL
    \enddiagram
  \end{equation}
\end{df}
One of the ideas behind this definition is that the $\mathsf L$-coalgebras have
the left lifting property with respect to the $\mathsf R$-algebras, as explained
below.
An $\mathsf L$-coalgebra structure on a morphism $f\colon A\to B$, respectively,
an $\mathsf R$-algebra structure on $f$, is given by morphisms in
$\C^{\mathbf{2}}$ of the form
\begin{equation}
  \label{eq:7}
  \xymatrixrowsep{0.5cm}
  \diagram
  A\ar[d]_f\ar@{=}[r]&A\ar[d]^{Lf}\\
  B\ar[r]_-s&Kf
  \enddiagram
  \quad\text{and}\quad
  \diagram
  Kf\ar[d]_{{R}f}\ar[r]^-p&A\ar[d]^f\\
  B\ar@{=}[r]&B
  \enddiagram
\end{equation}
The domain and codomain components depicted by equality symbols are identity
morphisms as a consequence of the counit axiom
of the comonad $\mathsf L$, respectively
unit axiom of the monad $\mathsf
R$. These axioms also imply
${R}f\cdot{}s=1_B$ and $p\cdot{}Lf=1_A$.

Continuing, given a morphism $(h,k)$ in $\C^{\mathbf2}$ as in~\eqref{eq:4}, we
get a diagonal filler as depicted.
\begin{equation}
  \label{eq:8}
  \xymatrixcolsep{1.5cm}
  \xymatrixrowsep{.5cm}
  \diagram
  A\ar[rr]^h\ar[dd]_f\ar[dr]^{Lf}
  &&
  C\ar[d]_{Lg}\ar@{=}[r]
  &
  C\ar[dd]^g
  \\
  &
  Kf\ar[r]^-{K(h,k)}\ar[d]^{{R}f}
  &
  Kg\ar[dr]_{{R}g}\ar[ur]_{p}
  &
  \\
  B\ar[ur]^s\ar@{=}[r]
  &
  B\ar[rr]^k
  &&
  B
  \enddiagram
\end{equation}

\begin{rmk}
\label{rmk:15}
Every \textsc{awfs} $(\mathsf{L},\mathsf{R})$
has an underlying \textsc{wfs} $(\mathcal{L},\mathcal{R})$, where $\mathcal{L}$
consists of those morphisms of \C\ that admit a coalgebra structure for the copointed endofunctor $(L,\Phi)$ and
$\mathcal{R}$ consists of those morphisms that admit an algebra structure for the pointed endofunctor
$(R,\Lambda)$.
To verify this, one can observe that both $\mathcal{L}$ and $\mathcal{R}$ are
closed under retracts, and that
each morphism $f$ factors as $f=Rf\cdot Lf$, where $Lf$ admits the $(L,\Phi)$-coalgebra
structure $\Sigma_f\colon Lf\to L^2f$ and $Rf$ admits the $(R,\Lambda)$-algebra
structure $\Pi_f\colon R^2f\to Rf$.
\end{rmk}

\subsection{Orthogonal factorisations as AWFSs}
\label{sec:orth-fact-as}

We continue with some more background, in this case, the characterisation of
orthogonal factorisation systems in terms of the associated \textsc{awfs}.
Clearly, any orthogonal factorisation system $(\mathcal E,\mathcal M)$ in a
category $\mathcal C$ induces an \textsc{awfs}. This is a consequence of the uniqueness
of the factorisations. One can easily characterise the \textsc{awfs} obtained in this way.

\begin{prop}[{\cite[Thm~3.2]{MR2283020}}]
  \label{prop:1}
  The following are equivalent for an {\normalfont\textsl{\textsc{awfs}}} $(\mathsf{L},\mathsf{R})$:
  \begin{itemize}
  \item The comonad $\mathsf L$ and the monad $\mathsf R$ are idempotent.
  \item The underlying {\normalfont\textsl{\textsc{wfs}}} is an {\normalfont\textsl{\textsc{ofs}}}.
  \end{itemize}
\end{prop}

Furthermore, if $\mathsf{R}$ is idempotent, then so is $\mathsf{L}$, a proof of
which can be found in~\cite{MR3393453}.

\subsection{Right morphisms form a fibration}
\label{sec:right-morphisms-form}

This section collects some of the material of~\cite[\S 3.4]{MR3393453} that will
be crucial later on.

A functor $P\colon \mathcal{A}\to\C^\two$ is a \emph{discrete
  pullback-fibration} if it is just like a discrete fibration except that only pullback
squares have cartesian liftings. More explicitly, for each $a\in \mathcal{A}$ and each pullback square
$(h,k)\colon f\to P(a)$ --~$f$ is the pullback of the morphism $P(a)$ along
$k$~-- there exists a unique morphism $\alpha\colon\bar a\to a$ in $\mathcal{A}$
such that $P\alpha=(h,k)$.
\begin{lemma}
  \label{cor:7}
  Suppose the category \C\ has pullbacks.
  For any codomain preserving monad $\mathsf{R}$ on $\C^\two$, the codomain functor exhibits
  $\mathsf{R}\text-\mathrm{Alg}$ as a discrete pullback-fibration over
  $\mathcal{C}$.
\end{lemma}
To give an idea of the proof, suppose that $g\colon C\to D$ has an $\mathsf{R}$-algebra
structure $p_g\colon Kg\to C$, and that
\begin{equation}
  \label{eq:67}
  \xymatrixrowsep{.3cm}
  \diagram
  A\ar[d]_f\ar[r]^h&
  C\ar[d]^g\\
  B\ar[r]^k&
  D
  \enddiagram
\end{equation}
is a pullback square. Then, the $\mathsf{R}$-algebra structure on $f$ is given
by the morphism $p_f\colon Kf\to A$ induced by the universal property of
pullbacks and the equality displayed below. This is the unique algebra structure
that makes $(h,k)$ a morphism of algebras.
\begin{equation}
  \label{eq:15}
  \xymatrixrowsep{.5cm}
  \xymatrixcolsep{1.3cm}
  \diagram
  Kf\ar[r]^-{p_f}\ar[d]_{Rf}&
  A\ar[r]^-{h}\ar[d]^f&
  C\ar[d]^g\\
  B\ar@{=}[r]&
  B\ar[r]^-k&
  D
  \enddiagram
  \quad=\quad
  \diagram
  Kf\ar[r]^-{K(h,k)}\ar[d]_{Rf}&
  Kg\ar[r]^-{p_g}\ar[d]_{Rg}&
  C\ar[d]^g\\
  B\ar[r]^-k&
  D\ar@{=}[r]&
  D
  \enddiagram
\end{equation}

\subsection{Miscellaneous remarks}
\label{sec:miscelaneous-remarks}

Before moving to the next section and the subject of double categories, we
collect three observations that will be of use later on. We use the
adjunctions $\cod\dashv\id\dashv\dom\colon\C^\two\to\C$, the first of which has
identity counit and the second has identity unit.
\begin{rmk}
  \label{rmk:3}
  Suppose given a functorial factorisation, with associated copointed
  endofunctor $(L,\Phi)$ and pointed endofunctor $(R,\Lambda)$.
  The identity natural transformation $1_{\C}=\dom\cdot{}L\cdot{}\id$ corresponds under
  $\id\dashv\dom$ to a natural transformation $(1,L)$ with $f$-component
  equal to the morphism depicted on the right hand side below.
  \begin{equation}
    \label{eq:62}
    {\xymatrixcolsep{1.7cm}
    \diagram
    \C^{\mathsf{2}}\rtwocell^{\id\cdot{}\dom}_L{\hole\hole(1,L)}&\C^{\mathsf{2}}
    \enddiagram
    }
    \qquad
    \xymatrixrowsep{.6cm}
    \diagram
    {\dom f}
    \ar@{=}[r]\ar[d]_{1}
    &
    {\dom f}
    \ar[d]^{Lf}
    \\
    {\dom f}
    \ar[r]_-{Lf}
    &
    {Kf}
    \enddiagram
  \end{equation}
\end{rmk}
\begin{rmk}
  \label{rmk:5}
  Given a functorial factorisation in \C\ with associated copointed endofunctor
  $(L,\Phi)$ and pointed endofunctor $(R,\Lambda)$, denote by
  $V\colon(R,\Lambda)\text-\mathrm{Alg}\to\C^{\mathbf{2}}$ the corresponding
  forgetful functor from the category of algebras for the pointed endofunctor $(R,\Lambda)$. Define a
  natural transformation as the composition of two transformations, as displayed.
  \begin{equation}
    \label{eq:65}
    \xymatrixrowsep{.6cm}
    \diagram
    {(R,\Lambda)\text-\mathrm{Alg}}
    \ar[r]^-{V}\ar[d]_{V}\drtwocell<\omit>{^{(1,p)}\hole\hole}
    &
    {\C^{\mathbf{2}}}
    \ar[d]^{\id\cdot{}\dom}
    \\
    {\C^{\mathbf{2}}}
    \ar[r]_-{L}
    &
    {\C^{\mathbf{2}}}
    \enddiagram
    \qquad
    (1,p)\colon L\cdot{}V\Longrightarrow{}\id\cdot{}\dom\cdot{}R\cdot{}V\Longrightarrow \id\cdot{}\dom\cdot{}V
  \end{equation}
  The first arrow is the mate of the identity natural transformation $\cod\cdot{} L=
  \dom \cdot{}R$ under the adjunction $\cod\dashv\id$. The second arrow is the
  application of the $(R,\Lambda)$-algebra structure of $R\cdot{}V\Rightarrow V$.
  Explicitly, the component of $(1,p)$ on an $(R,\Lambda)$-algebra $(f,p_f)$ is
  \begin{equation}
    \label{eq:25}
    \xymatrixrowsep{.6cm}
    \diagram
    \dom f
    \ar@{=}[r]\ar[d]_{Lf}
    &
    \dom f
    \ar[d]^{1}
    \\
    {Kf}
    \ar[r]_-{p_f}
    &
    \dom f
    \enddiagram
  \end{equation}
\end{rmk}
\begin{rmk}
  \label{rmk:6}
  The pasting along $L$ of the transformation $(1,L)$ of Remark~\ref{rmk:3} with the
  transformation $(1,p)$ of Remark~\ref{rmk:5} is the
  identity. This is a consequence of the unit axiom for $(R,\Lambda)$-algebras:
  if $(f,p_f)$ is an algebra, then $p_f\cdot{}Lf=1$.
\end{rmk}

\section{Double categories of algebras and coalgebras}
\label{sec:double-categ-aspects}
This section collects remarks on double categories and \textsc{awfs}s, due to
R~Garner. The definition of \textsc{awfs}s used in~\cite{MR2506256} differs of
the original one~\cite{MR2283020} in the requirement of an extra distributivity
condition: the transformation $\Delta\colon LR\Rightarrow RL$ displayed
in~\eqref{eq:98} should be a mixed distributive law. This condition is what
makes possible the definition of a composition of $\mathsf{R}$-algebras and of
$\mathsf{L}$-coalgebras, as we proceed to explain.

The standard cocategory object in $\Cat$ displayed on the
left below induces a category object in \Cat, that is, a double category,
displayed in the centre, that we may call the double category of squares and
denote by $\Sq(\mathcal{C})$.
Objects of $\Sq(\C)$ are those of \C, vertical and horizontal morphisms are
morphisms of \C, while 2-cells in $\Sq(\C)$ are commutative squares in
\C.
\begin{equation}
  \label{eq:22}
  \diagram
  \mathbf{3}&
  \mathbf{2}\ar@<8pt>[l]_{\phantom{c}}\ar[l]\ar@<-8pt>[l]\ar[r]&
  \mathbf{1}\ar@<8pt>[l]\ar@<-8pt>[l]
  \enddiagram
  \qquad
  \diagram
  \mathcal{C}^{\mathbf{3}}&
  \mathcal{C}^{\mathbf{2}}
  \ar@{<-}@<10pt>[l]\ar@{<-}[l]\ar@{<-}@<-10pt>[l]\ar@{<-}[r]^{\id}&
  \mathcal{C}\ar@{<-}@<10pt>[l]_{\dom}\ar@{<-}@<-10pt>[l]_{\cod}^{\phantom{c}}
  \enddiagram
  \qquad
  \diagram
  \mathsf{R}\text-\mathrm{Alg}\ar[d]_V &
  \mathcal{C}\ar@{<-}@<10pt>[l]\ar@{<-}@<-10pt>[l]\ar@{=}[d]\ar[l]
  \\
  \mathcal{C}^{\mathbf{2}}
  \ar@{<-}[r]^{\id}&
  \mathcal{C}\ar@{<-}@<10pt>[l]_{\dom}\ar@{<-}@<-10pt>[l]_{\cod}
  \enddiagram
\end{equation}

The central result of this section is the following.
\begin{prop}
  \label{prop:2}
  If $\mathsf{R}$ is a codomain-preserving monad on \C, there is a bijection
  between {\normalfont\textsl{\textsc{awfs}s}} with monad $\mathsf{R}$ and
  extensions of the diagram on the right hand side of~\eqref{eq:22} to a double
  functor, by which we mean extensions of the reflexive graph
  $\mathsf{R}\text-\mathrm{Alg}\rightrightarrows{}\C$ to a category object that
  makes~\eqref{eq:22} into a functor internal to \Cat\ --~a double functor into
  $\Sq(\C)$.
\end{prop}
Below we give an indication of the proof of this proposition; a more detailed account can
be found in~\cite[\S 3]{MR3393453}.

\subsection{From AWFSs to double categories}
\label{sec:from-norm-double}
If $(\mathsf{L},\mathsf{R})$ is an \textsc{awfs} on \C, $\mathsf{R}$-algebras can be
composed, in the sense that if $f\colon A\to B$ and $g\colon B\to C$ are
$\mathsf{R}$-algebras, then an $\mathsf{R}$-algebra structure for $g\cdot{}f$ can be
constructed from the \textsc{awfs}. Explicitly, if $(p_f,1_B)\colon Rf\to f$ and
$(p_g,1_C)\colon Rg\to g$ are the $\mathsf{R}$-algebra structures of $f$ and $g$,
with respective domain components $p_f\colon Kf\to A$ and $p_g\colon Kg\to B$,
then the $\mathsf{R}$-algebra structure $(p_{g\cdot f},1_C)\colon R(g\cdot f)\to
g\cdot f$ is constructed in the following manner. The reader will recall from
Section~\ref{sec:backgr-natur-weak}, especially from diagram~\eqref{eq:8}, that
any morphism $f\to g$ from an $\mathsf{L}$-coalgebra to an $\mathsf{R}$-algebra
has a canonical diagonal filler. Consider the canonical diagonal filler $a$ of the
square $(f,R({g\cdot f}))\colon L({g\cdot f})\to g$, and then the
canonical diagonal filler $p_{g\cdot f}$ of the square $(1_A,a)\colon
L({g\cdot f})\to f$, as depicted in the diagram. It can be shown without
much problem that this is an $\mathsf{R}$-algebra structure on $g\cdot f$.
\begin{equation}
  \label{eq:29}
  \xymatrixcolsep{1.8cm}
  \xymatrixrowsep{.5cm}
  \diagram
  A\ar[dd]_{L({g\cdot f})}\ar@{=}[r]
  &
  A\ar[d]^f\\
  &
  B\ar[d]^g\\
  K(g\cdot f)\ar[r]_-{R({g\cdot f})}
  \ar[ur]_a\ar[uur]^{p_{g\cdot f}}&
  C
  \enddiagram
\end{equation}
We write $(g,p_g)\bullet(f,p_f)$ for the $\mathsf{R}$-algebra $(g\cdot
f,p_{g\cdot f})$ described. This operation on pairs of composable
$\mathsf{R}$-algebras can be shown to be associative and has identities
$(1_A,{R}{1_A})$, so there is a double category
$\mathsf{R}\text-\mathbb{A}\mathrm{lg}$.
The forgetful functor from $\mathsf{R}$-algebras
forms part of a double functor, depicted on the right hand of~\eqref{eq:22}.

\subsection{From double categories to AWFSs}
\label{sec:from-double-categ}

Suppose that $\mathsf{R}$ is a codomain-pre\-serv\-ing
monad on $\C^\two$, with associated codomain preserving copointed endofunctor
$(L,\Phi)$ on $\C^\two$. Each double category structure on
$\mathsf{R}\text-\mathrm{Alg}\rightrightarrows\C$ that is compatible with the
composition of morphisms in \C\ induces a comultiplication $\Sigma\colon
L\Rightarrow L^2$ that makes $\mathsf{L}=(L,\Phi,\Sigma)$ into a comonad
and $(\mathsf{L},\mathsf{R})$ an \textsc{awfs}. In the following we explain how
to construct the comultiplication from the double category structure.

If $(f,p_f)$ and $(g,p_g)$ are $\mathsf{R}$-algebras with $\cod(f)=\dom(g)$, the
double category structure provides for a vertical composition
$(g,p_g)\bullet(f,p_f)=(g\cdot{}f,p_g\bullet p_f)$ with underlying morphism $g\cdot{}f$. The
identities for the vertical composition are the $\mathsf{R}$-algebras
$(1,{R}1)$. Morphisms of $\mathsf{R}$-algebras can be vertically composed too:
given such morphisms $(h,k)\colon f\to g$ and $(k,\ell)\colon f'\to g'$,
then $(h,\ell)$ is a morphism
$f'\bullet f\to g'\bullet g$.

The comultiplication
$\Sigma_f=(1,\sigma_f)\colon Lf\to L^2f$ can be constructed from
the double category structure in the following manner. Consider the morphism of
$\mathsf{R}$-algebras $(\sigma_f,1)\colon Rf\to Rf\bullet RLf$ that
corresponds under free $\mathsf{R}$-algebra adjunction to the morphism
$(L^2f,1)\colon f\to Rf\cdot{}RLf$ in $\C^\two$.
\begin{equation}
  \label{eq:14}
  \xymatrixrowsep{.5cm}
  \diagram
  Kf\ar[r]^-{\sigma_f}\ar[dd]_{{R}f}&KLf\ar[d]^{{R}{Lf}}\\
  &Kf\ar[d]^{{R}f}\\
  B\ar@{=}[r]&B
  \enddiagram
  \qquad
  \sigma_f\cdot{}Lf=L^2{f}
\end{equation}
The detail of the proof that the components $\sigma_f$ yield a comultiplication
for $L$ can be found in~\cite[Prop.~4]{MR3393453}.

\section{Lax orthogonal AWFSs}
\label{sec:kz-a.f.s}
This section introduces the fundamental definition of this work, lax orthogonal
\textsc{awfs}s, and describes the most basic 2-categorical example. Before all that, we
shall recall some facts about lax idempotent 2-monads.

\subsection{2-monads}
\label{sec:2-monads}
We shall assume throughout the paper that the reader is familiar with the basic
notions of 2-category, 2-functor, 2-natural transformation and
modification. Familiarity with 2-(co)monads shall also be assumed, but we can
take this opportunity to remind the reader of the definitions; a complete account
can be found in~\cite{BKP}. A 2-monad
$\mathsf{T}=(T,i,m)$ on a 2-category \K\ is a 2-functor $T\colon \K\to \K$ with
2-natural transformations $i\colon 1_{\K}\Rightarrow T$, called the unit, and
$m\colon T^2\Rightarrow T$, called the multiplication, that satisfy the usual axioms of
a monad; in other words, the underlying functor of $T$ with the underlying
natural transformations of $i$ and $m$ form an ordinary monad on the underlying
category of \K. The definition of 2-comonad is dual.

An algebra for the 2-monad $\mathsf{T}$ is, by definition, an algebra for its
underlying monad. This amounts to an object $A$ with a morphism $a\colon TA\to
A$ that satisfies the usual algebra axioms --~the 2-cells play no role here. We
shall usually be concerned with the so-called \emph{strict morphisms} of
$\mathsf{T}$-algebras, which are the morphisms of algebras for the underlying
monad of $\mathsf{T}$; ie a strict morphism $(A,a)\to(B,b)$ is a morphism
$f\colon A\to B$ in \K\ such that $b\cdot Tf=f\cdot a$. However, the
2-dimensional aspect of \K\ enable us to speak of \emph{lax morphisms}, which
are morphisms $f\colon A\to B$ equipped with a 2-cell $\bar f\colon b\cdot
Tf\Rightarrow f\cdot a$ that must satisfy certain coherence axioms. For example,
(lax) monoidal functors are examples of lax morphisms for a certain
2-monad. There is a dual notion of \emph{oplax morphism}, which is a morphism
$f\colon A\to B$ with a 2-cell $\bar f\colon f\cdot a\Rightarrow b\cdot Tf$ that
must satisfy coherence axioms. A lax morphism $(f,\bar f)$ whose 2-cell $\bar f$ is
invertible is said to be a \emph{pseudomorphism}.

The four types of morphisms described in the previous paragraph are the
morphisms of four 2-categories, all with the $\mathsf{T}$-algebras as objects:
$\mathsf{T}\text-\mathrm{Alg}_s$ has the strict morphisms as morphisms;
$\mathsf{T}\text-\mathrm{Alg}_\ell$ has the lax morphisms as morphisms;
$\mathsf{T}\text-\mathrm{Alg}_c$ has the oplax morphisms as morphisms; and
$\mathsf{T}\text-\mathrm{Alg}$ has the pseudomorphisms as morphisms.

A useful fact about adjunctions and (op)lax morphisms is the so-called
\emph{doctrinal adjunction} theorem, of which we state the version that we will use
later.
\begin{prop}
  Let $\mathsf{T}$ be a 2-monad on \K. An oplax morphism $(f,\bar f)\colon
  (A,a)\to (B,b)$ between $\mathsf{T}$-algebras has a left adjoint in the
  2-category $\mathsf{T}\text-\mathrm{Alg}_c$ if and only if $f$ has a left
  adjoint in $\K$ and $\bar f$ is invertible.\label{prop:4}
\end{prop}
\subsection{Lax idempotent 2-monads}
\label{sec:kz-2-monads}
An essential part of our definition of lax orthogonal \textsc{awfs} is the
concept of a lax idempotent 2-monad, or \textsc{kz} 2-monad, that we recount in
this section.
We begin by introducing some space-saving terminology. Suppose given
an adjunction $f\dashv g $ in a 2-category, with unit $\eta\colon 1\Rightarrow
g\cdot{}f$ and counit $\varepsilon\colon f\cdot{}g\Rightarrow 1$. We say that $f\dashv g$ is
a \emph{retract (coretract)} adjunction when the counit (unit) is an identity 2-cell.

\begin{df}
  \label{df:12}
A 2-monad $\mathsf{T}=(T,i,m)$ on a 2-category $\mathscr K$ is \emph{lax
  idempotent}, or \emph{Kock-Z\"oberlein}, or simply \textsc{kz}, if any of the
following equivalent conditions hold.
\begin{enumerate}[label=(\roman*)]
\item \label{item:6} $Ti\dashv m$ with identity unit (coretract adjunction).
\item \label{item:8} $m\dashv iT$ with identity counit (retract adjunction).
\item \label{item:9} Each $\mathsf{T}$-algebra structure $a\colon TA\to A$ on an object $A$ is part of
  an adjunction $a\dashv i_A$ with identity counit (retract adjunction).
\item \label{item:10} There is a modification $\delta\colon Ti\Rightarrow iT$ satisfying
  $\delta \cdot i=1$ and $m\cdot \delta=1$.
\item The forgetful 2-functor $U_\ell\colon \mathsf{T}\text-\mathrm{Alg}_\ell\to \mathscr K$ is fully
  faithful.
\item For any pair of $\mathsf{T}$-algebras $A$, $B$, every morphism $f\colon UA\to UB$
  in $\mathscr K$ admits a unique structure of a lax morphism of $\mathsf{T}$-algebras.
\item For any morphism $f\colon X\to A$ into a $\mathsf{T}$-algebra $(A,a)$, the
  identity 2-cell exhibits $a\cdot Tf$ as a left extension of $f$ along $i_X$.
\end{enumerate}
The conditions \ref{item:6}, \ref{item:8} and \ref{item:10} appeared
in~\cite{Kock:KZmonads} and~\cite{MR0424896}, albeit in a slightly different
context; \cite{MR0424896}~shows the equivalence of these three conditions. The
proof of the equivalence of the whole list, in the case of a 2-monad, can be
found in~\cite{Kelly:Prop-like}.
\end{df}

Being lax idempotent can be regarded as a property of the 2-monad, since, for
example, there can exist at most one counit for the adjunction in~\ref{item:6}.

It may be useful to say a few words about how to obtain a left extension from
the modification $\delta$. If $f\colon X\to A$ and $g\colon TX\to A$ are
morphisms into a $\mathsf{T}$-algebra $(A,a)$, and $\alpha\colon f\Rightarrow
g\cdot i_X$ a 2-cell, then the corresponding 2-cell $a\cdot Tf\Rightarrow g$ is
constructed as $(a\cdot Tg\cdot\delta_X)\cdot(a\cdot T\alpha)$.

\begin{df}
\label{df:16}
A 2-comonad $\mathsf{G}=(G,e,d)$ on $\mathscr K$ is \emph{lax idempotent}, or \textsl{\textsc{kz}}, if
the 2-monad $(G^{\mathrm{op}},e^{\mathrm{op}},d^{\mathrm{op}})$ on $\mathscr
K^{\mathrm{op}}$ is lax idempotent. This means that we have conditions dual to
the ones spelled out above for 2-monads; eg adjunctions
$eG\dashv d\dashv Ge$, a modification $\delta\colon Ge\Rightarrow eG$, etc.
We state one of the conditions in full: given a morphism $f\colon A\to X$ from a
$\mathsf{G}$-coalgebra $(A,s)$, the identity 2-cell exhibits $Gf\cdot s\colon
A\to GX$ as a left lifting of $f$ through $e_X$.
\end{df}
% \begin{rmk}
%   \label{rmk:14}
%   Given a lax idempotent 2-monad $\mathsf{T}$ on \K, the right adjoint of its Kleisli construction
%   $U_{\mathsf{T}}\colon\mathrm{Kl}(\mathsf{T})\to\K$
%   is locally fully faithful. This is easily verified, since $U_{\mathsf{T}}$ is
%   the composition of the full and faithful comparison 2-functor
%   $\mathrm{Kl}(\mathsf{T})\to\mathsf{T}\text-\mathrm{Alg}_s$ with the forgetful
%   2-functor from the 2-category of $\mathsf{T}$-algebras.
% \end{rmk}

\subsection{Definition and basic properties of lax orthogonal AWFSs}
\label{sec:definition}
\label{sec:lax-orthogonal-weak}
A lax orthogonal \textsc{awfs} will be, first of all, an \textsc{awfs} on a
2-category. We shall start, thus, with the definition of 2-functorial
factorisations and \textsc{awfs}s on 2-categories.
\begin{df}
  \label{df:17}
  A \emph{2-functorial factorisation} on a 2-category \K\ is a 2-functor that is
  a section of the 2-functor $\K^{\mathbf{3}}\to\K^\two$ that sends a pair of
  composable morphisms to its composition.
\end{df}
A 2-functorial factorisation on \K\ induces a functorial factorisation $f\mapsto
Rf\cdot Lf$ on the underlying ordinary category of \K, and in addition it
factorises 2-cells, as depicted.
\begin{equation}
  \label{eq:119}
  \diagram
  X\ar[d]_f\rtwocell^h_k{\alpha}&Y\ar[d]^g\\
  X'\rtwocell^{h'}_{k'}{\alpha'}&Y'
  \enddiagram
  \longmapsto
  \xymatrixcolsep{2.6cm}
  \diagram
  X\ar[d]_{Lf}\rtwocell^h_k{\alpha}&Y\ar[d]^{Lg}\\
  Kf\ar[d]_{Rf}\rtwocell^{K(h,k)}_{K(h',k')}{\hole\hole\hole\hole K(\alpha,\alpha')}&
  Kg\ar[d]^{Rg}\\
  X'\rtwocell^{h'}_{k'}{\alpha'}&Y'
  \enddiagram
\end{equation}

There is a bijection between the family of 2-functorial factorisations and the
family of copointed endo-2-functors $\Phi\colon L\Rightarrow 1$ of $\K^\two$
with $\dom\Phi_f=1_{\dom(f)}$; and also the family of pointed endo-2-functors
$\Lambda\colon 1\Rightarrow R$ with $\cod\Lambda_f=1_{\cod(f)}$, for all $f\in
\K^\two$.
\begin{df}
  \label{df:2}
  An \textsc{awfs} on a 2-category $\mathscr K$ consists of a pair
  $(\mathsf{L},\mathsf{R})$ formed by a 2-comonad and a 2-monad on $\K^\two$
  satisfying the same properties as \textsc{awfs}s on categories. More
  explicitly,
  \begin{itemize}
  \item the domain of the counit $\Phi\colon L\Rightarrow 1$ is an identity
    morphism;
  \item the codomain of the unit $\Lambda\colon 1\Rightarrow R$ is an identity
    morphism;
  \item Both $(L,\Phi)$ and $(R,\Lambda)$ must give rise to the same
    2-functorial factorisation on $\K$;
  \item the 2-natural transformation $\Delta$ of~\eqref{eq:98} must be a
    distributive law between the underlying comonad of $\mathsf{L}$ and the
    underlying monad of $\mathsf{R}$ on the ordinary underlying category of \K.
  \end{itemize}
% the domain of the counit $\Phi\colon L\Rightarrow 1$
%   2-comonad must be domain-preserving and the 2-comonad codomain-preserving, the
%   copointed endofunctor of $\mathsf{L}$ and the pointed endofunctor of
%   $\mathsf{R}$ must give rise to the same functorial factorisation of morphisms
%   in $\K$, and the 2-natural transformation $\Delta$ of~\eqref{eq:98} must be a distributive law.
\end{df}

\begin{df}
  \label{df:1}
  An \textsc{awfs} $(\mathsf L,\mathsf R)$ in a 2-category $\mathscr K$ is
  \emph{lax orthogonal} if the 2-comonad $\mathsf L$ and the 2-monad $\mathsf R$ are lax idempotent.
\end{df}

We will later see in Section~\ref{sec:2-comonad-lax} that it is enough to
require that either $\mathsf{L}$ \emph{or}
$\mathsf{R}$ be lax idempotent.

\begin{rmk}
  \label{rmk:7}
  It was observed in Remark \ref{rmk:6} that the transformation
  $(1,p)\colon L\cdot{}V\to\id\cdot{}\dom\cdot{}V$ of Remark \ref{rmk:5} has as right inverse
  $(1,L)\cdot{}V$, where $(1,L)$ is the transformation of Remark
  \ref{rmk:3}. We claim that, when the 2-monad $\mathsf{R}$ is lax idempotent, we
  also have a retract adjunction in the 2-category
  $\mathbf{2}\text-\Cat(\mathsf{R}\text-\mathrm{Alg}_s,\mathscr K^{\mathbf{2}})$
  of 2-functors, 2-natural transformations and modifications
  \begin{equation}
    \label{eq:48}
    (1,p)\dashv (1,L)\cdot{}V\colon \id\cdot{}\dom\cdot{}V\Longrightarrow{} L\cdot{}V.
  \end{equation}
  The counit of this adjunction is the identity modification, and the unit has
  components
  \begin{equation}
    \label{eq:129}
    \xymatrixrowsep{.6cm}
    \diagram
    A\ar[d]_{Lf}\ar@{=}[r]&
    A\ar[d]_1\ar@{=}[r]&
    A\ar[d]^{Lf}\\
    Kf\ar[r]^{p_f}\rrlowertwocell_1{^\eta_f\hole}&
    A\ar[r]^{Lf}&
    Kf
    \enddiagram
  \end{equation}
  for $f\in \mathsf{R}\text-\mathrm{Alg}_s$, where $\eta_f$ is the domain
  component of the unit of the adjunction $(p_f,1)\dashv(Lf,1)$ provided by the fact
  that $\mathsf{R}$ is lax idempotent --~numeral~\ref{item:9} of Definition~\ref{df:12}. The fact that this defines a modification
  with components $(1,\eta_f)$
  follows, and clearly satisfies the triangle identities.
\end{rmk}

\subsection{A basic example}
\label{sec:basic-example}
There is a lax orthogonal \textsc{awfs} that will play the role analogous to
the role that the \textsc{ofs} $(\mathrm{Iso},\mathrm{Mor})$ plays in the
context of simple reflections --~as explained in
Section~\ref{sec:introduction}. The next few pages give a complete description
of this basic example of a lax orthogonal \textsc{awfs}.

Every functor $f\colon A\to B$ factors as
$Lf\colon A\to Kf=(f\downarrow B)$ followed by ${R}f\colon Kf\to B$,
where $Lf(a)=(a,1\colon f(a)\to f(a),f(a))$, and ${R}f(a,\beta\colon
f(a)\to b,b)=b$. The associated pointed endofunctor $R$ on $\Cat^{\mathbf2}$
given by $f\mapsto{R}f$
underlies the free split opfibration monad $\mathsf R$. Precisely the same
factorisation can be constructed in any 2-category \K\ with the necessary comma
objects. At this point one could
deduce that there is an \textsc{awfs} $(\mathsf{L},\mathsf{R})$ by observing that split
opfibrations compose and the results cited in
Section~\ref{sec:double-categ-aspects}, and furthermore, one could use
the results of Section~\ref{sec:2-comonad-lax} to prove that the \textsc{awfs} is lax
orthogonal. Instead, we shall give an explicit description of
the comonad and its coalgebras, as they will become important in later sections.

Given a 2-category \K\ we can perform two constructions to obtain new
2-categories. The first is the 2-category $\mathbf{Lari}(\K)$, whose objects are
morphisms $f$ in \K\ equipped with a right adjoint coretract $r_f$, ie a left adjoint
right inverse or \textsc{lari}, in the terminology used
in~\cite{MR3393453}; we may write an object of this 2-category as $(f,r)$, omitting the unit and
counit of the adjunction, since the unit is an identity 2-cell and the counit is,
therefore, the unique 2-cell that satisfies the adjunction triangle axioms for
$f\dashv r$. A morphism $(f,r)\to(f',r')$ in $\mathbf{Lari}(\K)$ is a morphism
$(h,k)\colon f\to f'$ in $\K^\two$ such that $r'\cdot k=h\cdot r$. It is not
difficult to show that $(h,k)$ is automatically compatible with the counits of
the adjunctions: if the counits are $\varepsilon$ and $\varepsilon'$, then
$\varepsilon'\cdot k=k\cdot \varepsilon$. The 2-cells between morphisms in
$\mathbf{Lari}(\K)$ are just the 2-cells in between them in $\K^\two$.
There is a forgetful 2-functor
$\mathbf{Lari}(\K)\to\K^\two$, and \textsc{lari}s can be composed via the usual
composition of adjunctions, so if $(f,r)$ and $(f',r')$ are \textsc{lari}s with
$f$ and $f'$ composable morphisms, then $(f'\cdot f,r\cdot r')$ is canonically a
\textsc{lari}. In this way, $\mathbf{Lari}(\K)$ has a double category structure,
and furthermore, the composition is obviously compatible with 2-cells, so the
double category structure extends to an internal category in the category of
2-categories.

The second construction is a 2-category
$\mathbf{OpFib}(\K)$ of split opfibrations in $\K$, by which we mean morphisms
$f$ in \K\ such that each functor $\K(-,f)$ is a split opfibration in
$[\K^{\mathrm{op}},\Cat]$. There is a forgetful 2-functor
$\mathbf{OpFib}(\K)\to\K^\two$ and the composition of two split opfibrations
is canonically a split opfibration, so $\mathbf{OpFib}(\K)$ is an internal
category in the category of 2-categories.

If the 2-category \K\ has enough comma objects, then there is an \textsc{awfs}
$(\mathsf{L},\mathsf{R})$ on \K\ that satisfies
$\mathbf{OpFib}(\K)\cong\mathsf{R}\text-\mathrm{Alg}_s$, and, as we will see in
Proposition~\ref{prop:19},
$\mathbf{Lari}(\K)\cong\mathsf{L}\text-\mathrm{Coalg}_s$.
Let us first say a few words about $\mathsf{R}$. The free split
  opfibration on $f$ is given by a comma object as depicted on the left hand
  side of~\eqref{eq:20}. The unit of $\mathsf{R}$ has components $\Lambda_f=(Lf,1)$, where
  $Lf\colon A\to Kf$  is the unique morphism such that
  ${R}f\cdot{L}f=f$, $q_f\cdot{L}f=1$ and
  $\nu_f\cdot{L}f=1$. The multiplication $\Pi_f=(\pi_f,1)$ is given by the
  unique morphism $\pi_f\colon K{R}f\to Kf$ satisfying the three equalities
  displayed on the right hand side.
\begin{equation}
  \label{eq:20}
  \xymatrixrowsep{.6cm}
  \diagram
  Kf\ar[d]_{{R}f}\ar[r]^{q_f}\drtwocell<\omit>{\hole\hole\nu_f}&
  A\ar[d]^f\\
  B\ar@{=}[r]&B
  \enddiagram
  \qquad
  \begin{array}{rcl}
  q_f\cdot\pi_f&=&q_f\cdot q_{{R}f}\\
  {R}f\cdot\pi_f&=&R^2f\\
  \nu_f\cdot\pi_f&=&\nu_{{R}f}(\nu_f\cdot q_{{R}f})
  \end{array}
\end{equation}

We remark that $Lf$
comes equipped with an adjunction $Lf\dashv q_f$ with identity
unit, where $q_f\colon Kf\to A$ is the projection. The counit $\omega_{Lf}\colon
  Lf\cdot q_f\Longrightarrow 1$ is the 2-cell induced by the universal property
  of comma objects and the conditions
  \begin{equation}
    q_f\cdot\omega_{Lf}=1\colon q_f\Longrightarrow q_f
    \quad \text{and} \quad
    {R}f\cdot\omega_{Lf}=\nu_f\colon {R}f\cdot{L}f\cdot q_f=
    f\cdot q_f\Longrightarrow {R}f.
    \label{eq:91}
  \end{equation}

The copointed endo-2-functor $L$ underlies a 2-comonad with comultiplication
$\Sigma\colon L\Rightarrow L^2$, defined by the following equality and the
universal property of comma objects.
\begin{equation}
  \label{eq:31}
  \xymatrixrowsep{.6cm}
  \diagram
  Kf\ar[r]^{\sigma_f}&
  KLf\ar[r]^-{q_{Lf}}\ar[d]_{{R}{Lf}}\drtwocell<\omit>{\hole\hole\nu_{Lf}}&
  A\ar[d]^{Lf}\\
  &Kf\ar@{=}[r]&Kf
  \enddiagram
  =
  \diagram
  Kf\ar@{=}[d]\ar[r]^{q_f}\drtwocell<\omit>{\hole\hole\omega_{Lf}}&
  A\ar[d]^{Lf}\\
  Kf\ar@{=}[r]&
  Kf
  \enddiagram
\end{equation}
The 2-monad $\mathsf R$ is well-known to be lax idempotent. To see that the
comonad $\mathsf L$ is lax idempotent, one can exhibit an adjunction
$\Phi_{Lf}\dashv \Sigma_f$ with identity counit.
The existence of an adjunction ${R}{Lf}\dashv\sigma_f$, with
identity counit follows from Remark~\ref{rmk:1} below. The fact that this
adjunction yields an adjunction $\Phi_{Lf}\dashv\Sigma_f$ in $\K^\two$ can be readily
checked.
We leave the verification of the
distributivity law between the 2-comonad and 2-monad to the reader.

\begin{rmk}
  \label{rmk:1}
  Given a comma object as exhibited on the left below, each adjunction $\ell\dashv r$
  induces a retract adjunction $p\dashv s$, where
  $s$ is defined by the equality on the right hand side.
  \begin{equation}
    \label{eq:35}
    \xymatrixrowsep{.6cm}
    \diagram
    \ell\downarrow t \ar[r]^-q\ar[d]_p\drtwocell<\omit>{\nu}&A\ar[d]^\ell\\
    X\ar[r]_t&B
    \enddiagram
    \qquad\quad
    \diagram
    X\ar[r]^-s& \ell\downarrow t \ar[r]^-q\ar[d]_p\drtwocell<\omit>{\nu}&A\ar[d]^\ell\\
    &X\ar[r]_t&B
    \enddiagram
    =
    \diagram
    X\ar[r]^t&B\ar[r]^r\drlowertwocell<0>_1{<-2>\varepsilon}&A\ar[d]^\ell\\
    &&B
    \enddiagram
  \end{equation}
  The unit $\eta\colon 1\Rightarrow s\cdot{}p$ is the unique 2-cell satisfying
  $p\cdot{}\eta=1$ and
  \begin{equation}
    \label{eq:90}
    \diagram
    \ell\downarrow t\rtwocell^1_{s\cdot{}p}{\eta}&\ell\downarrow t\ar[r]^-q&A
    \enddiagram
    \quad = \quad
    \xymatrixrowsep{.6cm}
    \diagram
    \ell\downarrow t \ar[r]^-q\ar[d]_p\drtwocell<\omit>{\nu}&
    A\ar[d]^\ell\ar@{=}[r]\rtwocell<\omit>{<2>}&A
    \\
    X\ar[r]_t&B\ar[ur]_r&
    \enddiagram
  \end{equation}

  We make a final observation that will be of use later on. Suppose that the
  unit of $\ell\dashv r$ is an identity and $h\colon Z\to
  \ell\downarrow t$ is any morphism such that $\nu\cdot{}h$ is an identity
  2-cell. Then $\eta\cdot{}h$ is an identity 2-cell.
\end{rmk}

\begin{prop}
  \label{prop:19}
  Let $(\mathsf{L},\mathsf{R})$ be the {\normalfont\textsl{\textsc{awfs}}}
  described above in this section.
  \begin{enumerate}
  \item \label{item:19} There is an isomorphism over $\K^\two$ between
    $(L,\Phi)\text-\mathrm{Coalg}$ (the 2-category of coalgebras for the
    copointed endo-2-functor $(L,\Phi)$\/) and the 2-category with
    \begin{itemize}
    \item Objects $(f,v,\xi)$ where $f\colon A\rightleftarrows B:v$ and
      $\xi\colon f\cdot{}v\Rightarrow 1$ that satisfy $v\cdot{}f=1_A$ and $\xi\cdot{}f=1$
      --~strong deformation retracts of $B$.
    \item Morphisms $(f,v,\xi)\to(f',v',\xi')$, morphisms $(h,k)\colon f\to f'$
      in $\K^\two$ such that $h\cdot{}v=v'\cdot{}k$ and $\xi'\cdot{}k=k\cdot{}\xi$.
    \item 2-cells $(h,k)\Rightarrow(\bar h,\bar k)
      \colon(f,v,\xi)\to(f',v',\xi')$, 2-cells $(\alpha,\beta)\colon
      (h,k)\Rightarrow (\bar h,\bar k)$ in $\K^\two$ such that
      $\alpha\cdot{}v=v'\cdot{}\beta$.
    \end{itemize}
  \item \label{item:20} There is an isomorphism over $\K^\two$ between
    $\mathsf{L}\text-\mathrm{Coalg}_s$ and the 2-category $\mathbf{Lari}(\K)$.
  \item \label{item:7} Cofree $\mathsf{L}$-coalgebras correspond to the
    coretract adjunctions $Lf\dashv q_f$.
  \item \label{item:25} The double category structure on
    $\mathsf{L}\text-\mathrm{Coalg}_s$ induced by this
    {\normalfont\textsl{\textsc{awfs}}} is that of $\mathbf{Lari}(\K)$, ie given by
    composition of coretract adjunctions.
\end{enumerate}
\end{prop}
\begin{proof}
The reader would recall from~\eqref{eq:20} the definition of $Kf$ as a comma object.
There is a bijection between morphisms $s\colon B\to Kf$ such that
${R}f\cdot{}s=1_B$ and morphisms $v\colon B\to A$ equipped with a 2-cell $\xi\colon
f\cdot{}v\Rightarrow 1_B$; the bijection is given by composing with the comma object
$\nu_f$, ie $v=q_f\cdot{}s$ and $\xi=\nu_f\cdot{}s$.
Under this
bijection, the condition $s\cdot{}f=Lf$, which means that $(1,s)$ is a morphism $f\to
Lf$, translates into $\xi\cdot{}f=1$.
This completes the description of $(L,\Phi)$-coalgebras.

Next we translate the condition $\sigma_f\cdot{}s=K(1,s)\cdot{}s$, that is the
coassociativity axiom for that makes an $(L,\Phi)$-coalgebra into an
$\mathsf{L}$-coalgebra.
Denote the counit of $Lf\dashv q_f$ by $\omega_f$, and recall that
$\sigma_f$ is defined by \eqref{eq:31}. The morphism
$\sigma_f\cdot{}s$ corresponds under the universal property of the comma object
$\nu_{Lf}$ to the 2-cell
% \begin{equation}
%   \label{eq:199}
%   \nu_{\lambda_f}\cdot{}\sigma_f\cdot{}s=\omega_f\cdot{}s\colon \lambda_f\cdot{}q_f\cdot{}s=\lambda_f\cdot{}v
%   \Longrightarrow s
% \end{equation}
\begin{equation}
  \label{eq:199}
  \xymatrixrowsep{.5cm}
  \diagram
  B\ar[r]^-s&
  Kf\ar[r]^-{\sigma_f}&
  KLf\ar[r]^-{q_{Lf}}\ar[d]_{{R}{Lf}}
  \drtwocell<\omit>{\hole\hole\nu_{Lf}}&
  A\ar[d]^{Lf}\\
  &&Kf\ar@{=}[r]&
  Kf
  \enddiagram
  =
  \diagram
  B\ar[r]^-{s}&Kf\ar[r]^-{q_f}\ar@{=}[d]\drtwocell<\omit>{\hole\hole\omega_{Lf}}
 &
 A\ar[d]^{Lf}\\
 &Kf\ar@{=}[r]&Kf
  \enddiagram
\end{equation}
while $K(1,s)\cdot{}s$ corresponds to the 2-cell displayed below.
% \begin{equation}
%   \label{eq:200}
%   \nu_{\lambda_f}\cdot{}K(1,s)\cdot{}s=s\cdot{}\nu_f\cdot{}s=s\cdot{}\xi\colon
%   s\cdot{}f\cdot{}v=\lambda_f\cdot{}v\Longrightarrow s.
% \end{equation}
\begin{equation}
%  \label{eq:200}
  \xymatrixrowsep{.5cm}
  \diagram
  B\ar[r]^-s&
  Kf\ar[r]^-{K(1,s)}&
  KLf\ar[r]^-{q_{Lf}}\ar[d]_{{R}{Lf}}
  \drtwocell<\omit>{\hole\hole\nu_{Lf}}&
  A\ar[d]^{Lf}\\
  &&Kf\ar@{=}[r]&
  Kf
  \enddiagram
  =
  \diagram
  B\ar[r]^-s&Kf\ar[r]^-{q_f}\ar[d]_{{R}f}\drtwocell<\omit>{\hole\nu_f}&A\ar[d]^{f}&\\
  &B\ar@{=}[r]&B\ar[r]_s&Kf
  \enddiagram
\end{equation}
\begin{equation}
  \label{eq:200}
  =
  \xymatrixrowsep{.5cm}
  \diagram
  B\ar[r]^v\ar@{=}[dr]\drtwocell<\omit>{<-1.7>\xi}&A\ar[d]^f&\\
  &B\ar[r]_s&Kf
  \enddiagram
\end{equation}
Therefore, $s$ is a coalgebra precisely when \eqref{eq:199} equals
\eqref{eq:200}. These are both 2-cells between morphisms with codomain $Kf$, and
as such they are equal if and only if their respective compositions with
the projections
${R}f$ and $q_f$ coincide. Their composition with ${R}f$ yield respectively
\begin{equation}
  \label{eq:201}
  {R}f\cdot{}\omega_{Lf}\cdot{}s=\nu_f\cdot{}s=\xi\qquad\text{and}\qquad {R}f\cdot{}s\cdot{}\xi=\xi
\end{equation}
while their composition with $q_f$ yield respectively
\begin{equation}
  \label{eq:202}
  q_f\cdot{}\omega_f\cdot{}s=1 \qquad\text{and}\qquad q_f\cdot{}s\cdot{}\xi=v\cdot{}\xi.
\end{equation}
It follows that $s$ is coassociative if and only if $v\cdot{}\xi=1$, completing
  the description of $\mathsf{L}$-coalgebras as coretract adjunctions $f\dashv v$.

We now describe the morphisms of $(L,\Phi)$-coalgebras from $(1,s)\colon f\to Lf$ to
$(1,s')\colon f'\to Lf'$. Such a morphism is a morphism $(h,k)\colon f\to f'$ in
$\K^\two$ satisfying $s'\cdot{}k=K(h,k)\cdot{}s$. Composing with the comma object $\nu_{f'}$,
this equality translates into $v'\cdot{}k=h\cdot{}v$ and $\xi'\cdot{}k=k\cdot{}\xi$. A morphism of
$\mathsf{L}$-coalgebras is just a morphism between the underlying
$(L,\Phi)$-coalgebras.

A 2-cell between morphisms $(h,k)$, $(\bar h,\bar k)\colon (f,s)\to (f',s')$ of
$(L,\Phi)$-algebras is a pair of 2-cells $\alpha\colon h\Rightarrow \bar h$ and
$\beta\colon k\Rightarrow\bar k$ satisfying
$K(\alpha,\beta)\cdot{}s=s'\cdot{}\beta$.
This is an equality of 2-cells between 1-cells with codomain $Kf'$, so it holds
  if and only if it does after composing with the projections ${R}{f'}$ and $q_{f'}$.
The
composition of this equality with ${R}{f'}$ yields $\beta=\beta$ --~no
information here~-- while its composition with $q_{f'}$ yields
$\alpha\cdot{}v=\beta\cdot{}v'$. This completes the description of
$(L,\Phi)\text-\mathrm{Coalg}$. When $(f,s)$ and $(f',s')$ are
$\mathsf{L}$-algebras, with associated coretract adjunctions $(f,v,\xi)$ and
$(f',v',\xi')$, this latter equality is void too, since its mate automatically
holds. Explicitly,
\begin{equation}
  \label{eq:1}
  \big(h\cdot{}v\xrightarrow{\alpha\cdot{}v}\bar h\cdot{}v\big)
  =
  \big(h\cdot{}v=v'\cdot{}k\xrightarrow{v'\cdot{}\beta}v'\cdot{}\bar k=\bar h\cdot{}v \big)
\end{equation}
holds if and only if it does after precomposing with $f$ and composing with the
unit $1=v\cdot{}f$ of $f\dashv v$:
\begin{equation}
  \label{eq:2}
  \alpha=\alpha\cdot{}v\cdot{}f=\big(
  h=h\cdot{}v\cdot{}f=v'\cdot{}k\cdot{}f\xrightarrow{v'\cdot{}\beta\cdot{}f}v'\cdot{}\bar k\cdot{}f = \bar h\cdot{}v\cdot{}f=\bar h \big).
\end{equation}
But this latter equality automatically holds, by $\beta\cdot{}f=f'\cdot{}\alpha$.
This shows that 2-cells in $\mathsf{L}\text-\mathrm{Coalg}_s$ are simply 2-cells in
$\K^\two$.

Finally, we prove the fourth statement of the proposition. The 2-category of
$\mathsf{L}$-coalgebras is equipped with an obvious composition: that of
coretract adjunctions. Any such composition corresponds to a
unique multiplication $\bar\Pi\colon R^2\to R$ that makes $(R,\Lambda,\bar\Pi)$ a
2-monad and satisfies the distributivity condition --~see
Section~\ref{sec:double-categ-aspects} for the details. We
  have to show that $\bar\Pi$ equals the multiplication $\Pi$ of the free
split opfibration 2-monad.

By the comments at the end of Section~\ref{sec:from-double-categ}, or rather the
dual version of those comments, $\bar\Pi_f=(\bar\pi_f,1)$ is defined by the
  property that $(1,\bar\pi_f)$ is the unique
morphism of $\mathsf{L}$-coalgebras from $L({R}f)\bullet Lf$ to $Lf$ that
composed with the counit $\Phi_f=(1,{R}f)\colon Lf\to f$ yields the morphism
$(1,R^2f)\colon LRf\cdot{}Lf\to f$ in
  $\K^\two$. %Here $\bar\pi_f$ has domain $K\rho_f$ and codomain $Kf$.
  \begin{equation}
    \label{eq:51}
    \xymatrixrowsep{.4cm}
    \diagram
    A\ar[d]_{Lf}\ar@{=}[r]&A\ar[dd]^{Lf}\\
    Kf\ar[d]_{L{{R}f}}&\\
    KLf\ar[r]^-{\bar \pi_f}&Kf
    \enddiagram
  \end{equation}
By
  the previous parts of the present proposition, to say
that $(1,\bar\pi_f)$ is a morphism of $\mathsf{L}$-coalgebras is equivalent to
saying that $1_A$ and $\bar\pi_f$ form a commutative square with the right
adjoints of $Lf$ and of $L{{R}f}\cdot Lf$.
It is worth keep in mind that composition of
$\mathsf{L}$-coalgebras that induces $\bar\pi_f$ is the usual composition of
coretract adjunctions, so the right
adjoint of ${L}f$ is $q_f$ and the right adjoint of
${L}{Rf}\cdot{L}f$ is $q_{f}\cdot q_{{R}f}$.
It follows that, to say that $(1,\bar\pi_f)$ is a morphism of
$\mathsf{L}$-coalgebras is equivalent to requiring the following equality.
\begin{equation}
q_f\cdot{}\bar\pi_f=q_f\cdot{}q_{{R}f}\label{eq:56}
\end{equation}
So far we have unravelled the definition of $\bar\pi_f$. In order to deduce that
$\bar\pi_f$ equals the multiplication $\pi_f$ of the free split
opfibration 2-monad, it suffices to verify that $(1,\pi_f)$ is too a morphism of
$\mathsf{L}$-coalgebras of the form~\eqref{eq:51} and $\Phi_f\cdot
(1,\pi_f)=(1,R^2f)$. The latter equation always holds, as it is
${R}f\cdot\pi_f=R^2f$, as remarked in the equation~\eqref{eq:20}. The
fact that $(1,\pi_f)$ is a morphism of $\mathsf{L}$-coalgebras is, by the same
argument applied to $\bar\pi_f$, the condition $q_f\cdot\pi_f=q_f\cdot
q_{{R}f}$, which holds again by~\eqref{eq:20}.  By uniqueness of $\bar\pi_f$,
we obtain $\pi_f=\bar\pi_f$ and thus the vertical composition of \textsc{lari}s
coincides with that of $\mathsf{L}$-coalgebras.
% Therefore, we can deduce that $\bar\pi_f=\pi_f$:
% \begin{multline}
%   \label{eq:208}
%   \nu_f\cdot{}\bar\pi_f=\rho_f\cdot{}\epsilon_f\cdot{}\pi_f=
%   (\rho_f\cdot{}\pi_f\cdot{}\epsilon_f)(\rho_f\cdot{}\pi_f\cdot{}\lambda_{\rho_f}\cdot{}\epsilon_{f}\cdot{}q_{\rho_f})
%   =
%   (\rho_{\rho_f}\cdot{}\epsilon_f)(\rho_{\rho_f}\cdot{}\lambda_{\rho_f}\cdot{}\epsilon_{f}\cdot{}q_{\rho_f})\\
%   =\nu_{\rho_f}(\rho_f\cdot{}\epsilon_f\cdot{}q_{\rho_f})=\nu_{\rho_f}(\nu_f\cdot{}q_{\rho_f}) =
%   \nu_f\cdot{}\pi_f
% \end{multline}
\end{proof}
\begin{rmk}
\label{rmk:18}
In general, for a copointed endofunctor $(G,\varepsilon)$ on a category
$\mathcal{C}$, and a retraction $r\colon Y\twoheadrightarrow{}X$ with section
$s$ in $\mathcal{C}$, each $(G,\varepsilon)$-coalgebra structure $\delta\colon
Y\to GY$ on $Y$ induces another on $X$. This induced coalgebra structure is
$(Gr)\cdot{}\delta\cdot{}s\colon X\to GX$.  Later, in the proof of Proposition~\ref{prop:12}, we shall need the description of this
construction in the case of the copointed endo-2-functor $(L,\Phi)$ of
Proposition~\ref{prop:19}. Let $(f,v,\xi)$ be a coalgebra and $(r_0,r_1)\colon
f\to \bar f$ a retraction on $\K^\two$ with section $(s_0,s_1)$. The induced
coalgebra structure $(\bar f,\bar v,\bar\xi)$ is given by $\bar v=r_0\cdot{}v\cdot{}s_1$ and
\begin{equation}
  \label{eq:147}
  \bar f\cdot{}\bar v=\bar f\cdot{}r_0\cdot{}v\cdot{}s_1=r_1\cdot{}f\cdot{}v\cdot{}s_1\xRightarrow{r_1\cdot{}\xi\cdot{}s_1}r_1\cdot{}s_1=1.
\end{equation}
\end{rmk}
\section{The 2-comonad is lax idempotent if the 2-monad is so}
\label{sec:2-comonad-lax}

In this section we show that, in order for an \textsc{awfs} on a 2-category to be lax
orthogonal, it suffices that \emph{either} its 2-monad \emph{or} its 2-comonad be lax
idempotent. This result can be seen as a two-dimensional generalisation of the
fact that an \textsc{awfs} on a category is orthogonal if either its monad or its comonad
is idempotent --~a fact that is explained
in~\cite{MR3393453}. However, the proof, as it is to be expected, is more
involved. Incidentally, our proof uses the double category structure on
$\mathsf{R}\text-\mathrm{Alg}_s$ mentioned in
Section~\ref{sec:double-categ-aspects}.

\begin{thm}
  \label{thm:11}
  The 2-comonad of an {\normalfont\textsl{\textsc{awfs}}} on a 2-category is lax idempotent provided the
  2-monad is lax idempotent.
\end{thm}
\begin{proof}
  Denote the \textsc{awfs} on the 2-category \K\ by $(\mathsf{L},\mathsf{R})$,
  where $\mathsf{L}$ is a 2-comonad with counit $\Phi\colon
  L\Rightarrow 1$ and comultiplication $\Sigma\colon L\Rightarrow L^2$, and
  $\mathsf{R}$ is a 2-monad with unit $\Lambda\colon 1\Rightarrow R$ and
  multiplication $\Pi\colon R^2\Rightarrow R$.
  We will verify one of the equivalent conditions that make $\mathsf{L}$ a lax
  idempotent 2-comonad --~the corresponding conditions for a 2-monad are
  mentioned in Section~\ref{sec:kz-2-monads}~-- namely, that there exist
  coretract adjunctions $\Sigma_f\dashv L\Phi_f$ whose counits form
  a modification in $f\colon A\to B$. In
  Section~\ref{sec:from-double-categ} we mentioned that $(\sigma_f,1_B)$ is a
  morphism of $\mathsf{R}$-algebras $Rf\to Rf\bullet RLf$, where the codomain is
  the vertical composition of the $\mathsf{R}$-algebras $RLf$ and $Rf$.
  Consider the morphism
  \begin{equation}
    \label{eq:186}
    RLf\xrightarrow{R\Phi_f}Rf\xrightarrow{(\sigma_f,1_B)}Rf\bullet RLf
  \end{equation}
  which is, by Section~\ref{sec:kz-2-monads}, a left extension along $\Lambda_{Lf}$
  of its composition with the unit $\Lambda_{Lf}$
  \begin{equation}
    \label{eq:187}
    (\sigma_f,1_B)\cdot{}R\Phi_f\cdot{}\Lambda_{Lf}=
    (\sigma_f,1_B)\cdot{}\Lambda_f\cdot{}\Phi_f=
    ({L}^2f,{R}f)\colon Lf\longrightarrow Rf\cdot{} RLf.
  \end{equation}

  The morphism $(1_{KLf},{R}f)\colon RLf\to Rf\cdot{}RLf$ in $\K^\two$ satisfies
  $(1_{KLf},{R}f)\cdot{}\Lambda_{{Lf}}=({L}^2f,{R}f)$ too, therefore
  the universal property of left extensions gives a unique 2-cell
  $(\sigma_f,1_B)\cdot{}R\Phi_f\Rightarrow (1_{KLf},{R}f)$ in $\K^\two$ whose
  composition with $\Lambda_{Lf}$ is the identity 2-cell. This forces the 2-cell
  to be of the form
  \begin{equation}
    \label{eq:402}
    (\varepsilon_f,1_{1_B})\colon
    (\sigma_f,1_B)\cdot{}R\Phi_f\Longrightarrow (1_{KLf},{R}f)
  \end{equation}
  for a 2-cell in \K
  \begin{equation}
    \label{eq:188}
    \varepsilon_f\colon \sigma_f\cdot{}K(1_A,{R}f)\Longrightarrow 1_{K{L}f}\colon
    K{L}f \to K{L}f
  \end{equation}
  since the codomain component of $\Lambda_f$ is an identity. This definition makes
  $(\varepsilon_f,1_{1_B})$, and hence $\varepsilon_f$, a modification in $f$, a
  fact that can be verified by using the universal property of left extensions.

  We now proceed to prove that $\varepsilon_f$ is the counit of a coretract
  adjunction $\sigma_f\dashv K(1_A,{R}f)$ in $\K$, for which we must show
  three conditions:
  \begin{equation}
    \label{eq:39}
    \varepsilon_f\cdot{}\sigma_f=1\qquad
    K(1_A,{R}f)\cdot{}\varepsilon_f=1\qquad
    \varepsilon_f\cdot{}{L}^2{f}=1.
  \end{equation}
  The first two conditions are the triangle identities of the adjunction,
  while the last one means that $\varepsilon_f$ is a 2-cell in $\K^\two$.

  Consider the morphism of
  $\mathsf{R}$-algebras
  \begin{equation}
    \label{eq:190}
    Rf\xrightarrow{(\sigma_f,1_B)}Rf\bullet
    RLf\xrightarrow{({R}f,1_B)\bullet R\Phi_f} 1_B\bullet Rf
    \xrightarrow{1_B\bullet(\sigma_f,1_B)} 1_B\bullet Rf\bullet RLf
  \end{equation}
  that can be depicted in the way of the following diagram --~of solid arrows~--
  where, as always, objects of $\K^\two$ are represented by vertical arrows and
  morphisms of $\K^\two$ by commutative squares.
  \begin{equation}
    \label{eq:57}
    \xymatrixrowsep{.5cm}
    \xymatrixcolsep{1.1cm}
    \diagram
    Kf\ar[r]^-{\sigma_f}\ar[ddd]_{{R}f}&
    K{L}f\ar[r]_-{K(1_A,{R}f)}\rruppertwocell<\omit>{^<-2.3>\varepsilon_f\hole}\ar@/^18pt/@{..>}[rr]^1
    \ar[dd]_{{R}{{L}f}}&
    Kf\ar[r]_-{\sigma_f}\ar[dd]_{{R}f}&
    K{L}f\ar[d]^{{R}{{L}f}}\\
    &&&Kf\ar[d]^{{R}f}\\
    &Kf\ar[r]^-{{R}f}\ar[d]_{{R}f}&
    B\ar@{=}[r]\ar[d]_1&
    B\ar[d]^1\\
    B\ar@{=}[r]&
    B\ar@{=}[r]&
    B\ar@{=}[r]&
    B
    \enddiagram
  \end{equation}
  This morphism is equal to $(\sigma_f,1_B)$, since
  $K(1_A,{R}f)\cdot\sigma_f=1$. Now consider the dotted identity arrow and the
  2-cell $\varepsilon_f$ in~\eqref{eq:57}, observing that it defines a 2-cell
  \begin{equation}
    (\varepsilon_f,1)\colon (\sigma_f,1)\cdot{}R\Phi_f\Longrightarrow 1\label{eq:64}
  \end{equation}
  in $\K^\two$, and, upon precomposing with $\sigma_f$, a 2-cell
  \begin{equation}
    \label{eq:47}
    % \big( (\rho_f,1_B)\bullet(\varepsilon_f,1_{1_B})\big)\cdot{}(\sigma_f,1)=
    (\varepsilon_f\cdot{}\sigma_f,1_{1_B})\colon
    (\sigma_f,1_B)\Longrightarrow(\sigma_f,1_B)
  \end{equation}
  with equal domain and codomain.
  This 2-cell~\eqref{eq:47} precomposed with
  $\Lambda_f\colon f\to Rf$ equals the identity, for
  $\sigma_f\cdot{}{L}f={L}^2f$ and
  $\varepsilon_f\cdot{}L^2f=1$ by definition of $\varepsilon_f$. Since
  $(\sigma_f,1)$ is a left extension of $(\sigma_f,1)\cdot{}\Lambda_f$ along
  $\Lambda_f$, we must have $(\varepsilon_f\cdot{}\sigma_f,1_{1_B})=1$, the first of
  the equalities~\eqref{eq:39}.

  In order to prove the second equality of~\eqref{eq:39}, consider the morphism of $\mathsf{R}$-algebras
  \begin{equation}
    \label{eq:193}
    RLf\xrightarrow{R\Phi_f}RL\xrightarrow{(\sigma_f,1_B)}Rf\bullet
    RLf\xrightarrow{({R}f,1_B)\bullet R\Phi_f}1_B\bullet Rf
  \end{equation}
  that can be depicted as in the following diagram --~of solid arrows.
  \begin{equation}
    \label{eq:192}
    \xymatrixrowsep{.5cm}
    \xymatrixcolsep{1.1cm}
    \diagram
    K{L}f\ar[dd]_{{R}{{L}f}}\ar[r]_-{K(1_A,{R}f)}
    \rruppertwocell<\omit>{^<-2.3>\varepsilon_f\hole}
    \ar@/^18pt/@{..>}[rr]^1&
    Kf\ar[dd]_{{R}f}\ar[r]_-{\sigma_f}&
    K{L}f\ar[d]_{{R}{{L}f}}\ar[r]^-{K(1_A,{R}f)}&
    Kf\ar[d]^{{R}f}\\
    &&Kf\ar[r]^-{{R}f}\ar[d]_{{R}f}&
    B\ar[d]^1\\
    Kf\ar[r]^-{{R}f}&
    B\ar@{=}[r]&
    B\ar@{=}[r]&
    B
    \enddiagram
  \end{equation}
  This morphism equals
  $(K(1_A,{R}f),1)$, since $K(1_A,{R}f)\cdot\sigma_f=1$ by the counit
  axiom of the comonad $\mathsf{L}$. If we now consider the dotted identity
  arrow, the 2-cell $\varepsilon_f$ induces an endo-2-cell
  \begin{equation}
    \label{eq:191}
    % ((\rho_f,1_B)\bullet R\Phi_f)\cdot{}(\varepsilon_f,1_{1_B})
    (K(1_A,{R}f)\cdot\varepsilon_f,1_{{R}f})\colon
    (K(1_A,{R}f),{R}f)\Longrightarrow(K(1_A,{R}f),{R}f),
  \end{equation}
  which, by definition of $\varepsilon_f$~\eqref{eq:402}, equals the identity when precomposed
  with $\Lambda_{Lf}$. The morphism $(K(1_A,{R}f),{R}f)$ is a morphism of
  $\mathsf{R}$-algebras, and hence a left extension along $\Lambda_{Lf}$, from
  where we deduce that~\eqref{eq:191} must be the identity 2-cell. That is,
  $K(1_A,{R}f)\cdot{}\varepsilon_f=1$, the second equality of~\eqref{eq:39}.

  All that remains to verify is $\varepsilon_f\cdot{}{L}^2{f}=1$, but this
  is part of the definition of $\varepsilon_f$, completing the proof.
\end{proof}

\section{Lifting operations}
\label{sec:lifting-operations}
We turn to the second part of the article where we put the emphasis on
\emph{lifting operations} and their relationship to \textsc{awfs}s.
In this section and the next we leave the case of 2-categories
and return to the framework of ordinary categories. After setting out our own
approach to lifting operations, we recall a number of notions known for ordinary
\textsc{awfs}. This is a necessary step previous to extending these notions to lax
idempotent \textsc{awfs}s from Section~\ref{sec:lax-natur-diag} onwards.

\subsection{Background on modules}
\label{sec:digr-into-modul}
As a preamble to the next section, let us briefly remind the reader about the
language of modules or profunctors, which will be heavily used henceforth.
\begin{df}
  \label{df:13}
A \emph{module} or \emph{profunctor} $\phi$ from a category $\mathcal{A}$ to a
category $\mathcal{B}$, denoted by
$\phi\colon\mathcal{A}\tor\mathcal{B}$, is a functor
$\mathcal{B}^{\mathrm{op}}\times\mathcal{A}\to\mathbf{Set}$, and a module
morphism is a natural transformation. Given another module
$\psi\colon\mathcal{B}\tor\mathcal{C}$, the composition $\psi\cdot{}\phi$ is
defined by the coend formula $(\psi\cdot{}\phi)(c,a)=\int^B\psi(c,b)\times\phi(b,a)$;
the identity $1_{\mathcal{A}}$ for this composition is given by
$1_{\mathcal{A}}(a,a')=\mathcal{A}(a,a')$. In this way we obtain a bicategory
$\mathbf{Mod}$.
\end{df}
Each functor $F\colon\mathcal{A}\to\mathcal{B}$ induces two modules
$F_*$ and $F^*$ given by $F_*(b,a)=\mathcal{B}(b,Fa)$ and
$F^*(a,b)=\mathcal{B}(Fa,b)$. Furthermore, there is an adjunction $F_*\dashv
F^*$ with unit and counit given by components
\begin{equation}
  \label{eq:10}
  \mathcal{A}(a,a')\xrightarrow{F}\mathcal{B}(Fa,Fa')\cong F^*\cdot{}F_*(a,a')
\end{equation}
\begin{equation}
  F_*\cdot{}F^*(b,b')=
  \int^A\mathcal{B}(b,Fa)\times\mathcal{B}(Fa,b')\xrightarrow{\mathrm{comp}}
  \mathcal{B}(b,b').
\end{equation}
The coend form of the Yoneda lemma implies that $(\psi\cdot F_*)(a,c)\cong
\psi(c,Fa)$ and $F^*\cdot\chi(a,d)\cong \chi(Fa,d)$, whenever these compositions
of modules are defined.

Similarly, if $\alpha\colon F\Rightarrow G$ is a natural transformation between
functors $\mathcal{A}\to\mathcal{B}$, then there are morphisms of modules
$\alpha_*\colon F_*\to G_*$ and $\alpha^*\colon G^*\to F^*$, with components
\begin{equation}
  \label{eq:101}
  \alpha_*(b,a)=\mathcal{B}(b,\alpha_a)
  \qquad
  \alpha^*(a,b)=\mathcal{B}(\alpha_a,b).
\end{equation}

\subsection{Lifting operations}
\label{sec:natur-diag-fill}

Fix a category $\mathcal C$. Recall that there are adjunctions
\begin{equation}
  \label{eq:59}
  \cod\dashv\id\dashv\dom\colon\C^{\mathbf2}\longrightarrow\C
\end{equation}
the first of which has identity counit and the second of which has identity unit.

Define a module (profunctor) ${\dd}\colon\mathcal C^{\mathbf2}\tor\mathcal C^{\mathbf2}$ in
the following way. Given two morphisms $f$, $g$ in $\mathcal C$, ${\dd}(f,g)$ is the set of commutative diagrams of the form
\begin{equation}
  \label{eq:34}
  \xymatrixrowsep{0.6cm}
  \diagram
  A\ar[d]_f\ar[r]^h&C\ar[d]^g\\
  B\ar[r]_k\ar[ur]^d&D
  \enddiagram
\end{equation}
The action of $\mathcal C^{\mathbf2}$ on either side is simply by pasting the
appropriate commutative square.

\begin{lemma}
  \label{l:2}
  There
  are isomorphisms of modules between ${\dd}$ and the following four modules
  $\mathcal{C}^\two \tor\mathcal{C}^\two$.
  \begin{equation}
    \cod^*\cdot{}\dom_*\qquad \id_*\cdot{}\id^*\qquad
    (\id\cdot{}\dom)_*\qquad (\id\cdot{}\cod)^*
    \label{eq:78}
  \end{equation}
\end{lemma}
\begin{proof}
  The bijection ${\dd}(f,g)\cong (\cod^*\cdot\dom_*)(f,g)\cong
  \mathcal{C}(\cod(f),\dom(g))$ is the obvious one, that sends a commutative
  square with a diagonal filler $d$ as in~\eqref{eq:34} to the morphism $d$. Due
  to the adjunctions~\eqref{eq:59}, $\cod^*\cong\id_*$ and $\id^*\cong\dom_*$,
  and we obtain isomorphisms of $\cod^*\cdot\dom_*$ with
  $(\id\cdot\dom)_*$, and  with $\id_*\cdot\id^*$, and with
  $\cod^*\cdot{}{\id^*}\cong(\id\cdot\cod)^*$.
\end{proof}
The isomorphism of the previous lemma can be given in very explicit terms. For
example, the isomorphism ${\dd}(f,g)\cong(\id\cdot{}\dom)_*(f,g)=\C^{\mathbf{2}}(f,1_{\dom g})$ is given by
  \begin{equation}
    \label{eq:60}
    \xymatrixrowsep{0.6cm}
    \diagram
    {\cdot}
    \ar[r]^-{h}\ar[d]_{f}
    &
    {\cdot}
    \ar[d]^{g}
    \\
    {\cdot}
    \ar[r]_-{k}\ar[ur]^d
    &
    {\cdot}
    \enddiagram
    \longmapsto
    \diagram
    {\cdot}
    \ar[r]^-{h}\ar[d]_{f}
    &
    {\cdot}
    \ar[d]^{1_{\dom g}}
    \\
    {\cdot}
    \ar[r]_-{d}
    &
    {\cdot}
    \enddiagram
  \end{equation}

The counit of $\id_*\dashv\id^*$ is a module morphism
\begin{equation}
  {\dd}\longrightarrow 1_{\mathcal C^{\mathbf2}}\label{eq:54}
\end{equation}
whose component at $(f,g)$ sends the element \eqref{eq:34} to
the outer commutative square. It corresponds, under ${\dd}\cong(\id\cdot{}\dom)_*$, to the module morphism induced by the natural
transformation with $f$-component
\begin{equation}
  \label{eq:58}
  \id\cdot{}\dom\Longrightarrow{}1_{\C^{\mathbf{2}}}
  \qquad
  \xymatrixrowsep{0.6cm}
  \diagram
  {\cdot}
  \ar@{=}[r]
  \ar[d]_{1_{\dom f}}
  &
  {\cdot}
  \ar[d]^{f}
  \\
  {\cdot}\ar[r]^f
  &
  {\cdot}
  \enddiagram
\end{equation}

\begin{df}
  Let $(\mathcal A,U)$, $(\mathcal B,V)$ be two objects of $\Cat\slash\mathcal
  C^{\mathbf2}$, and define a module
  \begin{equation}
    \label{eq:36}
    {\dd}(U,V)\colon \mathcal B\xrightarrow{V_*}
    \mathcal C^{\mathbf2} \xrightarrow{{\dd}}
    \mathcal C^{\mathbf2} \xrightarrow{U^*}\mathcal A.
  \end{equation}
  The module morphism ${\dd}\to 1$ induces another ${\dd}(U,V)\to
  U^*\cdot{}V_*$. A \emph{lifting operation} for $U$,
  $V$ is a section for this module morphism, and amounts to a
  choice, for each square in $\mathcal C$ of the form
  \begin{equation}
    \label{eq:37}
    \xymatrixrowsep{0.6cm}
    \diagram
    {A}
    \ar[r]^-{h}\ar[d]_{Ua}
    &
    {C}
    \ar[d]^{Vb}
    \\
    {B}
    \ar[r]_-{k}
    &
    {D}
    \enddiagram
    \qquad
    a\in\mathcal A,\ b\in\mathcal B
  \end{equation}
  of a diagonal filler $d(h,k)$, in such a way that it is natural with respect to
  composition on either side.\label{df:5}
  To expand on this point, suppose given morphisms $\alpha\colon
  a'\to a$ in $\mathcal{A}$ sent by $U$ to the square $U\alpha=(x,y)\colon
  Ua'\to Ua$, and $\beta\colon b\to b'$ in $\mathcal{B}$ sent by $V$ to
  $V\beta=(u,v)\colon Vb\to Vb'$. The naturality of $d(h,k)$ means that the
  equality below holds.
  \begin{equation}
    \label{eq:95}
    \diagram
    \cdot\ar[r]^x\ar[d]_{U a'}&
    \cdot\ar[r]^h\ar[d]_{Ua}&
    \cdot\ar[r]^{u}\ar[d]^{Vb}&
    \cdot\ar[d]^{Vb'}\\
    \cdot\ar[r]_{y}&
    \cdot\ar[r]_-k\ar[ur]|{d(h,k)}&
    \cdot\ar[r]_v&
    \cdot
    \enddiagram
    =
    \xymatrixcolsep{2cm}
    \diagram
    \cdot\ar[r]^-{u\cdot h\cdot x}\ar[d]_{Ua'}&
    \cdot\ar[d]^{Vb'}\\
    \cdot\ar[r]_-{v\cdot k\cdot y}\ar[ur]|{d(u\cdot h\cdot x,v\cdot k\cdot y)}&
    \cdot
    \enddiagram
  \end{equation}
\end{df}

\begin{ex}
  \label{ex:3}
  A functorial factorisation system, with associated copointed endofunctor
  $(L,\Phi)$ and pointed endofunctor $(R,\Lambda)$, gives rise to a lifting
  operation for the forgetful functors $U\colon
  (L,\Phi)\text-\mathrm{Coalg}\to\mathcal C^{\mathbf 2}$ and $V\colon
  (R,\Lambda)\text-\mathrm{Alg}\to\mathcal C^{\mathbf 2}$. The section to the
  module morphism ${\dd}(U,V)\to U^*\cdot V_*$ has component at
  an object $(f,g)$ of
  $(L,\Phi)\text-\mathrm{Coalg}\times(R,\Lambda)\text-\mathrm{Alg}$ described in
  the following terms. If
  $(1,s)\colon f\to Lf$ is the coalgebra structure of $f$ and $(p,1)\colon Rg\to
  g$ the algebra structure of $g$, the component is given by
  \begin{equation}
    \label{eq:71}
    \C^{\mathbf{2}}(f,g)\xrightarrow{L}\C^{\mathbf{2}}(Lf,Lg)
    \xrightarrow{\C^{\mathbf{2}}(s,1)} \C^{\mathbf{2}}(f,Lg)
    \xrightarrow{\C^{\mathbf{2}}(1,(1,p))} \C^{\mathbf{2}}(f,1_{\dom(g)})
  \end{equation}
  \begin{equation}
    \label{eq:87}
    \xymatrixrowsep{0.6cm}
    \diagram
    {\cdot}
    \ar[r]^-{h}\ar[d]_{f}
    &
    {\cdot}
    \ar[d]^{g}
    \\
    {\cdot}
    \ar[r]_-{k}
    &
    {\cdot}
    \enddiagram
    \longmapsto
    \xymatrixcolsep{1.6cm}
    \diagram
    {\cdot}
    \ar[r]^-{h}\ar[d]_{f}
    &
    {\cdot}
    \ar[d]^1
    \\
    {\cdot}
    \ar[r]_-{p\cdot{}K(h,k)\cdot{}s}
    &
    {\cdot}
    \enddiagram
  \end{equation}
  The reader would have noticed that the diagonal filler $p\cdot K(h,k)\cdot s$
  so obtained is the same one mentioned in Section~\ref{sec:definition-awfs}
  and that we reproduce below.
  \begin{equation}
    \label{eq:113}
    \xymatrixcolsep{1.5cm}
    \xymatrixrowsep{.5cm}
    \diagram
    A\ar[rr]^h\ar[dd]_f\ar[dr]^{{L}f}
    &&
    C\ar[d]_{{L}g}\ar@{=}[r]
    &
    C\ar[dd]^g
    \\
    &
    Kf\ar[r]^-{K(h,k)}\ar[d]^{{R}f}
    &
    Kg\ar[dr]_{{R}g}\ar[ur]_{p}
    &
    \\
    B\ar[ur]^s\ar@{=}[r]
    &
    B\ar[rr]^k
    &&
    B
    \enddiagram
  \end{equation}

  There is an equivalent description of~\eqref{eq:71} that, instead of using
  $\mathcal{C}^\two(f,1_{\dom(g)})\cong{\dd}(U,V)(f,g)$, uses
  $\mathcal{C}^\two(1_{\cod(f)},g)\cong{\dd}(U,V)(f,g)$.
  From this point of view, the section takes the form
  \begin{equation}
    \label{eq:72}
    \C^{\mathbf{2}}(f,g)\xrightarrow{R}\C^{\mathbf{2}}(Rf,Rg)
    \xrightarrow{\C^{\mathbf{2}}(1,p)} \C^{\mathbf{2}}(Rf,g)
    \xrightarrow{\C^{\mathbf{2}}((s,1),1)} \C^{\mathbf{2}}(1_{\cod(f)},g)
  \end{equation}
  \begin{equation}
    \label{eq:73}
    \xymatrixrowsep{0.6cm}
    \diagram
    {\cdot}
    \ar[r]^-{h}\ar[d]_{f}
    &
    {\cdot}
    \ar[d]^{g}
    \\
    {\cdot}
    \ar[r]_-{k}
    &
    {\cdot}
    \enddiagram
    \longmapsto
    \xymatrixcolsep{1.6cm}
    \diagram
    {\cdot}
    \ar[r]^-{p\cdot{}K(h,k)\cdot{}s}\ar[d]_{1}
    &
    {\cdot}
    \ar[d]^{g}
    \\
    {\cdot}
    \ar[r]_-{k}
    &
    {\cdot}
    \enddiagram
  \end{equation}
\end{ex}

\begin{ex}
  A functorial factorisation corresponds to an orthogonal factorisation system
  when ${\dd}(U,V)\to U^*\cdot{}V_*$ is invertible.
  \label{ex:2}
\end{ex}

\begin{rmk}
  \label{rmk:2}
  Let us now assume that in Definition~\ref{df:5} $U$ has a right adjoint
  $G$. Then, the module~\eqref{eq:36} is isomorphic to $(G\cdot{}\id\cdot{}\dom\cdot{}V)_*$,
  $U^*\cdot{}V_*$ is isomorphic to $(G\cdot{}V)_*$, and the module morphism
  $U^*\cdot{}{\dd}\cdot{}V_*\to U^*\cdot{}V_*$ corresponds to the natural transformation
  \begin{equation}
    \label{eq:132}
    G\cdot{}\id\cdot{}\dom\cdot{}V\Longrightarrow G\cdot{}V
  \end{equation}
  induced by the counit of the adjunction $\id\dashv\dom$, ie the
  transformation with component at $b\in\mathcal{B}$
  \begin{equation}
    \label{eq:135}
    G(1,Vb)\colon G1_{\dom(Vb)}\longrightarrow GVb.
  \end{equation}
\end{rmk}
Now suppose that the functor $U$ is the forgetful functor
$U\colon\mathsf{L}\text-\mathrm{Coalg}\to\C^\two$, for a comonad $\mathsf{L}$,
and still denote by $G$ its right adjoint. Denote by
$F_{\mathsf{L}}\colon\C^\two \to \mathrm{Kl}(\mathsf{L})$ the Kleisli construction
of $\mathsf{L}$. The natural transformation~\eqref{eq:132}, belonging to the full image of $G$, can be described as a
morphism in $[\mathcal{B},\mathrm{Kl}(\mathsf{L})]$
\begin{equation}
  \label{eq:133}
  F_{\mathsf{L}}\cdot{}\id\cdot{}\dom\cdot{}V\Longrightarrow F_{\mathsf{L}}\cdot{}V.
\end{equation}

\begin{prop}
  \label{prop:8}
  Given a comonad $\mathsf{L}$ on $\C^\two$, lifting operations for the functors
  $U\colon\mathsf{L}\text-\mathrm{Coalg}\to\C^\two$ and
  $V\colon\mathcal{B}\to\C^\two$ are in bijective correspondence with sections of
  the natural transformation~\eqref{eq:133}.
\end{prop}
\begin{proof}
  The proof is an application of Remark~\ref{rmk:2} to the case when the right
  adjoint is the universal functor into the Kleisli category of the comonad.
\end{proof}

\begin{ex}
  \label{ex:1}
  As we saw in Example~\ref{ex:3}, each \textsc{awfs} $(\mathsf{L},\mathsf{R})$
  on \C\ induces a lifting operation which corresponds to a section
  of~\eqref{eq:133}, by Proposition~\ref{prop:8}. We can describe in explicit
  terms this section as follows.

  If $V\colon\mathsf{R}\text-\mathrm{Alg}\to\C^\two$ is the forgetful functor,
  consider the transformation $(1,p)\colon L\cdot{}V\Rightarrow \id\cdot{}\dom\cdot{}V$ as in
  Remark~\ref{rmk:5}, and denote by $\theta\colon F_{\mathsf{L}}V\tor F_{\mathsf{L}}\id\cdot{}\dom\cdot{}V$ the associated
  morphism in $[\mathsf{R}\text-\mathrm{Alg},\mathrm{Kl}(\mathsf{L})]$. It is
  easy to check that $\theta$ is the required section:
  $F_{\mathsf{L}}(1,g)\cdot{}\theta_g$ is, as a morphism in $\C^\two$,
  \begin{equation}
    \label{eq:145}
    (1,g)\cdot{}(1,p)=(1,{R}g)=\Phi_g.
  \end{equation}
\end{ex}
\section{The universal category with lifting operations}
\label{sec:univ-categ-with}
R.~Garner defined in~\cite{MR2506256} for each functor
$U\colon\mathcal{A}\to\C^\two$ a
category $\mathcal{A}^\pitchfork{}$ and a functor
$U^\pitchfork\colon\mathcal{A}^\pitchfork{}\to\C^\two$ as follows. The objects
of $\mathcal{A}^\pitchfork{}$ are pairs $(g,\phi^g)$, where $g\in\C^\two$ and
$\phi^g$ is an assignment of a diagonal filler for each square
\begin{equation}
  \label{eq:49}
  \xymatrixcolsep{1.4cm}
  \xymatrixrowsep{.7cm}
  \diagram
  \cdot
  \ar[r]^-{h}\ar[d]_{Ua}
  &
  \cdot
  \ar[d]^{g}
  \\
  \cdot
  \ar[r]_-{k}\ar@{..>}[ur]|-{\phi^g(a,h,k)}
  &
  \cdot
  \enddiagram
\end{equation}
which are compatible with morphisms $U\alpha\colon Ua'\to Ua$, in the sense that
\begin{equation}
  \phi^g(a,h,k)\cdot{}\cod(U\alpha)=\phi^g(a',h\cdot{}\dom(U\alpha),k\cdot{}\cod(U\alpha)).
  \label{eq:50}
\end{equation}
A morphism $(g,\phi^g)\to(e,\phi^e)$ is a morphism $(u,v)\colon g\to e$ in
$\C^\two$ such that $u\cdot{}\phi^g(a,h,k)=\phi^e(a,u\cdot{}h,v\cdot{}k)$, for all $(h,k)$.

The functor $U^\pitchfork{}$ just forgets the lifting operations, or in other
words, $U^\pitchfork(g,\phi^g)=g$. There is a canonical lifting operation from
$U$ to $U^\pitchfork$, namely the lifting operation that given a commutative
square $(h,k)\colon Ua\to U^\pitchfork(g,\phi^g)=g$ picks out the diagonal
$\phi^g(a,h,k)$, as in the diagram~\eqref{eq:49}. Furthermore, $U^\pitchfork$
equipped with this lifting operation is universal among functors into $\C^\two$
that are equipped with a lifting
operation against $U$.

The category $\mathcal{A}^\pitchfork$ and the functor
$U^\pitchfork$ can be constructed
as a certain limit in \Cat, of the form
\begin{equation}
  \label{eq:130}
  \xymatrixrowsep{-0.2cm}
  \diagram
  &\C^\two\ar[r]^{Y}&
  \mathcal{P}(\C^\two)\ar[dr]^{\mathcal{P}(U^*)}&\\
  \mathcal{A}^\pitchfork{}\ar[dr]_{U^\pitchfork{}}\ar[ur]^{U^\pitchfork{}}&
  \rtwocell<\omit>{}&&
  \mathcal{P}(\mathcal{A})\\
  &\C^\two\ar[r]_{\widehat{{\dd}}_{\C}}&
  \mathcal{P}(\C^\two)\ar[ur]_{\mathcal{P}(U^*)}&
  \enddiagram
\end{equation}
where $\mathcal{P}(\mathcal{X})$ denotes the presheaf category on $\mathcal{X}$,
and $\widehat{{\dd}}_{\C}$ is the functor associated to ${\dd}$. Equally well, $U^\pitchfork{}$ is a certain enhanced limit, in the
sense of \cite{MR2854177}.

We continue with some further observations from~\cite{MR2506256}.
  The universal property of $U^\pitchfork{}$ implies that lifting operations for
  the pair of functors
  $U\colon\mathcal{A}\to\C^\two\leftarrow\mathcal{B}:V$ are in bijection with
  functors $\mathcal{B}\to\mathcal{A}^\pitchfork{}$ over $\C^\two$.
  In particular, each \textsc{awfs} $(\mathsf{L},\mathsf{R})$ in
  \C\ gives rise to a canonical functor
  $\mathsf{R}\text-\mathrm{Alg}\to\mathsf{L}\text-\mathrm{Coalg}^\pitchfork{}$.
  Furthermore, this functor is fully faithful, as we proceed to show.
  Let $(p,1)\colon Rg\to g$ and $(p',1)\colon Rg'\to g'$ be two $\mathsf{R}$-algebra
  structures, and $(u,v)\colon g\to g'$ a morphism in
  $\mathsf{L}\text-\mathrm{Coalg}^\pitchfork{}$. We know that the chosen
  diagonal filler of the square $(1,{R}g)\colon {L}g\to g$ is $p$, and
  similarly
  for $g'$ and $p'$, so we have $u\cdot{}p=p'\cdot{}K(u,v)$. Hence,
  $(p',1)\cdot{}R(u,v)=(u,v)\cdot{}(p,1)$, so $(u,v)$ is a morphism of $\mathsf{R}$-algebras.

\begin{lemma}
  \label{l:11}
  Given a functor $U\colon\mathcal{A}\to\C^\two$, an adjunction $U\dashv G$, and $g\in\C^\two$,
  there is a bijection between the structure of an object
  $(g,\phi^g)\in\mathcal{A}^\pitchfork{}$ and sections $s$ of $G(1,g)\colon
  G(1_{\dom g})\to Gg$ in $\mathcal{A}$.
  If $(f,\phi^f)\in\mathcal{A}^\pitchfork{}$ is another object, with associated
  section $t$, morphisms
  $(g,\phi^g)\to(f,\phi^f)$ in $\mathcal{A}^\pitchfork{}$ are in bijection with
  morphisms $(h,k)\colon g\to f$ in $\C^\two$ such that $G(h,h)\cdot{}s=t\cdot{}G(h,k)$.
\end{lemma}
\begin{proof}
  See discussion before Proposition \ref{prop:8}.
\end{proof}

\begin{lemma}
  \label{l:12}
  Assume the conditions of Lemma \ref{l:11}. Then, for any full subcategory
  $\mathcal{F}\subset\mathcal{A}$ containing the full image of $G$, the functor
  $\mathcal{A}^\pitchfork\to\mathcal{F}^\pitchfork$ induced by the inclusion is
  an isomorphism.
\end{lemma}
\begin{proof}
  Denote by $J\colon\mathcal{F}\hookrightarrow{}\mathcal{A}$ the fully faithful
  inclusion functor. The composition $UJ$ has a right adjoint $H\colon
  \C^\two\to\mathcal{F}$, given by $H(f)=G(f)$.
  An object of $\mathcal{F}^\pitchfork$ is a lifting operation for the
  functors $UJ$ and $g\colon\mathbf1\to\mathcal{C}^\two$,
  ie a section to the module morphism
  $(UJ)^*\cdot{}{\dd}\cdot{}g_*\to(UJ)^*\cdot{}g_*$. The same data can be
  equally given by a section to the morphism $H(1,g)\colon H(1_{\dom(g)})\to Hg$
  in $\mathcal{F}$; or a section to the image of this morphism under the fully
  faithful $J$. But $JH=G$, so we simply have a section of $G(1,g)$, which is precisely
  an object of $\mathcal{A}^\pitchfork$ by Lemma~\ref{l:11}. This shows that
  $\mathcal{A}^\pitchfork\to\mathcal{F}^\pitchfork$ is bijective on objects.
  The proof that it is fully faithful is along the same lines, and is left to
  the reader.
\end{proof}

\begin{cor}
  \label{cor:5}
  If $\mathsf{L}$ is a domain preserving comonad on $\C^\two$,
  the category $\mathsf{L}\text-\mathrm{Coalg}^\pitchfork{}$ can be described
  as the category with objects pairs $(g,d^g)$ satisfying the commutativity of
  \begin{equation}
    \label{eq:134}
    \xymatrixrowsep{.7cm}
    \diagram
    \cdot\ar@{=}[r]\ar[d]_{{L}g}&
    \cdot\ar[d]^g\\
    \cdot\ar[r]_{{R}g}\ar[ur]^{d^g}&
    \cdot
    \enddiagram
  \end{equation}
  and morphisms $(g,d^g)\to(f,d^f)$ morphisms $(h,k)\colon g\to f$ in $\C^\two$
  such that $h\cdot{}d^g=d^f\cdot{}K(h,k)$.
  If $\mathcal{F}\subset\mathsf{L}\text-\mathrm{Coalg}$ is a full
  subcategory containing the cofree $\mathsf{L}$-coalgebras, the induced functor
  % over $\C^\two$
  % \begin{equation}
  %  \label{eq:142}
  $  \mathsf{L}\text-\mathrm{Coalg}^\pitchfork{}\longrightarrow{}
    \mathcal{F}^\pitchfork{}$
  % \end{equation}
  over $\C^\two$
  is an isomorphism.
\end{cor}
\begin{proof}
  % The proof begins with an application of Lemma~\ref{l:12} to the cofree
  % $\mathsf{L}$-coalgebra adjunction, in virtue of which
  % $\mathsf{L}\text-\mathrm{Coalg}^\pitchfork$ is isomorphic to
  % $\mathcal{F}^\pitchfork$, where $\mathcal{F}$ is the full subcategory defined
  % by the coalgebras of the form $Lf$.
  An object of $\mathsf{L}\text-\mathrm{Coalg}^\pitchfork$ can be described, by
  Lemma~\ref{l:11}, as a morphism $g$ of $\C$ equipped with a section
  $s=(s_0,s_1)$ to $L(1,g)\colon  L1_{\dom g}\to Lg$ that is a morphism of
  $\mathsf{L}$-coalgebras. In fact, $s_0=1$ since $L(1,g)$ has identity domain
  component. The morphism of coalgebras $(1,s)\colon Lg\to L1_{\dom g}$
  corresponds to a unique morphism $(1,d^g)\colon Lg\to 1_{\dom g}$ in
  $\C^\two$, where $d^g={R}{1_{\dom g}}\cdot s_1$, by the cofree coalgebra
  adjunction. Similarly, the equality of coalgebra morphisms $1_{Lg}=L(1,g)\cdot s$
  can be translated into
  \begin{equation}
    \Phi_g=\Phi_g\cdot L(1,g)\cdot s=(1,g)\cdot
    \Phi_{1_{\dom g}}\cdot s;\label{eq:79}
  \end{equation}
  the domain component of each morphism in this string of equalities is an
  identity morphism, while the codomain component yields
  ${R}g= g\cdot {R}{1_{\dom g}}\cdot s_1=g\cdot d^g$. So far we have proven
  that $1=L(1,g)\cdot s$ is equivalent to the commutativity of the bottom
  triangle in~\eqref{eq:134}. Again by the free coalgebra adjunction, the top
  triangle in~\eqref{eq:134}, namely $d^g\cdot {L}g=1$,
  is equivalent to $s_1\cdot{L}g=L({1_{\dom g}})$, which says that
  $(1,s_1)$ is a morphism $Lg\to L1_{\dom g}$. This completes the description of
  the objects of $\mathsf{L}\text-\mathrm{Coalg}^\pitchfork$.

  We now prove the part of the statement relating to morphisms. Suppose that
  $(h,k)\colon (g,\phi^g)\to(f,\phi^f)$ is a morphism in
  $\mathsf{L}\text-\mathrm{Coalg}^\pitchfork$ and $s\colon Lg\to L1_{\dom_g}$
  and $t\colon Lf\to L1_{\dom f}$ the sections of $L(1,g)$ and $L(1,f)$ provided
  by Lemma~\ref{l:11}. By the same lemma, the condition of $(h,k)$ being a
  morphism in $\mathsf{L}\text-\mathrm{Coalg}^\pitchfork$ is equivalent to the
  equality $L(h,h)\cdot s=t\cdot L(h,k)$, which is equivalent to
  \begin{equation}
    (h,h)\cdot\Phi_g\cdot s=
    \Phi_f\cdot L(h,h)\cdot s=\Phi_f\cdot t\cdot L(h,k).\label{eq:85}
  \end{equation}
  The domain component of this equality is trivial, and so it is equivalent to
  the equality of its codomain component, which is
  \begin{equation}
    \label{eq:89}
    h\cdot d^g= h\cdot{R}g\cdot s = {R}f\cdot t\cdot K(h,k)= d^f\cdot K(h,k).
  \end{equation}
  The last sentence of the statement follows from Lemma~\ref{l:12}, completing the proof.
\end{proof}

\section{KZ lifting operations}
\label{sec:lax-natur-diag}
Section~\ref{sec:natur-diag-fill} described the algebraic structure that
provides a lifting operation, and the category
$\mathcal{A}^\pitchfork{}$, in terms of modules. This section introduces
variations of these notions that are suitable to lax orthogonal factorisations.

\subsection{Lifting operations in 2-categories}
\label{sec:lax-natural-diagonal-1}

Before introducing the main definitions of this section, let us remind the
reader about some facts around \Cat-modules. A \Cat-module $\phi$ from a 2-category
$\mathscr{A}$ to another $\mathscr{B}$, denoted by
$\phi\colon\mathscr{A}\tor \mathscr{B}$, is a 2-functor
$\mathscr{B}^{\mathrm{op}}\times\mathscr{A}\to\Cat$. A difference with the case
of modules between ordinary categories is that for \Cat-modules there is a
2-category $\Cat\text-\mathbf{Mod}(\mathscr{A},\mathscr{B})$ of \Cat-modules
$\mathscr{A}\tor\mathscr{B}$: the morphisms are 2-natural transformations and
the 2-cells are the modifications.

Given a \Cat-module $\phi\colon\mathscr{B}\tor\mathscr{C}$, and 2-functors
$F\colon\mathscr{A}\to\mathscr B$ and $G\colon\mathscr{D}\to\mathscr{C}$, we
write $\phi\cdot F_*\colon\mathscr{A}\tor\mathscr{C}$ and
$G^*\cdot\phi\colon\mathscr{B}\tor\mathscr D$ for the modules defined by the formulas
\begin{equation}
  \label{eq:63}
  (\phi\cdot F_*)(c,a)=\phi(c,Fa)\quad\text{and}\quad (G^*\cdot\phi)(d,b)=\phi(Gd,b).
\end{equation}

In a completely analogous way to the case of
$\mathbf{Set}$-modules or profunctors addressed in
Section~\ref{sec:digr-into-modul}, we have the following facts:
\begin{itemize}
\item Each 2-functor $F\colon\mathscr{A}\to\mathscr{B}$ induces a pair of
  $\Cat$-modules $F_*\colon \mathscr{A}\tor\mathscr{B}$ and
  $F^*\colon\mathscr{B}\tor\mathscr{A}$, by the formulas
  $F_*(b,a)=\mathscr{B}(b,Fa)$ and $F^*(a,b)=\mathscr{B}(Fa,b)$.
\item Each 2-natural transformation $\alpha\colon F\Rightarrow G$ induces a
  morphism of \Cat-modules $\alpha_*\colon F_*\to G_*$ by
  $\alpha_*(b,a)=\mathscr{B}(b,\alpha_a)$.
\item If $F$, $G\colon\mathscr{A}\to\mathscr{B}$ are 2-functors, each morphism
  of \Cat-modules $F_*\to G_*$ is of the form $\alpha_*$ for a unique 2-natural
  transformation $\alpha\colon F\Rightarrow G$.
\end{itemize}

Let us write $|\mathscr{A}|$ for the objects of the 2-category $\mathscr{A}$.
The forgetful 2-functor from $\Cat\text-\mathbf{Mod}(\mathscr{A},\mathscr{B})$
to $\Cat^{|\mathscr{B}|\times|\mathscr{A}|}$ that sends a module $\phi$ to the
family of categories $\{\phi(a,b)|a\in|\mathscr{A}|,b\in|\mathscr{B}|\}$ is
2-monadic, by the usual argument: its left adjoint is given by left Kan
extension along the inclusion of objects
$|\mathscr{B}|\times|\mathscr{A}|\hookrightarrow\mathscr{B}^{\mathrm{op}}\times\mathscr{A}$,
and it is conservative. We denote the associated 2-monad by $\mathsf{T}$.

We now substitute the category $\mathcal C$ in Section~\ref{sec:natur-diag-fill} by a
2-category $\mathscr K$, and make the modules into \Cat-enriched modules. So
${\dd}(f,g)$ is now the category with objects commutative
squares with a diagonal filler as depicted in~\eqref{eq:34},
with morphisms, from an object with diagonal $d$ to another with diagonal $d'$,
given by 2-cells $d\Rightarrow d'$ in $\K$; in other words, ${\dd}(f,g)$ is
isomorphic to the category $\K(\cod f,\dom g)$.
Given 2-functors $U\colon\mathscr A\to\mathscr K^{\mathbf2}$ and
$V\colon\mathscr B\to\mathscr K^{\mathbf2}$, we define ${\dd}(U,V)$
% and ${\sss}(U,V)$
in the same way as we did in Section~\ref{sec:natur-diag-fill} in
the case of ordinary categories, ie ${\dd}(U,V)(a,b)=\K^\two(Ua,Vb)$, with the difference that now the modules are
$\Cat$-enriched.

\begin{df}
  \label{df:7}
  \begin{enumerate}
  \item \label{item:11} A \emph{lifting operation} for a pair of 2-functors $U$,
    $V$ into $\K^\two$ is a \Cat-module that is a section of ${\dd}(U,V)\to
    {U^*\cdot{}V_*}$; in other words, it is a section in
    $\Cat\text-\mathbf{Mod}(\mathscr{A},\mathscr{B})$
    --~the reader may want to compare with Definition~\ref{df:5}.
  \item \label{item:12} A \emph{lax natural lifting operation} for $U$, $V$ is a
    section of ${\dd}(U,V)\to{U^*\cdot{}V_*}$ in
    $\mathsf{T}\text-\mathrm{Alg}_c$, the 2-category of $\mathsf{T}$-algebras
    and oplax morphisms, also known as colax morphisms, for the 2-monad $\mathsf{T}$ on $\Cat^{|\mathscr
      B|\times|\mathscr A|}$ whose algebras are $\Cat$-modules --~an explicit
    description can be found below.
  \end{enumerate}
\end{df}

An object of the 2-category $\mathsf{T}\text-\mathrm{Alg}_c$ is a \Cat-module $\varphi\colon \mathscr
A\tor\mathscr B$, while a morphism $t\colon\varphi\to\psi$ is a family of
morphisms $t_{b,a}\colon \varphi(b,a)\to\psi(b,a)$ that is oplax with respect to
the action of $\mathscr A$ and $\mathscr B$. This means that, given a morphism $f\colon a\to a'$ in
$\mathscr A$, and $g\colon b\to b'$ in $\mathscr B$, there is extra data
\begin{equation}
  \label{eq:40}
  \xymatrixrowsep{.6cm}
  \diagram
  \varphi(b,a')\ar[r]^{t(b,a')}\ar[d]_{\varphi(g,f)}
  \drtwocell<\omit>{^\bar t_{f,g}\hole\hole}&
  \psi(b,a')\ar[d]^{\psi(g,f)}\\
  \varphi(b',a)\ar[r]_{t(b',a)}&
  \psi(b'a)
  \enddiagram
\end{equation}
satisfying coherence axioms.

Each component ${U^*\cdot{}V_*}(a,b)\to{\dd}(U,V)(a,b)$ of the section of
Definition~\ref{df:7} gives a
diagonal filler for each square $Ua\to Vb$ in $\K^\two$. The
oplax morphism structure on the section
can be described as follows. Suppose the morphisms $\alpha\colon a'\to a$ in
$\mathscr A$ and $\beta\colon b\to b'$ in $\mathscr B$ are mapped by $U$ and $V$
to commutative squares in \K
\begin{equation}
  \label{eq:41}
  \xymatrixrowsep{.6cm}
  \diagram
  A'\ar[d]_{Ua'}\ar[r]^x&A\ar[d]^{Ua}\\
  B'\ar[r]^{y}&B
  \enddiagram
  \quad
  \text{and}
  \quad
  \diagram
  C\ar[r]^u\ar[d]_{Vb}&C'\ar[d]^{Vb'}\\
  D\ar[r]^{v}&D'
  \enddiagram
\end{equation}
Consider the diagonal fillers given by the respective components of the section:
\begin{equation}
  \label{eq:42}
  \xymatrixrowsep{.6cm}
  \diagram
  {A}
  \ar[r]^-{h}\ar[d]_{Ua}
  &
  {C}
  \ar[d]^{Vb}
  \\
  {B}
  \ar[r]_-{k}\ar[ur]^d
  &
  {C}
  \enddiagram
  \quad\text{and}\quad
  \diagram
  {A'}
  \ar[r]^-{u\cdot{}h\cdot{}x}\ar[d]_{Ua'}
  &
  {C'}
  \ar[d]^{Vb'}
  \\
  {B'}
  \ar[r]_-{v\cdot{}k\cdot{}y}\ar[ur]^j
  &
  {D'}
  \enddiagram
\end{equation}
Then, the oplax morphism structure on ${U^*\cdot{}V_*}\to{\dd}(U,V)$ provides a 2-cell $\omega=\omega(\alpha,\beta)\colon
j\Rightarrow u\cdot{}d\cdot{}y$, satisfying $(Vb')\cdot{}\omega=1$, $\omega\cdot{} (Ua')=1$, and coherence
conditions that we proceed to describe.
Suppose given an object $d$ of ${\dd}(U,V)(a,b)$ as above, and morphisms in
$\mathscr A$ and $\mathscr B$
\begin{equation}
  \label{eq:45}
  a''\xrightarrow{\alpha'}a'\xrightarrow{\alpha}a
  \quad\text{and}\quad
  b\xrightarrow{\beta}b'\xrightarrow{\beta'}b''
\end{equation}
we have the following diagram, where the dashed arrows are chosen diagonal
fillers.
\begin{equation}
  \label{eq:44}
  \xymatrixcolsep{1.4cm}
  \xymatrixrowsep{1.5cm}\diagram
  A''\ar[d]_{Ua''}\ar[r]^{x'}&
  A'\ar[d]_(.4){Ua'}\ar[r]^x&
  A\ar[d]_(.35){Ua}\ar[r]^h&
  C\ar[d]^(.6){Vb}\ar[r]^{u}&
  C'\ar[d]^{Vb'}\ar[r]^{u'}&
  C''\ar[d]^{Vb''}\\
  B''\ar[r]_{y'}\ar@{..>}[urrrrr]^(.13)e&
  B'\ar[r]_y\ar@{..>}[urrr]|(.18)j&
  B\ar[r]_k\ar@{..>}[ur]^(.7)d&
  D\ar[r]_v&
  D'\ar[r]_{v'}&
  D''
  \enddiagram
\end{equation}
The condition corresponding to the associativity axiom of the oplax morphism
${U^*\cdot{}V_*}\to{\dd}(U,V)$ says that
\begin{equation}
  \label{eq:46}
  \big(e\xrightarrow{\omega(\alpha\cdot{}\alpha',\beta'\cdot{}\beta)}u'\cdot{}u\cdot{}d\cdot{}y\cdot{}y'\big)
  =\big( e\xrightarrow{\omega(\alpha',\beta')}u'\cdot{}j\cdot{}y'
  \xrightarrow{u'\cdot{}\omega(\alpha,\beta)\cdot{}y'}u'\cdot{}u\cdot{}d\cdot{}y\cdot{}y'\big)
\end{equation}
The axiom corresponding to the unit axiom of the oplax morphism $U^*\cdot{}V_*\to{\dd}(U,V)$ says that
$\omega(1,1)=1$.

\subsection{KZ lifting operations}
\label{sec:lax-natural-diagonal}
Having introduced in the previous section lifting operations and lax natural
lifting operations, we now introduce a version of lifting operation that
corresponds to lax orthogonal \textsc{awfs}s, namely, \textsc{kz} lifting
operations.

\begin{df}
  \label{df:8}
  A \textsl{\textsc{kz}} \emph{lifting operation} in $\mathscr K$ for the 2-functors
  $U$, $V$ is a left adjoint section to the morphism ${\dd}(U,V)\to{U^*\cdot{}V_*}$ in the
  2-category $\Cat\text-\mathbf{Mod}(\mathscr A,\mathscr B)$.
\end{df}

In more explicit terms, a \textsc{kz} lifting operation is given by,
for each square~\eqref{eq:37} in $\mathscr K$, a diagonal filler $d(h,k)$, with the
following universal property. For any $d'\colon B\to C$ and any pair of 2-cells
$\alpha$, $\beta$ satisfying
\begin{equation}
  \label{eq:38}
  \diagram
  A\rtwocell^h_{d'\cdot{}Ua}{\alpha}&C\ar[r]^{Vb}&D
  \enddiagram
  =
  \diagram
  A\ar[r]^{Ua}&B\rtwocell^{k}_{Vb\cdot{}d'}{\beta}&D
  \enddiagram
\end{equation}
there exists a unique 2-cell $\gamma\colon d(h,k)\Rightarrow d'$ such that $\gamma\cdot{}Ua=\alpha$
and $Vb\cdot{}\gamma=\beta$.
\begin{equation}
  \label{eq:124}
  \xymatrixcolsep{2cm}
  \diagram
  A\ar[r]^h\ar[d]_{Ua}&B\ar[d]^{Vb}\\
  C\ar[r]_k\urtwocell^{d(h,k)\hole\hole}_{d'}{\hole\exists !\gamma}&D
  \enddiagram
\end{equation}
This universal property makes, by the usual
argument, $d(h,k)$ functorial in $(h,k)$: for any pair of 2-cells $\mu\colon
h\Rightarrow h'$ and $\kappa\colon k\Rightarrow k'$ such that $Vb\cdot
\mu=\kappa\cdot Ua$, there is a 2-cell $d(\mu,\kappa)\colon d(h,k)\Rightarrow d(h',k')$.
Furthermore, the diagonal fillers $d(h,k)$ must be 2-natural in
$a$ and $b$, in the following sense. % If $\alpha\colon a'\to a$ is a 2-cell in
% $\mathscr{A}$, sent by $U$ to the square $(x,y)\colon Ua'\to Ua$, and
% $\beta\colon b\to b'$ is sent by $V$ to the square $(u,v)\colon Vb\to Vb'$, then
Suppose that $\Gamma$ is a 2-cell in $\mathscr{A}$ and $\Theta$ a 2-cell in
$\mathscr{B}$, sent by $U$ and $V$ to 2-cells in $\K^\two$ as depicted.
\begin{equation}
  \label{eq:96}
  \diagram
  \cdot\rtwocell^\alpha_{\alpha'}{\Gamma}&\cdot
  \enddiagram
  \overset{U}{\longmapsto}
  \diagram
  \cdot\rtwocell^x_{x'}{\gamma}\ar[d]_{Ua'}&\cdot\ar[d]^{Ua}\\
  \cdot\rtwocell^y_{y'}{\delta}&\cdot
  \enddiagram
  \qquad
  \diagram
  \cdot\rtwocell^\beta_{\beta'}{\Theta}&\cdot
  \enddiagram
  \overset{V}{\longmapsto}
  \diagram
  \cdot\rtwocell^u_{u'}{\theta}\ar[d]_{Vb}&\cdot\ar[d]^{Vb'}\\
  \cdot\rtwocell^v_{v'}{\tau}&\cdot
  \enddiagram
\end{equation}
The 2-naturality of the lifting operation means that there is an equality of
2-cells $\theta\cdot d(h,k)\cdot\delta=d(\theta\cdot h\cdot\gamma,\tau\cdot
k\cdot\delta)$.
\begin{equation}
  \label{eq:97}
  \diagram
  \cdot\rtwocell^x_{x'}{\gamma}\ar[d]_{Ua'}
  &
  \cdot\ar[d]\ar[r]^h
  &
  \cdot\rtwocell^u_{u'}{\theta}\ar[d]
  &
  \cdot\ar[d]^{Vb'}
  \\
  \cdot\rtwocell^y_{y'}{\delta}
  &
  \cdot\ar[r]_k\ar[ur]|{d(h,k)}
  &
  \cdot\rtwocell^v_{v'}{\tau}
  &
  \cdot
  \enddiagram
  =
  \xymatrixcolsep{3cm}
  % \diagram
  % \cdot\rtwocell^{u\cdot h\cdot x}_{u'\cdot h\cdot x'}\ar[d]_{Ua'}
  % &
  % \cdot\ar[d]^{Vb'}
  % \\
  % \cdot\rtwocell^{v\cdot k\cdot y}_{v'\cdot k\cdot y'}
  % \urtwocell^{d()}
  % &
  % \cdot
  % \enddiagram
  \diagram
  \cdot\ar[d]_{Ua'}
  &{}&
  \cdot\ar[d]^{Vb'}
  \\
  \cdot\ar@/^14pt/[urr]^{d(u\cdot h\cdot x,v\cdot k\cdot y)}
  \ar@/_14pt/[urr]_{d(u'\cdot h\cdot x',v'\cdot k\cdot y')}
  \urrtwocell<\omit>{\omit d(\theta\cdot h\cdot\gamma,\tau\cdot k\cdot\delta)}
  \urtwocell<\omit>{<2>}
  &&
  \cdot
  \enddiagram
\end{equation}

\begin{df}
  \label{df:4}
  A \emph{lax natural} \textsl{\textsc{kz}} \emph{lifting operation} in $\mathscr K$ for the 2-functors
  $U$, $V$ is a left adjoint section to the morphism ${\dd}(U,V)\to{U^*\cdot{}V_*}$ in the
  2-category $\Cat^{\lvert\mathscr{B}\rvert\times\lvert\mathscr{A}\rvert}$.
\end{df}

This means that a lax natural \textsc{kz} lifting operation is given by a
left adjoint section for each component
\begin{equation}
  \label{eq:24}
  {\dd}(U,V)(a,b)\longrightarrow\K^\two(Ua,Vb)\qquad a\in \mathscr A,\
  b\in\mathscr B.
\end{equation}
More explicitly, it is given by a choice, for each square $(h,k)\colon Ua\to
Vb$, of a diagonal filler $d(h,k)\colon \cod(Ua)\to\dom(Vb)$ with the property
that 2-cells $d(h,k)\Rightarrow d'\colon\cod(Ua)\to\dom(Vb)$ are in bijection with
2-cells $(h,k)\Rightarrow (d'\cdot{}Ua,Vb\cdot{}d')$. In other words, the same universal
property of \textsc{kz} lifting operation, except that the chosen diagonals
$d(h,k)$ need not be natural in $a$, $b$.

\begin{rmk}
  \label{rmk:4}
  A lax natural \textsc{kz} lifting operation equates
  to providing, for each $a\in\mathscr A$ and $b\in\mathscr B$, with $Ua\colon A\to
  B$ and $Vb\colon C\to D$, a left adjoint section of the usual comparison functor
  \begin{equation}
    \label{eq:43}
    \mathscr K(B,C)\longrightarrow \mathscr K(A,C)\times_{\mathscr K(A,D)}
    \mathscr K(B,D).
  \end{equation}
  However, the presentation using modules effortlessly yields more,
  as discussed below.
\end{rmk}

\begin{prop}
  \label{prop:5}
  \begin{enumerate}
  \item \label{item:15} Each {\normalfont \textsl{\textsc{kz}}} lifting
    operation is also a lax natural {\normalfont \textsl{\textsc{kz}}} lifting
    operation.
  \item \label{item:16} Each {lax natural} {\normalfont \textsl{\textsc{kz}}}
    lifting operation is also a {lax natural} lifting operation.
  \end{enumerate}
\end{prop}
\begin{proof}
  The first part of the statement follows from the existence of a forgetful
  2-functor
  \begin{equation}
    \Cat\text-\mathbf{Mod}(\mathscr{A},\mathscr{B}) \longrightarrow
    \Cat^{|\mathscr{A}|\times|\mathscr{B}|}.\label{eq:126}
  \end{equation}
  Then, given 2-functors $U\colon\mathscr{A}\to\K^\two$ and
  $V\colon\mathscr{B}\to\K^\two$, a left adjoint section to the canonical morphism
  $\dd(U,V)\to U^*\cdot V_*$ in $\Cat\text-\mathbf{Mod}(\mathscr{A},\mathscr{B})
  $ is also a left adjoint section in $\Cat^{|\mathscr{A}|\times|\mathscr{B}|}$.

  For the second part of the statement, one needs to use the fact that the
  2-functor~\eqref{eq:126} is monadic; this means that there is a 2-monad
  $\mathsf{T}$ on $\Cat^{|\mathscr{A}|\times|\mathscr{B}|}$ and that
  $\Cat\text-\mathbf{Mod}(\mathscr{A},\mathscr{B})$ is the 2-category
  $\mathsf{T}\text-\mathrm{Alg}_s$ of $\mathsf{T}$-algebras and strict
  morphisms. A lax natural \textsc{kz} lifting operation is a left adjoint section
  of~\eqref{eq:126} in the 2-category
  $\Cat^{|\mathscr{A}|\times|\mathscr{B}|}$. Since \eqref{eq:126}~is a strict
  morphism of $\mathsf{T}$-algebras, the \emph{doctrinal adjunction}
  Proposition~\ref{prop:4} ensures that \eqref{eq:126}~has a left adjoint in
  $\mathsf{T}\text-\mathrm{Alg}_c$, which is the definition of lax natural
  lifting operation --~Definition~\ref{df:7}.
\end{proof}
\begin{lemma}
  \label{l:3}
  {\normalfont\textsl{\textsc{kz}}} lifting operations are unique up to
  canonical isomorphism. More precisely, if the diagonal filler $d(h,k)$ and
  $d'(h,k)$ define two {\normalfont\textsl{\textsc{kz}}} lifting operations
  between 2-functors $U\colon\mathscr{A}\to\K^\two\leftarrow\mathscr{B}\colon
  V$, then there exists a unique 2-cell $\gamma(h,k)\colon d(h,k)\Rightarrow
  d'(h,k)$ such that $Vb\cdot\gamma(h,k)=1$ and $\gamma(h,k)\cdot Ua=1$, as
  depicted in~\eqref{eq:124}. Furthermore, $\gamma$ is invertible.
\end{lemma}
\begin{proof}
  This is a direct consequence of the universal property of \textsc{kz} lifting
  operations.
\end{proof}

\begin{prop}
  \label{prop:11}
  {\normalfont\textsl{\textsc{kz}}} lifting operations for the 2-functors $U\colon\mathscr
  A\to\K^\two$ and $V\colon\mathscr{B}\to\K^\two$ are, if $U$ has a right
  adjoint $G$, in bijective correspondence with left adjoint sections of the
  morphism $G\cdot{}\id\cdot{}\dom\cdot{}V\to G\cdot{}V$ induced by the counit of $\id\dashv \dom$
  -- with components $G(1,Vb)$ -- in the 2-category $[\mathscr B,\mathscr A]$ of
  2-functors $\mathscr B\to\mathscr A$.
\end{prop}
\begin{proof}
  By the comments about \Cat-modules at the beginning of
  Section~\ref{sec:lax-natural-diagonal-1} and same argument deployed in
  Remark~\ref{rmk:2}, the \Cat-module
  transformation
  \begin{equation}
    U^*\cdot{}{\dd}\cdot{}V_*\longrightarrow U^*\cdot{}V_*\label{eq:11}
  \end{equation}
  corresponds to the 2-natural transformation of the statement
  \begin{equation}
    G\cdot{}\id\cdot{}\dom\cdot{}V\Longrightarrow G\cdot{}V.
    \label{eq:12}
  \end{equation}
  Since the 2-functor from $[\mathscr B,\mathscr A]$
  to $\Cat\text-\mathbf{Mod}(\mathscr B,\mathscr A)$ that sends $F$ to $F_*$ is
  full and faithful --~an isomorphism on hom-categories~--
  \eqref{eq:11}~has a left adjoint coretract if and only if
  \eqref{eq:12} does.
\end{proof}

\begin{prop}
  \label{prop:10}
  Given a lax orthogonal {\normalfont\textsl{\textsc{awfs}}} $(\mathsf{L},\mathsf{R})$ on \K, the 2-natural
  transformation $F^{\mathsf{L}}\cdot{}\id\cdot{}\dom\cdot{}V\Rightarrow F^{\mathsf{L}}\cdot{}V$ induced
  by the counit of\/ $\id\dashv \dom$ has a left adjoint section in
  $[\mathsf{R}\text-\mathrm{Alg}_s,\mathsf{L}\text-\mathrm{Coalg}_s]$, where
  $F^{\mathsf{L}}\colon\K^\two\to\mathsf{L}\text-\mathrm{Coalg}_s$ is the cofree
  coalgebra 2-functor and $V$ the forgetful 2-functor from
  $\mathsf{L}\text-\mathrm{Coalg}_s$.
\end{prop}
\begin{proof}
  Given an $\mathsf{R}$-algebra structure $(p_g,1)\colon Rg\to g$ we need to
  exhibit a coretract adjunction in $\mathsf{L}\text-\mathrm{Coalg}_s$ with
  right adjoint $L(1,g)\colon L1_{\dom(g)}\to Lg$. We know from
  Remark~\ref{rmk:7} that there is a coretract adjunction $(1,p_g)\dashv
  (1,{L}g)$, whose unit we denote by $\eta_g$; the same remark points out
  that these adjunctions are 2-natural in $(g,p_g)$. Together with the adjunction
  $\Sigma_g\dashv L\Phi_g$ that exhibits $\mathsf{L}$ as lax idempotent, we
  obtain
  \begin{equation}
    \label{eq:23}
    L(1,p_g)\cdot{}\Sigma_g\dashv L\Phi_g\cdot{}L(1,{L}g)=L(1,g).
  \end{equation}
  The unit of this composition of adjunctions is
  \begin{equation}
    \label{eq:26}
    1=L\Phi_g\cdot{}\Sigma_g\xrightarrow{L\Phi_g\cdot{}L(\eta_g)\cdot{}\Sigma_g}
    L\Phi_g\cdot{}L(1,{L}g)\cdot{}L(1,p)\cdot{}\Sigma_g=1,
  \end{equation}
  which is the identity since $\Phi_g\cdot{}\eta_g=1$ -- again by Remark~\ref{rmk:7}.
\end{proof}

\begin{thm}
  \label{thm:2}
  Each lax orthogonal {\normalfont\textsl{\textsc{awfs}}} $(\mathsf L,\mathsf R)$ on the 2-category $\mathscr K$
  induces
  \begin{enumerate}
  \item \label{item:39} A {\normalfont\textsl{\textsc{kz}}} lifting operation for
    $\mathsf{L}\text-\mathrm{Coalg}_s\to\mathscr K^{\mathbf{2}}$ and
    $\mathsf{R}\text-\mathrm{Alg}_s\to\mathscr K^{\mathbf{2}}$.
  \item \label{item:38} A lax natural {\normalfont\textsl{\textsc{kz}}} lifting operation for
    $U_\ell\colon\mathsf{L}\text-\mathrm{Coalg}_\ell\to\mathscr K^{\mathbf{2}}$ and
    $V_\ell\colon\mathsf{R}\text-\mathrm{Alg}_\ell\to\mathscr K^{\mathbf{2}}$.
  \end{enumerate}
  Moreover, the diagonal fillers are those given by the {\normalfont\textsl{\textsc{awfs}}} in the usual
  way --~\eqref{eq:8}.
\end{thm}
\begin{proof}
  The first part is a direct consequence of Propositions~\ref{prop:11} and
  \ref{prop:10}. The second part means that there must exist a left adjoint
  coretract to each functor
  \begin{equation}
    {\dd}(U_\ell,V_\ell)((f,s),(g,p))= \K(\cod(f),\dom(g))\longrightarrow
    \K^\two(f,g)\label{eq:27}
  \end{equation}
  where $(f,s)$ is an $\mathsf{L}$-coalgebra and $(g,p)$ an
  $\mathsf{R}$-algebra. We know that such a left adjoint coretract does exist, by
  the first part of the statement, and the proof is complete.
\end{proof}

\begin{rmk}
  \label{rmk:32}
  It may be useful to exhibit the counit of the coretract adjunction in the
  proof of Theorem~\ref{thm:2}, in which \eqref{eq:27}~is the right adjoint,
  even though it is not necessary to prove that result. Let $d$ be a diagonal
  filler for a square
  \begin{equation}
    \label{eq:70}
    \xymatrixrowsep{.4cm}
    \diagram
    A\ar[d]_f\ar[r]^h&
    C\ar[d]^g\\
    B\ar[r]^k&
    D
    \enddiagram
  \end{equation}
  from an $\mathsf{L}$-coalgebra $(f,s)$ to an
  $\mathsf{R}$-algebra $(g,p)$. The diagram on the left below shows the equality
  $K(h,k)=K(1_C,{R}g)\cdot K(1_C,{L}g)\cdot K(h,d)$, while the diagram on
  the right shows that $p={R}{1_C}\cdot K(1_C,p)\cdot\sigma_g$.
  \begin{equation}
    \label{eq:66}
    \xymatrixcolsep{1.3cm}
    \diagram
    \cdot\ar[d]_{{L}f}\ar[r]^h&
    \cdot\ar@{=}[r]\ar[d]^{{L}{1_C}}&
    \cdot\ar@{=}[r]\ar[d]^{L^2g}&
    \cdot\ar[d]^{{L}g}\\
    \cdot\ar[r]^-{K(h,d)}\ar[d]_{{R}f}&
    \cdot\ar[d]^{{R}{1_C}}\ar[r]^{K(1_C,{L}g)}&
    \cdot\ar[r]^{K(1_C,{R}g)}\ar[d]^{{R}{{L}g}}&
    \cdot\ar[d]^{{R}g}\\
    \cdot\ar[r]^d&
    \cdot\ar[r]^{{L}g}&
    \cdot\ar[r]^{{R}g}&
    \cdot
    \enddiagram
    \quad
    \diagram
    \cdot\ar@{=}[r]\ar[d]_{{L}g}&
    \cdot\ar@{=}[r]\ar[d]_{L^2g}&
    \cdot\ar[d]^{{L}{1_C}}\\
    \cdot\ar[r]^{\sigma_g}\ar@{=}[dr]&
    \cdot\ar[d]^{{R}{{L}g}}\ar[r]^{K(1_C,p)}&
    \cdot\ar[d]^{{R}{1_C}}\\
    &\cdot\ar[r]^p&
    \cdot
    \enddiagram
  \end{equation}
  Using these equalities, one can show that the counit
  $p\cdot{}K(h,k)\cdot{}s\Rightarrow d$ given by the \textsc{kz} lifting
  operation can be described by
  \begin{multline}
    \label{eq:403}
    p\cdot{}K(h,k)\cdot{}s=
    {R}{1_C}\cdot{}K(1_C,p)\cdot{}\sigma_g\cdot{}K(1_C,{R}g)\cdot{}K(1_C,{L}g)\cdot{}K(h,d)\cdot{}s
    \Longrightarrow
    \\
    \Longrightarrow
    {R}{1_C}\cdot{}K(1_C,p)\cdot{}K(1_C,{L}g)\cdot{}K(h,d)\cdot{}s
    = {R}{1_C}\cdot{}K(h,d)\cdot{}s
    = d\cdot{}{R}f\cdot{}s=d
  \end{multline}
  where the unlabelled arrow is the one induced by the counit
  $\sigma_g\cdot{}K(1,{R}g)\Rightarrow 1_{Kg}$ that endows the comonad $\mathsf{L}$
  with its lax idempotent structure.
\end{rmk}
Theorem~\ref{thm:2}~(\ref{item:38}) can be rephrased by saying that the usual lifting operation
for $(\mathsf{L},\mathsf{R})$ is, when both $\mathsf{L}$ and
$\mathsf{R}$ are lax idempotent, lax natural with respect to all morphisms
in \K.

\subsection{Lax orthogonal functorial factorisations}
\label{sec:lax-orth-funct}
We have seen in the previous sections that the lifting operation of a
lax orthogonal \textsc{awfs} has the extra structure of a \textsc{kz} lifting
operation. One could ask what extra structure is inherited from a lax orthogonal
\textsc{awfs} to its underlying \textsc{wfs}. Since we work with algebraic factorisations, we have
at our disposal not only mere \textsc{wfs}s but functorial factorisations, and it is for
these that we answer the question.

Let $\mathscr A$, $\mathscr B$ be 2-categories and
$\Cat\text-\mathbf{Mod}(\mathscr B,\mathscr A)$ the 2-category of \Cat-modules
from $\mathscr{B}$ to $\mathscr{A}$. Denote by
$\mathsf{M}$ the 2-monad $(M,\Lambda^{M},\Pi^M)$ on $\Cat\text-\mathbf{Mod}(\mathscr{B},\mathscr{A})^\two$ whose
algebras are morphisms in $\Cat\text-\mathbf{Mod}(\mathscr B,\mathscr A)$ equipped with a left adjoint coretract. A
dual of $\mathsf{M}$ has been described in Section~\ref{sec:basic-example}; more
precisely, if $\mathsf{L}$ is the 2-comonad of Proposition~\ref{prop:19}, whose
algebras are morphisms equipped with a right adjoint retract defined on the 2-category
$(\Cat\text-\mathbf{Mod}(\mathscr B,\mathscr A)^{\mathrm{op}})^\two\cong(\Cat\text-\mathbf{Mod}(\mathscr B,\mathscr A)^\two)^{\mathrm{op}}$, then $\mathsf{M}$ is $\mathsf{L}^\mathrm{op}$. An
algebra for the pointed endo-2-functor $(M,\Lambda^M)$ is a morphism $\alpha\colon\phi\to\psi$ equipped with a
coretract $\sigma\colon\psi\to\phi$ and a 2-cell $m\colon \sigma\cdot{}\alpha\Rightarrow
1$ such that $\sigma\cdot{}m=1$.
This is a dual form of Proposition~\ref{prop:19}~(\ref{item:19}).

\begin{df}
  \label{df:15}
  Consider 2-functors $U$ and $V$ from $\mathscr A$ and $\mathscr B$ into
  $\K^\two$.
  A \emph{lax orthogonality structure} on $U$, $V$ is an $(M,\Lambda^M)$-coalgebra
  structure on the morphism of \Cat-modules $U^*\cdot{}{\dd}\cdot{}V_*\to
  U^*\cdot{}V_*$. Consider a functorial factorisation on \K\ with associated copointed
  endo-2-functor $(L,\Phi)$ and associated pointed endo-2-functor
  $(R,\Lambda)$. A lax orthogonality structure on the functorial factorisation
  is one on $U$, $V$, for $U$ the
  forgetful 2-functor from $(L,\Phi)$-coalgebras and $V$ the forgetful 2-functor
  from
  $(R,\Lambda)$-algebras.
\end{df}
Explicitly, a lax orthogonality structure as in the definition is a choice of
2-natural diagonal fillers $d(a,b)(h,k)\colon\cod(Ua)\to\dom(Vb)$ that is functorial on
squares $(h,k)\colon Ua\to Vb$, and 2-natural in $a\in\mathscr A$, $b\in\mathscr
B$. Furthermore, for any diagonal filler $e$ of $(h,k)$ we are given a 2-cell
$\theta(a,b)(e)\colon d(a,b)(h,k)\Rightarrow e$ that is 2-natural in $e$ and a
modification on $a$, $b$.
\begin{equation}
  \label{eq:9}
  \xymatrixcolsep{1.5cm}
  \diagram
  \cdot\ar[r]^h\ar[d]_{Ua}&
  \cdot\ar[d]^{Vb}\\
  \cdot\ar[r]_k\urtwocell^{d}_e{\hole\theta(e)}&
  \cdot
  \enddiagram
\end{equation}
The 2-cells $\theta(a,b)(e)$ must satisfy
$(Vb)\cdot{}\theta(a,b)(e)=1_k$ and $\theta(a,b)(e)\cdot{}(Ua)=1_h$.
Naturality in $e$ means that for each 2-cell
$\epsilon\colon e\Rightarrow\bar e$ the equality
\begin{equation}
  \label{eq:69}
  \big(\theta(a,b)(\bar e)\big)\big(d(a,b)(\epsilon\cdot{}Ua,Vb\cdot{}\epsilon)\big)
  =
  \epsilon \theta(a,b)(e)
\end{equation}
holds. The modification property for $\theta$ means that,
if $\alpha\colon a'\to a$ and
$\beta\colon b\to b'$ are morphisms in $\mathscr A$ and $\mathscr B$, then
\begin{equation}
  \label{eq:88}
  \dom(V\beta)\cdot{}\theta(a,b)(e)\cdot{}\cod(U\alpha) =
  \theta(a',b')(\dom(V\beta)\cdot{}e\cdot{}\cod(U\alpha)).
\end{equation}
\begin{equation}
  \label{eq:155}
  \xymatrixcolsep{1.5cm}
  \xymatrixrowsep{1.5cm}
  \diagram
  \cdot\ar[d]_{Ua'}\ar[r]^{\dom U\alpha}&
  \cdot\ar[d]_{Ua}\ar[r]^{h}&
  \cdot\ar[d]^{Vb}\ar[r]^{\dom V\beta}&
  \cdot\ar[d]^{Vb'}\\
  \cdot\ar[r]_{\cod U\alpha}&
  \cdot\ar[r]_k\urtwocell^{d}_e{\hole\theta(e)}&
  \cdot\ar[r]_{\cod V\beta}&
  \cdot
  \enddiagram
  =
  \xymatrixcolsep{4cm}
  \xymatrixrowsep{1.5cm}
  \diagram
  \cdot\ar[d]_{Ua'}\ar[r]^{\dom V\beta\cdot h\cdot \dom U\alpha} &
  \cdot\ar[d]^{Vb'}\\
  \cdot\ar[r]_{\cod V\beta\cdot k\cdot\cod U\alpha}
  % \urtwocell^d_{\hole\hole\hole\hole\hole\dom V\beta\cdot e\cdot\cod U\alpha}{}
  \uruppertwocell^d{\theta}\ar[ur]_{\hole\hole\dom V\beta\cdot e\cdot\cod U\alpha}
  &
  \cdot
  \enddiagram
\end{equation}

Observe that there is no reason why $\theta$ should satisfy the extra property
that the endo-2-cell $\theta(a,b)(d(a,b)(h,k))$ of $d(a,b)(h,k)$ be an identity
2-cell.
\begin{rmk}
  In the particular instance when $\mathscr A=\mathscr B=\mathbf{1}$, the 2-functors
  $U$ and $V$ pick out morphisms $f\colon A\to B$ and $g\colon C\to D$ in \K,
  and a lax orthogonality structure for $f$, $g$ can be described
  simply as a functor $D$ that is a section of the canonical comparison functor
  $H$ into the pullback, together with a natural transformation $\theta\colon
  DH\Rightarrow 1$ that satisfies $H\theta=1$. This structure can be described
  as a choice of a diagonal filler $D(h,k)$ for each square $(h,k)$ and a 2-cell
  $\theta(e)\colon D(h,k)\Rightarrow e$ for any other diagonal filler $e$, that
  satisfies $g\cdot{}\theta(e)=1$ and $\theta(e)\cdot{}f=1$.
  \begin{equation}
    \label{eq:92}
    \diagram
    \K(B,C)\ar@<4pt>[r]^-H\ar@<-4pt>@{<-}[r]_-D&
    \K(A,C)\times_{\K(A,D)}\K(B,D)
    \enddiagram
  \end{equation}\label{rmk:17}
\end{rmk}

\begin{prop}
  \label{prop:12}
  The underlying 2-functorial factorisation of a lax orthogonal {\normalfont\textsl{\textsc{awfs}}} carries a
  canonical lax orthogonal structure, whose diagonal fillers are those induced
  by the 2-functorial factorisation in the usual way
  --~as in Example~\ref{ex:3}.
\end{prop}
\begin{proof}
  For an \textsc{awfs} $(\mathsf{L},\mathsf{R})$, consider the forgetful functors $U$ and
  $V$ from, respectively, the 2-categories of $(L,\Phi)$-coalgebras and
  $(R,\Lambda)$-algebras. Denote by $\bar\Sigma=(1,s)\colon U\Rightarrow LU$ the
  $\mathsf{L}$-coalgebra structure of $U$, and $\bar\Pi=(p,1)\colon RV\Rightarrow V$
  the $\mathsf{R}$-algebra structure of $V$.
  In this proof we use the notation introduced in the second paragraph of this
  section: $\mathsf{M}=(M,\Lambda^M,\Pi^M)$ the 2-monad on $\Cat\text-\mathbf{Mod}(\mathsf{R}\text-\mathrm{Alg}_s,\mathsf{L}\text-\mathrm{Coalg}_s)^\two$ whose
  algebras are right adjoint retracts.

  We can form two objects of $\Cat\text-\mathbf{Mod}(\mathsf{R}\text-\mathrm{Alg}_s,\mathsf{L}\text-\mathrm{Coalg}_s)^\two$ depicted as the vertical arrows
  in the square below, induced by the \Cat-module morphism ${\dd}\to 1$
  introduced in~\eqref{eq:54}.
  The morphisms of \Cat-modules $\bar\Sigma^*\colon (LU)^*\to U^*$ and $\bar\Lambda_*\colon
  (RV)_*\to V_*$ induce a morphism in $\Cat\text-\mathbf{Mod}(\mathsf{R}\text-\mathrm{Alg}_s,\mathsf{L}\text-\mathrm{Coalg}_s)^\two$, depicted as the
  commutative diagram in $\Cat\text-\mathbf{Mod}(\mathsf{R}\text-\mathrm{Alg}_s,\mathsf{L}\text-\mathrm{Coalg}_s)$ below.
  \begin{equation}
    \label{eq:61}
    \xymatrixrowsep{.5cm}
    \xymatrixcolsep{1.5cm}
    \diagram
    U^*L^* {}{\dd} {}R_*V_*\ar[d]\ar[r]^-{\bar\Sigma^*{\dd}\bar\Pi_*}&
    U^* {}{\dd} {}V_*\ar[d]\\
    U^*L^* {}R_*V_*\ar[r]^-{\bar\Sigma^*\bar\Pi_*}&
    U^* {}V_*
    \enddiagram
  \end{equation}
  This morphism is a retraction in $\Cat\text-\mathbf{Mod}(\mathsf{R}\text-\mathrm{Alg}_s,\mathsf{L}\text-\mathrm{Coalg}_s)^\two$, since $\bar\Sigma^*$ and
  $\bar\Pi_*$ are retractions with respective sections $\Phi^*$ and $\Lambda_*$.
  Theorem~\ref{thm:2} implies that the object of $\Cat\text-\mathbf{Mod}(\mathsf{R}\text-\mathrm{Alg}_s,\mathsf{L}\text-\mathrm{Coalg}_s)^\two$ depicted
  by the leftmost vertical arrow in the diagram carries a structure of an
  $\mathsf{M}$-algebra.
  Hence the object on the right hand side, as a retract of an
  $\mathsf{M}$-algebra, carries an $(M,\Lambda^M)$-algebra
  structure that makes the retraction~\eqref{eq:61} a morphism of $(M,\Lambda^M)$-algebras. It
  remains to show that the section of
  $U^*\cdot{}{\dd}\cdot{}V_*\to U^*\cdot{}V_*$ so obtained is equal to that induced by the
  functorial factorisation, as described in Example~\ref{ex:3}, for which we
  appeal to Remark~\ref{rmk:18}. The
  induced section is
  \begin{equation}
    \label{eq:148}
    U^* V_*\xrightarrow{U^* \Phi^* \Lambda_* V_*}U^* L^* R_* V_*
    \to U^* L^* {\dd}_\K R_* V_*
    \xrightarrow{(1,s)^* {\dd} (p,1)_*} U^* {\dd} V_*
  \end{equation}
  where the middle morphism is the \textsc{kz} lifting operation for
  $LU$, $RV$. One can verify that the diagonal filler of a square $(h,k)\colon
  f\to g$, where $(f,s)$ is an $(L,\Phi)$-coalgebra and $(g,p)$ and
  $(R,\Lambda)$-algebra, is $p\cdot{}d\cdot{}s$ where $d$ is the diagonal filler of
  $({L}g\cdot{}h,k\cdot{}{R}f)\colon Lf\to Rg$. But
  $d=K(h,k)$, so $p\cdot{}d\cdot{}s$ is precisely the diagonal filler induced by the
  functorial factorisation.
\end{proof}

\section{Algebraic KZ Injectivity}
\label{sec:freeness}
In previous sections we have visited the construction of the universal category
$\mathcal{A}^\pitchfork$ with lifting operations against a functor
$\mathcal{A}\to\C^\two$, and the fact that, for any \textsc{awfs}
$(\mathsf{L},\mathsf{R})$ on $\C$, each $\mathsf{R}$-algebra comes equipped with
a lifting operation against $\mathsf{L}$-coalgebras; in other words, the
existence of a functor
$\mathsf{R}\text-\mathrm{Alg}\to\mathsf{L}\text-\mathrm{Coalg}^\pitchfork$. In
this section we concentrate in the analogous constructions adapted to the case
of lax orthogonal \textsc{awfs}s, where \textsc{kz} lifting operations will play
an important role.

The reader would
recall from Section~\ref{sec:univ-categ-with}, and originally
from~\cite{MR2506256}, the definition of the free category with a lifting operation
$U^\pitchfork{}\colon\mathcal{A}^\pitchfork{}\to\mathcal{C}^\two$ for
$U\colon\mathcal{A}\to\mathcal{C}^\two$. If $U\colon\mathscr A\to\K^\two$ is a
2-functor instead, $\mathscr{A}^\pitchfork{}$ has objects $(g,\phi)$ where
$\phi$ is a section of the morphism $U^*\cdot{}{\dd}\cdot{}g_*\to U^*\cdot{}g_*$ in the
2-category $\Cat\text-\mathbf{Mod}(\mathbf{1},\mathscr{A})$, which is isomorphic
to $[\mathscr{A}^{\mathrm{op}},\Cat]$. Morphisms $(g,\phi)\to(g',\phi')$ are
those morphisms $(u,v)\colon g\to g'$ in $\K^\two$ that are compatible with
the sections, while 2-cells $(u,v)\Rightarrow(\bar u,\bar v)$ are pairs of
2-cells $\alpha\colon u\to \bar u$ and $\beta\colon v\to \bar v$ in $\K$
such that the equality below holds --~we omit the dots that denote composition to
save space.
\begin{equation}
  \label{eq:137}
  \xymatrixcolsep{.6cm}
  \diagram
  U^*{}g_*\ar[r]^-{\phi}&
  U^*{}{\dd}{}g_*
  \rrtwocell^{U^*{}{\dd}(u,v)_*}_{U^*{}{\dd}(\bar u,\bar v)_*} &&
  U^*{}{\dd}{}g'_*
  \enddiagram
  =
  \diagram
  U^*{}g_*\rtwocell^{U^*{}(u,v)_*}_{U^*{}(\bar u,\bar v)_*}&
  U^*{}g_*'\ar[r]^-{\phi'}&
  U^*{}{\dd}{}g'_*
  \enddiagram
\end{equation}
In more elementary terms, $\alpha\cdot{}\phi(a,h,k)=\phi'(a,\alpha\cdot{}h,\beta\cdot{}k)$, for
each $a\in \mathscr A$ and each square $(h,k)\colon Ua\to g$. The 2-functor
$U^\pitchfork{}\colon \mathscr A^\pitchfork{}\to\K^\two$ is the obvious one,
analogous to the case of ordinary categories.

Next we introduce a different construction, the universal 2-category with a
\textsc{kz} lifting operation.
\begin{df}
  Given a 2-functor $U\colon \mathscr A\to\K^\two$ define another
  $U^\lperp\colon\mathscr A^\lperp\to\K^\two$ in the following manner.
  \begin{itemize}[labelindent=*,leftmargin=*]
  \item Its objects are morphisms $g\in\K^\two$ that are \emph{algebraically}
    \textsl{\textsc{kz}} \emph{injective to $U$}, by which we mean that they are equipped with
    a \textsc{kz}~lifting operation for the 2-functors $U$,
    $g\colon\mathbf{1}\to\K^\two$; ie a left adjoint coretract to the morphism
    of \Cat-modules
    $U^*\cdot{}{\dd}\cdot{}g_*\to U^*\cdot{}g_*$. Hence, an object of
    $\mathscr{A}^\lperp$ is an object of $\mathscr A^\pitchfork{}$ equipped with
    the extra structure of a coretract adjunction.
  \item A morphism $g\to g'$ in $\mathscr{A}^\lperp$ is a
    morphism $(h,k)$ in $\K^\two$ such that in the diagram below not only the
    square formed with the right adjoints commutes --~this always holds~-- but
    moreover the diagram represents a morphism of adjunctions; ie the square
    formed by the vertical arrows and the horizontal left adjoints commutes, and
    the vertical morphisms are compatible with the counits.
    \begin{equation}
      \label{eq:18}
      \xymatrixcolsep{1.7cm}
      \diagram
      U^*\cdot{}{\dd}\cdot{}g_*\ar@<-5pt>[r]\ar@{}[r]|-\bot\ar@<5pt>@{<-}[r]
      \ar[d]_{U^*\cdot{}{\dd}\cdot{}(h,k)_*}&
      U^*\cdot{}g_*\ar[d]^{U^*\cdot{}(h,k)_*}\\
      U^*\cdot{}{\dd}\cdot{}g'_*\ar@<-5pt>[r]\ar@{}[r]|-\bot\ar@<5pt>@{<-}[r]&
      U^*\cdot{}g'_*\\
      \enddiagram
    \end{equation}
  \item The 2-cells in $\mathscr A^\lperp$ are those of $\K^\two$. Observe that any
    such 2-cell is automatically compatible with the left adjoints
    in~\eqref{eq:18} --~by Proposition~\ref{prop:19}~\eqref{item:20}.  There are
    obvious forgetful 2-functors $\mathscr A^\lperp\to\mathscr A^\pitchfork$ and
    $\mathscr A^\lperp\to\K^\two$, the first of which is locally fully faithful.
  \end{itemize}

  Dually, given a 2-functor $V\colon\mathscr B\to\K^\two$ define
  $\prescript{\lperpleft}{}{V}\colon \prescript{\lperpleft}{}{\mathscr
    B}\to\K^\two$ by $\prescript{\lperpleft}{}{\mathscr B}=(\mathscr
  B^{\mathrm{op}})^\lperp$ and
  $\prescript{\lperpleft}{}{V}=(V^\mathrm{op})^\lperp$. Here we use the obvious
  isomorphism $(\K^\two)^{\mathrm{op}}\cong (\K^\mathrm{op})^\two$. More
  explicitly, objects of $\prescript{\lperpleft}{}{\mathscr B}$ are
  $f\in\K^\two$ equipped with a \textsc{kz} lifting operation for the 2-functors
  $f\colon \mathbf{1}\to\K^\two$, $V$.
  \label{df:6}
\end{df}
\begin{rmk}
\label{rmk:22}
There is a concise way of describing $\mathscr A^\lperp$, for a small 2-category
$\mathscr{A}$ over $\K^\two$. Let $\mathsf{M}$ be
the 2-monad on the 2-category $\mathscr P(\mathscr A)^\two$ whose algebras
are right adjoint retract morphisms in $\mathscr P(\mathscr A)=[\mathscr
A^{\mathrm{op}},\mathbf{CAT}]$. This 2-monad can be described by performing the
construction of the 2-monad of Section~\ref{sec:basic-example} starting from the
2-category $\mathscr P(\mathscr A)^{\mathrm{op}}$. More explicitly, if $\phi$ is
a morphism in $\mathscr P(\mathscr A)$, then $M\phi$ is the morphism with domain
the co-comma object depicted and whose composition with this co-comma object is
an identity 2-cell.
\begin{equation}
  \label{eq:156}
  \xymatrixrowsep{.4cm}
  \diagram
  \cdot\ar[d]_{\phi}\ar@{=}[r]\drtwocell<\omit>{^}&\cdot\ar[d]\ar@/^/[ddr]^\phi&\\
  \cdot\ar[r]\ar@/_/[drr]_1&\cdot\ar@{..>}[dr]|{M\phi}&\\
  &&\cdot
  \enddiagram
\end{equation}
% Yet another description of $\mathsf{M}$ can be given by saying that the
% component $(M\phi)_a$, which is a functor, has domain the Grothendieck
% construction of the indexed category $\two\to\mathbf{CAT}$ that picks out the
% functor $\phi_a$. Then $(M\phi)_a$ is the projection of this
The \Cat-module morphism
${\dd}\longrightarrow 1_{\K^\two}$ can be equivalently described as a 2-functor
\begin{equation}
  \label{eq:195}
  E\colon \K^\two\longrightarrow \mathscr P(\K^\two)^\two
\end{equation}
that sends $g\in\K^\two$ to ${\dd}(-,g)\to\K^\two(-,g)$. Then
$\mathscr A^\lperp$ is the pullback of the 2-category of $\mathsf{M}$-algebras
along $\mathscr P(U^*)^\two E$.
\begin{equation}
  \label{eq:203}
  \diagram
  \mathscr A^\lperp\ar[d]\ar[rr]&&
  \mathsf{M}\text-\mathrm{Alg}_s\ar[d]\\
  \K^\two\ar[r]^-E&
  \mathscr P(\K^\two)^\two\ar[r]^-{\mathscr P(U^*)^\two}&
  \mathscr P(\mathscr A)^\two
  \enddiagram
\end{equation}
\end{rmk}

\begin{rmk}
\label{rmk:23}
One can express the compatibility of the morphism $(h,k)$ in $\mathscr A^\lperp$
with the counits required in Definition~\ref{df:6}
in terms of diagonal fillers. Given a diagonal filler $j$ as on
the left hand side below, the counit provides for a 2-cell
$\varepsilon_j\colon\phi(a,u,v)\Rightarrow j$. The compatibility means that
$h\cdot{}\varepsilon_j=\varepsilon_{h\cdot{}j}$.
\begin{equation}
  \label{eq:19}
  \diagram
  \cdot\ar[d]_{Ua}\ar[r]^u&
  \cdot\ar[d]^g\\
  \cdot\ar[ur]^{j}\ar[r]_-v&
  \cdot
  \enddiagram
  \qquad
  \diagram
  \cdot\ar[r]^u\ar[d]_{Ua}&
  \cdot\ar[d]^g\ar[r]^h&
  \cdot\ar[d]^{g'}\\
  \cdot\urtwocell^\phi_j{}\ar[r]_v&
  \cdot\ar[r]_k&
  \cdot
  \enddiagram
  =
  \xymatrixcolsep{1.5cm}
  \diagram
  \cdot\ar[r]^{h\cdot{}u}\ar[d]_{Ua}&
  \cdot\ar[d]^{g'}\\
  \cdot\urtwocell_{\hole h\cdot{}j}\ar[r]_{k\cdot{}v}&
  \cdot
  \enddiagram
\end{equation}
These constructions are functorial, in the sense that if $F\colon(\mathscr
A,U)\to(\mathscr B,V)$ is a 2-functor over $\K^\two$, there is another 2-functor
$F^\lperp\colon(\mathscr B^\lperp,V^\lperp)\to(\mathscr A^\lperp,U^\lperp)$,
which sends $g\in\K^\two$ equipped with a \textsc{kz} lifting operation for $V,g$
to the induced choice for $U=VF,g$. A 2-functor $\prescript{\lperpleft}{}{F}$
can be similarly defined.
\end{rmk}
\begin{rmk}
  Given $V\colon\mathscr B\to\K^\two$, there is an isomorphism of categories
  between 2-functors $\mathscr B\to\mathscr A^\lperp$ over $\K^\two$ and
  \textsc{kz} lifting operations for the pair of 2-functors $U,V$. Similarly, there is
  an isomorphism of categories between 2-functors $\mathscr
  A\to\prescript{\lperpleft{}}{}{\mathscr B}$ over $\K^\two$ and \textsc{kz} lifting
  operations for the pair of 2-functors $U,V$. We hence have a natural
  isomorphism of \emph{sets}
  \begin{equation}
    \label{eq:165}
    \mathbf 2\text-\Cat/\K^\two((\mathscr A,U),(\mathscr B^\lperp,V^\lperp)) \cong
    \mathbf 2\text-\Cat/\K^\two((\mathscr B,V),(\prescript{\lperpleft}{}{\mathscr A},\prescript{\lperpleft}{}{U}))
  \end{equation}
  so there is an adjunction % $(-)^\lperp\colon \bigl(\mathbf 2\text-\Cat/\K^\two\bigr)^{\mathrm{op}}
  % \longrightarrow \mathbf 2\text-\Cat/\K^\two$ is a right adjoint of
  % $\prescript{\lperpleft}{}{(-)}\colon \mathbf 2\text-\Cat/\K^\two
  % \longrightarrow \bigl(\mathbf 2\text-\Cat/\K^\two\bigr)^{\mathrm{op}}$.
  \begin{equation}
    \xymatrixcolsep{2.5cm}
    \diagram
    \bigl(\mathbf 2\text-\Cat/\K^\two\bigr)^{\mathrm{op}}
    \ar@<7pt>[r]^-{(-)^\lperp}
    \ar@{}[r]|{\top}
    \ar@{<-}@<-7pt>[r]_-{\prescript{\lperpleft}{}{(-)}}
    &
    \mathbf 2\text-\Cat/\K^\two
    \enddiagram
  \end{equation}

  The unit and counit of this adjunction -- or rather, both units -- are
  2-functors $N_U\colon \mathscr A\to\prescript{\lperpleft}{}{(\mathscr
    A^\lperp)}$ and $M_U\colon\mathscr A\to(\prescript{\lperpleft}{}{\mathscr
    A})^\lperp$ commuting with the functors into $\K^\two$. The first one
  corresponds to the tautological \textsc{kz} lifting operation for the pair of
  2-functors $U,U^\lperp$, and the second one to the tautological \textsc{kz} lifting
  operation for $\prescript{\lperpleft}{}{U},U$.\label{rmk:24}
\end{rmk}

\begin{ex}
  \label{ex:4}
  In the case when $U$ is the 2-functor $f\colon \mathbf{1}\to\K^\two$ that
  picks out a morphism $f$, the objects of the 2-category $f^\lperp$ are
  morphisms \emph{algebraically} \textsl{\textsc{kz}} \emph{injective with respect to $f$}. This is a
  slight abuse of language, as a morphism can be algebraically \textsc{kz} injective to
  $f$ in more than one way --~but two such are, of course, isomorphic.
\end{ex}
\begin{lemma}
  \label{l:13}
Given a 2-functor $U\colon\mathscr A\to\K^\two$, a 2-adjunction $U\dashv G$ and
$g\in\K^\two$, there is an isomorphism of 2-categories over $\K^\two$ between
$\mathscr A^\lperp$ and the 2-category described by:
\begin{itemize}
\item Objects are coretract adjunctions $\ell_g\dashv G(1,g) \colon
  G(1_{\dom(g)})\to Gg$ in $\mathscr A$.
\item Morphisms from $\ell_g\dashv G(1,g)$ to $\ell_{\bar g}\dashv G(1,\bar
  g)$ are morphisms $(h,k)\colon g\to\bar g$ in $\K^\two$ such that $G(h,k)$
  defines a morphism of adjunctions: $G(h,k)\cdot{}\ell_g=\ell_{\bar g}\cdot{}G(h,k)$ and
  $G(h,k)$ commutes with the counits.
\item 2-cells $(h,k)\Rightarrow(\bar h,\bar k)$ are 2-cells in $\K^\two$,
  with no additional conditions.
\end{itemize}
\end{lemma}
\begin{proof}
  By Proposition \ref{prop:11} there is a bijection between objects of $\mathscr
  A^\lperp$ and coretract adjunctions as in the statement. The description of
  the morphisms and 2-cells is a direct translation from the ones of $\mathscr
  A^\lperp$ --~Definition~\ref{df:6}.
\end{proof}
\begin{lemma}
  \label{l:14}
  Assume the conditions of Lemma \ref{l:13}. Then, for any full sub-2-category
  $\mathscr F\subset\mathscr A$ containing the full image of $G$, the functor
  $\mathscr A^\lperp\to\mathscr F^\lperp$ induced by the inclusion is an
  isomorphism.
\end{lemma}
\begin{proof}
  If we denote by $J\colon \mathscr F\hookrightarrow\mathscr A$ the inclusion
  and $H=JG\colon \mathscr A\to\mathscr F$ the right adjoint of $UJ$,
  Lemma~\ref{l:13} allows us to describe $\mathscr F^\lperp$ as the 2-category
  with objects coretract adjunctions $\ell_g\dashv H(1,g)\colon
  H(1_{\dom(g)})\to Hg$ in $\mathscr F$. But to give this retract adjunction in
  $\mathscr F$ is equivalent to giving a retract adjunction $\ell_g\dashv G(1,g)$
  in $\mathscr A$. The rest of the proof is similarly easy.
\end{proof}
\begin{cor}
  \label{cor:1}
  If $(\mathsf{L},\mathsf{R})$ is a lax orthogonal
  {\normalfont\textsl{\textsc{awfs}}} on \K, there exists a 2-functor
  \begin{equation}
    \label{eq:28}
    \mathsf{R}\text-\mathrm{Alg}_s\longrightarrow \mathsf{L}\text-\mathrm{Coalg}_s^\lperp
    \text{ over }
    \mathsf{L}\text-\mathrm{Coalg}_s.
  \end{equation}
\end{cor}
\begin{proof}
Proposition~\ref{prop:10} together with Lemma~\ref{l:13} imply that, in order to
define the 2-functor on objects, we may send an
$\mathsf{R}$-algebra $(p,1)\colon Rg\to g$ to a corectract adjunction
$\ell\dashv L(1,g)$ in
$\mathsf{L}\text-\mathrm{Coalg}_s$. The adjunction is
$L(1,p)\cdot{}\Sigma_g\dashv L(1,g)$, which is the
composition of the adjunctions $\Sigma_g\dashv L\Phi_g$ and $L(1,p)\dashv
L(1,{L}g)$. On morphisms and 2-cells, the 2-functor is defined by the identity.
\end{proof}
\begin{thm}
  \label{thm:5}
  The following are equivalent for an {\normalfont\textsl{\textsc{awfs}}} $(\mathsf{L},\mathsf{R})$ on a 2-category.
  \begin{enumerate}
  \item \label{item:57} $(\mathsf{L},\mathsf{R})$ is a lax orthogonal {\normalfont\textsl{\textsc{awfs}}}.
  \item \label{item:3} There is a \slkz{} lifting operation
    for the forgetful
    2-functors from $\mathsf{L}$-coalgebras and from $\mathsf{R}$-algebras.
  \item \label{item:58} There is a 2-functor
    $\mathsf{R}\text-\mathrm{Alg}_s\to{}
    \mathsf{L}\text-\mathrm{Coalg}_s^\lperp$ making~\eqref{eq:138} commutative.
    Furthermore, this 2-functor is essentially unique.
  \item \label{item:53} There is a 2-functor
    $\mathsf{R}\text-\mathrm{Alg}_s\to{}\mathscr{F}^\lperp $ making the outer
    diagram in \eqref{eq:138} commutative, for any full sub-2-category $\mathscr
    F\subset \mathsf{L}\text-\mathrm{Coalg}_s$ containing the cofree
    $\mathsf{L}$-coalgebras. Furthermore, this 2-functor is essentially unique.
  \end{enumerate}
  \begin{equation}
    \label{eq:138}
    \xymatrixrowsep{.4cm}
    \diagram
    \mathsf{R}\text-\mathrm{Alg}_s\ar@{-->}[r]\ar[dr]&
    \mathsf{L}\text-\mathrm{Coalg}_s^\lperp\ar[d]\ar[r]&
    \mathscr F^\lperp\ar[d]\\
    &
    \mathsf{L}\text-\mathrm{Coalg}_s^\pitchfork\ar[r]&
    \mathscr F^\pitchfork
    \enddiagram
  \end{equation}
\end{thm}
\begin{proof}
  There is a bijection between structures in (\ref{item:3}) and those
  in~\eqref{item:58}, by definition of $\mathscr A^\lperp$, in which case both
  are essentially unique since \textsc{kz} lifting operations are unique up to
  isomorphism --~Lemma~\ref{l:3}. The equivalence of~\eqref{item:58}
  and~\eqref{item:53} follows from Lemma~\ref{l:14}, while that
  of~\eqref{item:57} and~\eqref{item:58} was already explained above.

  We now proceed to prove \eqref{item:58}$\Rightarrow$\eqref{item:57}. As it has
  been our convention, we will denote by \K\ the base 2-category, and by $U$ and $V$ the forgetful 2-functors
  from the 2-categories of $\mathsf{L}$-coalgebras and $\mathsf{R}$-algebras,
  respectively.

  Let
  $(g,p)$ be an $\mathsf{R}$-algebra. Its image in
  $\mathsf{L}\text-\mathrm{Coalg}_s^\pitchfork$ can be given as in
  Corollary~\ref{cor:5}, again by $(g,p)$. By hypothesis, $(g,p)$ carries a
  structure of
  an object of $\mathsf{L}\text-\mathrm{Coalg}_s^\lperp$.
  By definition $p\cdot{}{L}g=1$ and
  $g\cdot{}p={R}g$. Consider the diagonal
  \begin{equation}
    \label{eq:140}
    \xymatrixrowsep{.6cm}
    \xymatrixcolsep{1.4cm}
    \diagram
    \cdot
    \ar[r]^-{{L}g}\ar[d]_{{L}g}
    &
    \cdot
    \ar[d]^{{R}g}
    \\
    \cdot
    \ar[r]_-{{R}g}\ar[ur]|-{{L}g\cdot{}p}
    &
    \cdot
    \enddiagram
  \end{equation}
  and note that $Rg$ is an object of
  $\mathsf{L}\text-\mathrm{Coalg}_s^\lperp$, and that the chosen diagonal filler
  of the outer square is the identity morphism. It follows the existence of a
  unique 2-cell $\eta\colon 1\Rightarrow {L}g\cdot{}p$ such that
  $\eta\cdot{}{L}g=1$ and ${R}g\cdot{}\eta=1$. The first of these two equalities is one
  of the triangle identities required to obtain a retract adjunction
  $p\dashv {L}g$. The second of these equalities tells us that, if we
  can prove the other triangle identity, we obtain not only an adjunction in
  \K\ but also a retract adjunction $(p,1)\dashv\Lambda_g$ in $\K^\two$.

  We now show that $p\cdot{}\eta=1$. Consider the pasting below.
  \begin{equation}
    \label{eq:141}
    \xymatrixcolsep{1.4cm}
    \xymatrixrowsep{0.2cm}
    \diagram
    \cdot\ar@{=}[rrr]\ar[ddd]_{{L}g}&&&
    \cdot\ar[ddd]^g\\
    &&\cdot\ar[ur]_{p}&\\
    &\cdot\ar[ur]_{{L}g}&&\\
    \cdot\ar[ur]^{p}\uurruppertwocell^1{\eta}\ar[rrr]_{{R}g}&&&
    \cdot
    \enddiagram
  \end{equation}
  The chosen diagonal filler of the outer diagram is $p$, and $p\cdot{}\eta$ is an
  endo-2-cell of $p$. In addition, $g\cdot{}p\cdot{}\eta={R}g\cdot{}\eta=1$ and
  $p\cdot{}\eta\cdot{}{L}g=1$. By the universal property of \textsc{kz} lifting operations
  spelled out immediately after Definition~\ref{df:8}, it must be $p\cdot{}\eta=1$.
  This finishes the proof that $\mathsf{R}$-algebra structures are left adjoint
  retracts to the components of the unit of $\mathsf{R}$, ie that
  $\mathsf{R}$ is lax idempotent.

  One can show that $\mathsf{L}$ is lax idempotent either by appealing to
  Theorem~\ref{thm:11} or by a duality argument. By taking opposite
  2-categories, and taking into account the isomorphism
  $(\K^{\mathrm{op}})^\two\cong(\K^\two)^{\mathrm{op}}$, the 2-functor
  $\mathsf{L}\text-\mathrm{Coalg}_s\to\prescript{\lperpleft}{}{\mathsf{R}\text-\mathrm{Alg}_s}$,
  which exists by Remark~\ref{rmk:24}, transforms into a 2-functor
  $\mathsf{L}^{\mathrm{op}}\text-\mathrm{Alg}_s\to\mathsf{R}^{\mathrm{op}}\text-\mathrm{Coalg}_s^\lperp$
  that commutes with the 2-functors into
  $\mathsf{R}^{\mathrm{op}}\text-\mathrm{Coalg}_s^\pitchfork$. By the proof
  above we know that $\mathsf{L}^{\mathrm{op}}$ is a lax idempotent 2-monad on
  $(\K^\two)^{\mathrm{op}}$, which is to say that $\mathsf{L}$ is a lax
  idempotent 2-comonad.
\end{proof}

  Theorem \ref{thm:5} has a dual statement of the following form: an \textsc{awfs}
  $(\mathsf{L},\mathsf{R})$ is lax orthogonal if and only if there exists an --
  essentially unique -- 2-functor
  $\mathsf{L}\text-\mathrm{Coalg}_s\to\prescript{\lperpleft}{}{\mathsf{R}\text-\mathrm{Alg}_s}$
  commuting with the respective forgetful functors into
  $\prescript{\pitchfork}{}{\mathsf{R}\text-\mathrm{Alg}_s}$.

\begin{rmk}
  \label{rmk:16}
  For a lax orthogonal \textsc{awfs} $(\mathsf{L},\mathsf{R})$, objects of
  $\mathsf{L}\text-\mathrm{Coalg}_s^\lperp$ are in bijection with normal
  pseudo-$\mathsf{R}$-algebras. Indeed, the proof of Theorem~\ref{thm:5} shows
  that they are in bijection with retract adjunctions $(p,1)\dashv\Lambda_g$ in
  $\K^\two$, which are precisely normal pseudo-$\mathsf{R}$-algebras
\end{rmk}

\section{Simple 2-adjunctions and lax idempotent 2-monads}
\label{sec:extens-transf-a.f}
This section introduces the notion of simple 2-adjunction, which can be thought as
a lax version of that of simple reflection studied in~\cite{MR779198}.

In the same way that one can define a strict monoidal category as a category
with a bifunctor $(-\otimes-)$ that is associative and has a unit object, we may
define a \emph{strict monoidal 2-category} as a 2-category $\mathscr A$ with a
2-functor $\otimes\colon \mathscr A\times\mathscr A\to\mathscr A$ that is
associative and has a unit object $I$. A monoid in $\mathscr A$ is a monoid in
its underlying strict monoidal category, ie an object $T$ with a multiplication
and unit morphisms that satisfy the usual monoid axioms. The main example for us
will be $\mathscr A=\mathrm{End}(\mathscr B)$, the endo-2-morphisms of a
2-category $\mathscr B$, where a monoid is a 2-monad.

\begin{df}
  \label{df:888}
A \emph{lax idempotent monoid} in a strict monoidal 2-category $\mathscr A$ is a monoid
$j\colon I\to T\leftarrow T\otimes T\colon m$ that satisfies conditions
analogous to those of Definition~\ref{df:12} numerals \ref{item:6}, \ref{item:8}
and \ref{item:10}. These are, in turn,
\begin{itemize}[labelindent=\parindent,leftmargin=*]
\item $T\otimes j\dashv m$ with identity unit;
\item $m\dashv j\otimes T$ with identity counit;
\item there is a 2-cell $\delta\colon T\otimes j\Rightarrow j\otimes T\colon
  T\to T\otimes T$ that satisfies $\delta \cdot j=1$ and $m\cdot\delta =1$.
\end{itemize}
\end{df}

We can now make our first statement of the section. The reader would have noticed
that the monoidal 2-categories need not be strict in order for the results to
hold, but we keep the strictness hypothesis for simplicity.
\begin{lemma}
  \label{l:9}
  Let $\mathscr A$ be a monoidal 2-category and $\mathscr C\subseteq\mathscr A$
  a coreflective 2-category, closed under the monoidal structure, and $(T,i,m)$
  a monoid in $\mathscr A$. If $\alpha\colon S\to T$ is the coreflection of $T$
  into $\mathscr C$, then $S$ carries a structure of a monoid $(S,j,n)$ making
  $\alpha$ a monoid morphism. Assume further that $\alpha\otimes S\colon
  S\otimes S\to T\otimes S$ is the coreflection of $T\otimes S$. Then $S$ is lax
  idempotent if there exists a coretract adjunction
  \begin{equation}
    \label{eq:105}
    \big( T\xrightarrow{T\otimes j}T\otimes S\big )\dashv
    \big( T\otimes S\xrightarrow{T\otimes\alpha}T\otimes T\xrightarrow{m} T\big).
  \end{equation}
\end{lemma}
\begin{proof}
  The unit $j\colon I\to S$ and multiplication $n\colon S\otimes S\to S$ are
  defined by $\alpha\cdot{}j=i$ and $\alpha\cdot{}n=m\cdot{}(\alpha\otimes \alpha)$.
  We shall define a 2-cell $\delta\colon S\otimes j\Rightarrow j\otimes S\colon
  S\to S\otimes S$. From the fact that $\alpha\otimes S$ is a coreflection, it
  follows that to give $\delta$ is equally well to give a 2-cell
  $\delta'\colon(T\otimes j)\cdot{}\alpha\Rightarrow i\otimes S$, and by the
  adjunction \eqref{eq:105}, to give a 2-cell $\delta''\colon\alpha\Rightarrow
  m\cdot{}(T\otimes \alpha)\cdot{}(i\otimes S)$, which we choose to be the identity.

  The axiom $\delta\cdot{}j=1$ of a lax idempotent monoid follows from the triangle
  identity $\varepsilon\cdot{}(T\otimes j)=1$, where $\varepsilon$ is the counit of
  \eqref{eq:105}: we show that $\delta'\cdot{}j=1$ below.
  \begin{equation}
    \label{eq:53}
    \delta'\cdot{}j=
    ((T\otimes j)\cdot{}m\cdot{}(T\otimes\alpha)\cdot{}\delta'\cdot{}j)(\varepsilon\cdot{}(j\otimes S)\cdot{}j)=
    \varepsilon\cdot{}(T\otimes j)\cdot{}i=1.
  \end{equation}

  It only rests to verify the axiom $n\cdot{}\delta=1$. By the coreflection $\alpha$,
  we have to show
  $1=\alpha\cdot{}n\cdot{}\delta=m\cdot{}(\alpha\otimes\alpha)\cdot{}\delta=m\cdot{}(T\otimes\alpha)\cdot{}\delta'=\delta''$,
  which holds by our choice of $\delta''$.
\end{proof}
Before continuing, it is convenient to introduce some notation. Each endo-2-functor
$S$ of $\K^\two$ corresponds under the isomorphism
$\mathrm{End}(\K^\two)=[\K^\two,\K^\two]\cong[\K^\two,\K]^\two$ to a pair of 2-functors $S_0$,
$S_1\colon\C^\two\to\C$ with a 2-natural transformation $S_0\Rightarrow S_1$. We
denote the component of this natural transformation at $f$ by $Sf\colon S_0f\to
S_1f$. A morphism $S\to T$ in $\mathrm{End}(\K^\two)$ corresponds to a pair of
2-natural transformations $S_0\Rightarrow T_0$ and $S_1\Rightarrow T_1$,
compatible with $S_0\Rightarrow S_1$ and $T_0\Rightarrow T_1$.

A version for categories and functors, as opposite to 2-categories and
2-functors, of the following lemma is contained in~\cite[Prop~4.7]{MR2506256}.
\begin{lemma}
  \label{l:5}
    If \K\ has pushouts, then the category of
    codomain-preserving pointed endo-2-functors $1\backslash
    \mathrm{End}_{\cod}(\K^{\mathbf{2}})$ is a coreflective sub-2-category of the
    2-category of pointed endofunctors $1\backslash
    \mathrm{End}(\K^{\mathbf{2}})$. The coreflection of a 2-monad has a canonical
    structure of a codomain-preserving 2-monad that makes the coreflection counit a monad morphism.
\end{lemma}
  Given a pointed endo-2-functor $(T,\Theta)$, its codomain-preserving coreflection
  $(R,\Lambda)$ is given by the
  following pullback square, while the point $\Lambda\colon 1\Rightarrow R$ is
  induced by the universal property.
  The natural transformation $R\Rightarrow T$ with components given by the
  pullback square is
  the counit of the coreflection.
  \begin{equation}
    \label{eq:80}
    \xymatrixrowsep{.4cm}
    \xymatrixcolsep{1.7cm}
    \diagram
    A\ar@/_/[ddr]_f\ar@/^/[drr]^{\Theta_{0f}}\ar[dr]|{\Lambda_{0f}}&&
    \\
    &
    R_0f\ar@{}[dr]|{\mathrm{pb}}\ar[d]^{Rf}\ar[r]_{}&
    T_0f\ar[d]^{Tf}\\
    &
    B\ar[r]_{\Theta_{1f}}&T_1f
    \enddiagram
  \end{equation}

\begin{rmk}
  \label{rmk:25}
For future reference, we state that the coreflection $R\Rightarrow T$ of a monad $T$ on $\mathcal{C}^\two$ into a
  codomain-preserving monad $R$ is a monad morphism.
\end{rmk}
\begin{df}
  \label{df:9}
  Suppose given the following data.
  \begin{itemize}
  \item A 2-adjunction $F\dashv U\colon \mathscr A\to\K$, whose counit we
    denote by $e\colon FU\Rightarrow 1$.
  \item A 2-monad $\mathsf{P}$ on $\mathscr
    A^\two$ with multiplication $m\colon P^2\Rightarrow P$.
  \item The codomain-preserving coreflection of the 2-monad $U^\two
    \mathsf{P}F^\two$, that we denote by $\alpha\colon\mathsf{S}\to U^\two
    \mathsf{P}F^\two$, and whose unit we denote $j\colon 1\Rightarrow S$.
  \end{itemize}
The 2-adjunction is said to be \emph{simple} with respect
  to $\mathsf{P}$ if there is a coretract adjunction in the 2-category
  $[\K^\two,\mathscr A^\two]$, with components at $f\in \K^\two$
  \begin{equation}
    \label{eq:106}
    \big(PFf\xrightarrow{PF^\two j_f}PFSf\big)
    \dashv
    \big(PFSf\xrightarrow{PF^\two\alpha_f}PFUPFf\xrightarrow{Pe PF^\two f}PPFf
    \xrightarrow{mFf} PFf\big).
  \end{equation}
\end{df}
\begin{lemma}
  \label{l:10}
  Given a simple 2-adjunction as in Definition \ref{df:9}, the
  codomain-preserving reflection $\mathsf{S}$ is a lax idempotent 2-monad.
\end{lemma}
\begin{proof}
  Let us denote by $\mathsf{T}$ the 2-monad $U^\two PF^\two$. By the
  construction of the coreflection $\mathsf{S}$ as a pullback, it is clear that $\alpha T
  \colon SS\to TS$ is the coreflection of $TS$. Lemma \ref{l:9} tells us
  that $\mathsf{S}$ will be lax idempotent if we have a coretract adjunction
  in $[\K^\two,\K^\two]$
  \begin{equation}
    \label{eq:107}
    \big(T \xrightarrow{Tj}TS\big)\dashv
    \big(TS\xrightarrow{T\alpha}TT\longrightarrow{}T\big).
  \end{equation}
  Such an adjunction is obtained from the one of Definition~\ref{df:9} by
  applying $U^\two$.
\end{proof}
Let us now make an observation that shall be needed later on.
\begin{rmk}
  \label{rmk:12}
  Given $F\dashv U$, $\mathsf{P}$ and the codomain-preserving coreflection
  $\alpha\colon S\to U^\two PF^\two$ as in Lemma~\ref{l:10}, we claim that the
  composition of the multiplication with the counit $\alpha$
  \begin{equation}
    \label{eq:151}
    SS\longrightarrow S\xrightarrow{\alpha}U^\two PF^\two
  \end{equation}
  factors through $U^\two mF^\two$, where $m$ is the multiplication of
  $\mathsf{P}$. Indeed, since $\alpha$ is a 2-monad
  (strict) morphism, we know that
  \begin{equation}
    \label{eq:111}
    (SS\to S\xrightarrow{\alpha}U^\two PF^\two) =
    (SS\xrightarrow{\alpha\alpha}U^\two PF^\two U^\two PF^\two
    \to U^\two PP F^\two \to
    U^\two PF^\two).
  \end{equation}
\end{rmk}

Below we describe Definition \ref{df:9} in a particular case of interest, but
before let us recall a few facts about the Kleisli construction for the free
split opfibration 2-monad $\mathsf{R}'$ on $\mathscr A^\two$, for a 2-category
$\mathscr A$ with lax limits of morphisms. This Kleisli construction can be
described as the inclusion 2-functor of $\mathscr A^\two$ into the 2-category
$\mathrm{Lax}[\two,\mathscr A]$ of 2-functors from $\two$ to $\mathscr A$ and
lax transformations between them. Morphisms between
free $\mathsf{R}'$-algebras are in bijection with morphisms in
$\mathrm{Lax}[\two,\mathscr A]$, and the bijection is given as displayed below,
a fact we shall soon employ.
\begin{equation}
  \label{eq:3}
  \xymatrixrowsep{.6cm}
  \diagram
  \cdot\ar[d]_{R'f}\ar[r]^h&
  \cdot\ar[d]^{R'g}\\
  \cdot\ar[r]^k&
  \cdot
  \enddiagram
  \qquad\longmapsto\qquad
  \diagram
  \cdot\ar[r]^{{L}f}\ar[d]_f&
  \cdot\ar[d]_{R'f}\ar[r]^h&
  \cdot\ar[d]_{R'g}\ar[r]^{q_g}\drtwocell<\omit>{\nu}&
  \cdot\ar[d]^g\\
  \cdot\ar@{=}[r]&
  \cdot\ar[r]^k&
  \cdot\ar@{=}[r]&
  \cdot
  \enddiagram
\end{equation}
  %% Description of simple for coreflection--opfib.
\begin{prop}
  \label{prop:16}
  Let $F\dashv U\colon\mathscr A\to\K$ be a 2-adjunction, where $\mathscr A$ has
  comma objects and \K\ has pullbacks, and $\mathsf{R}'$ be the free split
  opfibration 2-monad on $\mathscr A^\two$. Denote by $\mathsf{R}$ the
  codomain-preserving coreflection of $U^\two\mathsf{R}'F^\two$. The
  2-adjunction is simple with respect to the coreflection--opfibration {\normalfont\textsl{\textsc{awfs}}}
  precisely when there are coretract adjunctions
  $F{L}f\dashv q_{Ff}\cdot{}e_{K'Ff}\cdot{}F\tau_f$ 2-natural in $f$, where these morphisms
  are those defined in~\eqref{eq:110}.
\end{prop}
\begin{proof}
  In this proof we shall denote the unit and counit of $F\dashv U$ by $i$ and
  $e$, respectively, and the comparison adjoint of the Kleisli construction of a
  2-monad  $\mathsf{P}$ on $\mathscr A^\two$ by $C\colon
  \mathrm{Kl}(\mathsf{P})\to\mathsf{P}\text-\mathrm{Alg}_s$. In a moment we will
  use the well-known fact
  that $C$ is full and
  faithful and its replete image is the full subcategory of free algebras.
  The definition of simple adjunction consists of a coretract adjunction in
  $[\K^\two,\mathscr A^\two]$ between 2-natural transformations whose components
  are strict morphisms of $\mathsf{P}$-algebras between free
  $\mathsf{P}$-algebras in $\mathscr A^\two$. Therefore, the said coretract
  adjunction is the image of a coretract adjunction in
  $[\K^\two,\mathrm{Kl}(\mathsf{P})]$ under the 2-functor
  \begin{equation}
    \label{eq:55}
    [1,C]\colon[\K^\two,\mathrm{Kl}(\mathsf{P})]
    \longrightarrow [\K^\two,\mathsf{P}\text-\mathrm{Alg}_s].
  \end{equation}
  % When $U_{\mathsf{P}}$ is locally fully faithful, such a coretract adjunction
  % can be lifted to $[\K^\two,\mathrm{Kl}(\mathsf{P})]$;
  % this is possible if $\mathsf{P}$ is lax idempotent by Remark~\ref{rmk:14}.

  When $\mathsf{P}$ is the free split opfibration 2-monad $\mathsf{R}'$, its
  Kleisli construction is
  isomorphic to the inclusion of $\mathscr A^\two$ into
  $\mathrm{Lax}[\two,\mathscr A]$, by the comments before the present
  proposition. We can use the correspondence between morphisms of free
  $\mathsf{R}'$-algebras and morphisms in $\mathrm{Lax}[\two,\mathscr A]$
  described in~\eqref{eq:3} to deduce the form of the lifting
  to $[\K^\two,\mathrm{Lax}[\two,\mathscr A]]$ of the coretract adjunction in
  $[\K^\two,\mathscr A^\two]$ that
  exhibits $F\dashv U$ as a simple 2-adjunction. The lifting has
  component at $f\in\K^\two$ displayed below, where $\nu$ is the comma object
  that defines $R'$.
  \begin{equation}
    \label{eq:110}
    \diagram
    FA\ar[d]_{Ff}\ar[r]^-{F{L}f}&FKf\ar[d]_{F{R}f}\\
    FB\ar@{=}[r]&FB
    \enddiagram
    \dashv
    \diagram
    FKf\ar[d]^{F{R}f}\ar[r]^-{F\tau_f} &
    FUK'Ff\ar[d]_{FUR'{Ff}}\ar[r]^-{e} &
    K'Ff\ar[r]^{q_{Ff}}\ar[d]_{R'{Ff}}\drtwocell<\omit>{\nu}&
    FA\ar[d]^{Ff}
    \\
    FB\ar[r]^-{Fi_B} &
    FUFB\ar[r]^-{e} &
    FB\ar@{=}[r]&
    FB
    \enddiagram
  \end{equation}
  This adjunction consists of a coretract adjunction as in the statement of this
  proposition, plus the requirement that its counit, say $\alpha_f$, is a 2-cell
  in $\mathrm{Lax}[\mathbf{2},\mathscr{A}]$; ie
  \begin{equation}
    F{R}f\cdot{}\alpha_f=\nu\cdot{}e_{K'Ff}\cdot{}F\tau_f.\label{eq:5}
  \end{equation}
  Thus, the direct part of the statement is trivial.  To prove the converse, we
  will show that
  if $F{L}f$ has a right adjoint retract in \K\ as in the statement,
  then \eqref{eq:5} automatically holds.
  As a consequence of the adjunction, the 2-cell
  on the left hand side of \eqref{eq:5} is the unique 2-cell
  $\beta$, with the appropriate domain and codomain, such that $\beta\cdot{}F{L}f=1$. We must verify that the 2-cell on
  the right hand side of \eqref{eq:5} satisfies the same property.
  By definition of ${L}f$,
  \begin{equation}
    {e}_{K'Ff}\cdot{}F\tau_f\cdot{}F{L}f=
    {e}_{K'Ff}\cdot{}FUL'({Ff})\cdot{}Fi_A=
    L'({Ff})\cdot{}{e}_{FA}\cdot{}Fi_A=L'({Ff}),
    \label{eq:179}
  \end{equation}
  from where it is clear that
  $\nu\cdot{} {e}_{K'Ff}\cdot{}F\tau_f\cdot{}F{L}f=\nu\cdot{}L'({Ff})=1$, concluding the proof.
\end{proof}
Lemma~\ref{l:10} yields:
\begin{cor}
  \label{cor:6}
  If $F\dashv U\colon\mathscr A\to \K$ is a 2-adjunction simple with respect to
  the free split opfibration 2-monad $\mathsf{R}'$ on $\mathscr A^\two$, then
  the codomain-preserving coreflection $\mathsf{R}$ of
  $U^\two\mathsf{R}'\mathsf{F}^\two$ is lax idempotent.
\end{cor}

\begin{rmk}
  \label{rmk:26}
  There is a 2-monad morphism with components $(\tau_f,i_{\cod(f)})\colon Rf\to
  UR'Ff$, by Remark~\ref{rmk:25}. Taking the mate along $F\dashv U$, we obtain
  an opmorphism of 2-monads $(\hat\tau_f,1_{F\cod(f)})\colon FRf\to R'Ff$.
\end{rmk}
Recall from Lemma~\ref{cor:7} that the codomain functor is a fibration from
(the underlying category of) $\mathsf{R}\text-\mathrm{Alg}_s$ to
$\mathcal{C}$. In particular the category of split opfibrations in a 2-category
\K\ with lax limits of morphisms is a fibration over $\K$.
\begin{thm}
  \label{thm:15}
  Assume given a 2-adjunction $F\dashv U\colon \mathscr{A}\to\K$ between
  2-categories equipped with chosen lax limits of morphisms and
  pullbacks, strictly preserved by $U$. If the 2-monad  $\mathsf{R}$ is as in
  Proposition~\ref{prop:16}, then there is a canonical 2-functor
  into the category of split opfibrations in \K\ that commutes with the
  forgetful functors into $\K^\two$.
  \begin{equation}
    \label{eq:275}
    \mathsf{R}\text-\mathrm{Alg}_s\longrightarrow \mathbf{Op}\mathbf{Fib}_s(\K)
  \end{equation}
\end{thm}
\begin{proof}
  Denote by $\mathsf{R}'_{\mathscr A}$ and $\mathsf{R}'_{\K}$ the free split opfibration 2-monad on
  $\mathscr{A}^\two$ and $\K^\two$ respectively. Clearly
  $U^\two\mathsf{R}'_{\mathscr A}=\mathsf{R}'_{\K}U^\two$, and there is a monad
  morphism $\mathsf{R}'_{\K}\to U^\two\mathsf{R}'_{\mathscr A}F^\two$. Since
  $\mathsf{R}$ is by definition the codomain-preserving coreflection of
  $U^\two\mathsf{R}'_{\mathscr A}F^\two$, there exists a 2-monad morphism
  $\mathsf{R}\to\mathsf{R}'_{\K}$, which induces the 2-functor of the
  statement.
\end{proof}
\begin{rmk}
  The 2-monad $\mathsf{R}$ on $\K^\two$ of Proposition~\ref{prop:16} has a
  slightly more elementary description that will be useful later in our
  work. The associated 2-functorial factorisation $f={R}f\cdot{L}f\colon
  A\to B$ can be described as follows. The morphism ${R}f$ is given by the
  comma object displayed below, and ${L}f$ is the unique morphism such that
  $\mu_f\cdot{}{L}f=1$.
  \begin{equation}
    \label{eq:159}
    \xymatrixrowsep{.6cm}
    \diagram
    A\ar[dr]|{{L}f}\ar@/^/[drr]^{i_A}\ar@/_/[ddr]_f&&\\
    &Kf\ar[r]^{q_f}\ar[d]^{{R}f}\drtwocell<\omit>{\hole \mu_f}&UFA\ar[d]^{UFf}\\
    &B\ar[r]_{i_B}&UFB
    \enddiagram
  \end{equation}
  This is so because $\mu_f$ is related to the comma object $\nu$ that appears
  in~\eqref{eq:110} via the equality below.
  \begin{equation}
    \label{eq:184}
    \xymatrixrowsep{.6cm}
    \diagram
    Kf\ar[r]^-{q_f}\drtwocell<\omit>{\hole\mu_f}\ar[d]_{{R}f}&
    UFA\ar[d]^{UFf}\\
    B\ar[r]_-{i_B}&
    UFB
    \enddiagram
    =
    \diagram
    Kf\ar[d]_{{R}f}\ar[r]^{\tau_f}\ar@{}[dr]|{\mathrm{pb}}&
    UK'Ff\ar[r]^-{Uq_{Ff}}\ar[d]|{UR'{Ff}}\drtwocell<\omit>{\hole\hole\hole U\nu_{Ff}}&
    UFA\ar[d]^{UFf}\\
    B\ar[r]_-{i_B}&
    UFB\ar@{=}[r]&
    UFB
    \enddiagram
  \end{equation}
  The multiplication is given by a morphism $\pi_f\colon R^2f\to Rf$ that
  satisfies the equality below.
  \begin{equation}
    \label{eq:68}
    \diagram
    K{R}f\ar[d]_{R^2f}\ar[r]^-{\pi_f}&
    Kf\ar[d]|{{R}f}\ar[r]^-{q_f}\drtwocell<\omit>{\hole\mu_f}&
    UFA\ar[d]|{UFf}\\
    B\ar@{=}[r]&
    B\ar[r]_-{i_B}&
    UFB
    \enddiagram
    =
    \diagram
    K{R}f\ar[d]|{R^2f}\ar[r]^-{q_{{R}f}}
    \drtwocell<\omit>{\hole\hole\mu_{{R}f}} &
    UFKf\ar[d]|{UF{R}f}\ar[r]^-{UFq_f}\drtwocell<\omit>{\hole\hole\hole UF\mu_f}&
    UFUFA\ar[d]^{UFUFf}\ar[r]^-{Ue_{FA}}&
    UFA\ar[d]^{UFf}\\
    B\ar[r]_{i_{B}}&
    UFB\ar[r]_-{UFi_{B}}&
    UFUFB\ar[r]_-{Ue_{FB}}&
    UFB
    \enddiagram
  \end{equation}
  \label{rmk:8}
\end{rmk}
\section{{AWFS}s through simple adjunctions}
\label{sec:transf-along-left}

If $(\mathsf{L}',\mathsf{R}')$ is an \textsc{awfs} on $\mathcal{A}$, and $F\dashv
U\colon \mathcal{A}\to\C$ an adjunction, we obtain a \emph{transferred right
  algebraic weak factorisation system} $((L,\Phi),\mathsf{R})$ in \C. The monad
$\mathsf{R}$ is the codomain-preserving coreflection of the monad
$U^\two\mathsf{R}'F^\two$ on $\C^\two$ --~Lemma~\ref{l:5}. This means that
$Rf$ is given by the pullback in \C\ depicted on the left hand side.
\begin{equation}
  \label{eq:86}
  \xymatrixrowsep{.5cm}
  \diagram
  Kf\ar[d]_{Rf}\ar[r]&UK'Ff\ar[d]^{UR'Ff}\\
  B\ar[r]^{i_B}&UFB
  \enddiagram
  \qquad
  \xymatrixrowsep{.5cm}
  \diagram
  A\ar[dr]|{{L}f}\ar@/_/[ddr]_f\ar[r]^{i_A}&UFA\ar[dr]^{UL'{Ff}}&\\
  &Kf\ar[d]^{{R}f}\ar[r]^{\tau_f}&UK'Ff\ar[d]^{UR'{Ff}}\\
  &B\ar[r]^{i_B}&UFB
  \enddiagram
\end{equation}
The associated 2-functorial factorisation is given by
$f={R}f\cdot{}{L}f\colon A\to B$ where ${L}f$ is as in the diagram
on the right hand side above.  The corresponding domain-preserving copointed
endofunctor $(L,\Phi)$ on $\C^\two$ sends $f$ to ${L}f$ and
$\Phi_f=(1,{R}f)\colon {L}f\to f$, and can be constructed as the
pullback on the left below. As a consequence, the diagram on the right is a
pullback, as can be easily shown.
\begin{equation}
  \label{eq:52}
  \diagram
  L\ar[d]_\Phi\ar[r]&U^\two L'F^\two\ar[d]^{U^\two\Phi'F^\two}\\
  1\ar[r]^-{i^\two}&U^\two F^\two
  \enddiagram
  \qquad
  \diagram
  (L,\Phi)\text-\mathrm{Coalg}\ar[d]\ar[r]&(L',\Phi')\text-\mathrm{Coalg}\ar[d]\\
  \C^{\two}\ar[r]^{F^\two}&\mathcal{A}^\two
  \enddiagram
\end{equation}
(Recall that an $(L,\Phi)$-coalgebra structure on $f\in\mathscr{C}^\two$ is a
section of $\Phi_f\colon Lf\to f$).

The above considerations hold not only in the case of categories and functors but
also in the case of 2-categories, 2-adjunctions, etc, which we assume for the
rest of the section. We also assume that the comonad $\mathsf{L}'$ on the
2-category $\mathscr{A}^\two $ has as coalgebras the left adjoint coretracts in
the 2-category $\mathcal{A}$, which we assume to have comma objects; the
2-category $\mathsf{L}'\text-\mathrm{Coalg}_s$ can be written in the notation
used in~\cite{MR3393453} as $\mathbf{Lari}(\mathcal{A})$, where \textsc{lari}
stands for left adjoint right inverse.
The resulting factorisation $f={R}f\cdot{L}f$ can be described by the
following diagram, where the 2-cell $\mu$ is a comma object.
\begin{equation}
  \label{eq:94}
  \xymatrixrowsep{.5cm}\diagram
  A\ar[dr]|{{L}f}&&\\
  &Kf\ar[d]_{{R}f}\ar[r]^-{q_f}\drtwocell<\omit>{\mu_f}&UF\ar[d]^{UFf}\\
  &B\ar[r]_-{i_B}&UFB
  \enddiagram
  \quad=\quad
  \diagram
  A\ar@/_/[ddr]_f\ar@/^/[drr]^{i_A}\ar@{}[ddrr]|=&&\\
  &&UFA\ar[d]^{UFf}\\
  &B\ar[r]_-{i_B}&UF
  \enddiagram
\end{equation}

\begin{df}
  \label{df:11}
  If $F\colon\K\to\mathscr{A}$ is a 2-functor, define a 2-category
  \begin{equation}
    \label{eq:81}
    \xymatrixrowsep{.6cm}
    \diagram
    F\text-\mathbf{Emb}\ar@{}[rd]|{\mathrm{pb}}\ar[r]\ar[d]_G&\mathbf{Lari}(\mathscr{A})\ar[d]\\
    \K^\two\ar[r]_-{F^\two}&\mathscr{A}^\two
    \enddiagram
  \end{equation}
  whose objects may be called \emph{$F$-embeddings} --~the terminology is widely
  used in the context of categories enriched in posets, as topological
  embeddings are $F$-embeddings for a certain choice of $F$. Explicitly, an
  object of $F\text-\mathbf{Emb}$ is a morphism $f$ in \K\ equipped with a right
  adjoint retract $r_f$ for $Ff$ in $\mathscr A$, with counit $\alpha_f\colon
  Ff\cdot r_f\Rightarrow 1$. A morphism $(f,r_f,\alpha_f)\to(g,r_g,\alpha_g)$ is
  a morphism $(h,k)\colon f\to g$ in $\K^\two$ such that $(Fh,Fk)$ is a morphism
  in $\mathbf{Lari}(\mathscr{A})$; this means that $Fh\cdot r_f=r_g\cdot Fk$. (It
  is not hard to show that the compatibility with the counits, expressed in
  $Fk\cdot \alpha_f=\alpha_g\cdot Fk$, is automatically satisfied.)
\end{df}
\begin{rmk}
  \label{rmk:10}
  The composition of \textsc{lari}s described in Section~\ref{sec:basic-example}
  induces a composition on $F\text-\mathbf{Emb}$ and the projections
  $\K^\two\leftarrow F\text-\mathbf{Emb}\to\mathbf{Lari}(\K)$ preserve
  it. Explicitly, if $(f,r)$ and $(f',r')$ are objects of $F\text-\mathbf{Emb}$
  with $f$ and $f'$ composable morphisms of $\K$, then $(f'\cdot f,r\cdot r')$
  has a canonical structure of an $F$-embedding arising from $F(f'\cdot f)\cong
  Ff'\cdot Ff\dashv r\cdot r'$.
\end{rmk}

The pullback square that defines $F\text-\mathbf{Emb}$ can be factorised as the
pasting of two pullback squares, one of which we have already met
in~\eqref{eq:52}. In particular, each $f\in F\text-\mathbf{Emb}$ has an
$(L,\Phi)$-coalgebra structure.
\begin{equation}
  \label{eq:109}
  \xymatrixrowsep{.4cm}
  \diagram
  F\text-\mathbf{Emb}\ar[r]\ar[d]&\mathbf{Lari}(\mathscr{A})\ar[d]\\
  (L,\Phi)\text-\mathrm{Coalg}\ar[r]\ar[d]&(L',\Phi')\text-\mathrm{Coalg}\ar[d]\\
  \K^\two\ar[r]^-{F^\two}&\mathscr{A}^\two
  \enddiagram
\end{equation}
\begin{thm}
  Suppose given a 2-adjunction $F\dashv U\colon \mathscr{A}\to\K$ where
  $\mathcal{A}$ has comma objects and $\K$ has pullbacks. Then the following are equivalent.
  \begin{enumerate}
  \item \label{item:1} The 2-adjunction is simple.
  \item \label{item:2} The copointed endo-2-functor $(L,\Phi)$ can be extended
    to a comonad $\mathsf{L}$ and $F\text-\mathbf{Emb}$ is isomorphic to
    $\mathsf{L}\text-\mathrm{Coalg}_s$ over $(L,\Phi)\text-\mathrm{Coalg}$.
  \item \label{item:4} The forgetful 2-functor $G\colon F\text-\mathbf{Emb}\to\K^\two$ has a right
    adjoint and the induced comonad on $\K^\two$ has underlying copointed
    endo-2-functor $(L,\Phi)$.
  \end{enumerate}
  \label{thm:4}
\end{thm}
\begin{proof}
(\ref{item:1})$\Rightarrow$(\ref{item:4})
The hypothesis tells us that there are adjunctions $F{L}g\dashv
r_{Lg}$ in $\mathscr A$, where $r_{Lg}=e_{FA}\cdot Fq_g\colon FKf\to FA$,
with counits $\varepsilon_{{L}g}\colon F{L}f\cdot r_{Lg}\Rightarrow 1$
that are modifications in $g\in\K^\two$. This defines a 2-functor $J\colon
\K^\two\to F\text-\mathbf{Emb}$ that sends $g$ to
$({L}g,r_{Lg},\varepsilon_{Lg})$. We shall show that $J$ is a right adjoint
to $G$.

Suppose that $(f,r_f,\alpha_f)$ is an object of $F\text-\mathbf{Emb}$, and
construct a morphism $s_f$ as the unique morphism satisfying the displayed
equality.
\begin{equation}
  \label{eq:112}
  \xymatrixrowsep{.5cm}
  \diagram
  B\ar[r]^-{s_f}&Kf\ar[d]_{{R}f}\drtwocell<\omit>{\hole\mu_f}\ar[r]^-{q_f}&
  UFA\ar[d]^{UFf}\\
  &B\ar[r]_-{i_B}&UFB
  \enddiagram
  \quad=\quad
  \diagram
  B\ar[r]^-{i_B}&UFB\drtwocell<\omit>{<-1>\hole\hole U\alpha_f}\ar[r]^-{Ur_f}\ar@/_/[dr]_1&UFA\ar[d]^{UFf}\\
  &&UFB
  \enddiagram
\end{equation}
To be more precise, $q_f\cdot s_f= Ur_f\cdot i_Y$, ${R}f\cdot s_f=1$ and
$\mu_f\cdot s_f= U\alpha_f\cdot i_Y$.
We claim that
\begin{equation}
  \label{eq:114}
  \xymatrixrowsep{.5cm}
  \diagram
  A\ar@{=}[r]\ar[d]_f&A\ar[d]^{{L}f}\\
  B\ar[r]^-{s_f}&Kf
  \enddiagram
  \qquad\text{is a morphism in $F\text-\mathbf{Emb}$.}
\end{equation}
We have to show the compatibility of $s_f$ with the right adjoints, which
means that $r_{Lf}\cdot Fs_f=r_f$ should hold, as it indeed
does, as witnessed by
\begin{equation}
  \label{eq:115}
  r_{Lf}\cdot Fs_f=e_{FA}\cdot Fq_{f}\cdot Fs_f=e_{FA}\cdot F(q_f\cdot s_f)=
  e_{FA}\cdot F(Ur_f\cdot i_B)=r_f.
\end{equation}

The morphisms $\Psi_{(f,r_f,\alpha_f)}=(1,s_f)\colon f\to {L}f$ in
$F\text-\mathbf{Emb}$ form a 2-natural transformation $\Psi\colon
1_{F\text-\mathbf{Emb}}\Rightarrow JU$. On the other hand, we already have a
2-natural transformation $\Phi\colon UJ=L\Rightarrow 1_{\K^\two}$ with
components $\Phi_g=(1,{R}g)$. We now proceed to show that these are the unit
and the counit of a 2-adjunction $G\dashv J$.
We start by considering
\begin{equation}
  \label{eq:125}
  G\xRightarrow{G\Psi}GJG=LG\xRightarrow{\Phi G}G
\end{equation}
and evaluating on $(f,r_f,\alpha_f)\in F\text-\mathbf{Emb}$ to obtain
$(1,{R}f)\cdot(1,s_f)=(1,{R}f,\cdot s_f)=(1,1)$, so one of the triangle
identities holds. The other triangle identity involves the composition
\begin{equation}
  \label{eq:127}
  J\xRightarrow{\Psi J}JGJ\xRightarrow{J\Psi}J
\end{equation}
which evaluated on $g\colon X\to Y$ in $\K^\two$ gives
\begin{equation}
  \label{eq:131}
  \xymatrixcolsep{1.7cm}
  \diagram
  \cdot\ar@{=}[r]\ar[d]_{{L}g}&
  \cdot\ar[d]_{L^2_g}\ar@{=}[r]&\cdot\ar[d]^{{L}g}\\
  \cdot\ar[r]^-{s_{{L}g}}&\cdot\ar[r]^-{K(1,{R}g)}&\cdot
  \enddiagram
\end{equation}
where $s_{{L}g}$ is defined according to~\eqref{eq:112}. We have to show
that $K(1,{R}g)\cdot s_{{L}g}=1$; since the codomain is the comma object
$Kg$, this equality is equivalent to the conjunction of the three conditions
\begin{equation}
  \label{eq:136}
  {R}g\cdot K(1,{R}g)\cdot s_{{L}g}={R}g\qquad
  q_g\cdot  K(1,{R}g)\cdot s_{{L}g}=q_g\qquad
   \mu_g\cdot K(1,{R}g)\cdot s_{{L}g}=\mu_g
\end{equation}
where $\mu_g\colon UFg\cdot q_g\Rightarrow i_Y\cdot {R}g$ is the universal
2-cell of the comma object. The first of these equalities always holds, since
${R}g\cdot K(1,{R}g)\cdot s_{{L}g}= {R}g\cdot {R}{{L}g}\cdot
s_{{L}g}={R}g$. The second equality also holds, by definition of $r_{Lg}$:
\begin{equation}
  \label{eq:144}
  q_g\cdot  K(1,{R}g)\cdot s_{{L}g}=q_{Lg}\cdot
  s_{{L}g}=Ur_{Lg}\cdot i_{Kg}
  =Ue_{FX}\cdot UFq_{g}\cdot i_{Kg}
  = q_{g}.
\end{equation}
The third and last equality we need to verify can be rewritten by using what we know of the
definition of $s_f$, as
\begin{equation}
  \label{eq:146}
  \mu_g=
  \mu_g\cdot K(1,{R}g)\cdot s_{{L}g}=UF{R}g\cdot\mu_{{L}g}\cdot
  s_{{L}g}=
  UF{R}g\cdot U\alpha_{{L}g}\cdot i_Y.
\end{equation}
In order to prove~(\ref{eq:146}) we may consider
the 2-cell $\gamma_g$ that is the transpose of $\mu_g$ under the
2-adjunction $F\dashv U$.
\begin{equation}
  \label{eq:149}
  \xymatrixcolsep{.5cm}
  \diagram
  Kg\ar[r]^-{q_g}\ar[d]_{{R}g}\drtwocell<\omit>{\hole\mu_g}&UFX\ar[d]^{UFg}\\
  Y\ar[r]_-{i_Y}&UFY
  \enddiagram
  \quad\longleftrightarrow{}\quad
  \diagram
  FUFX\ar[r]^-{e_{FX}}&FX\ar[d]^{Fg}\\
  FKg\ar[u]^{Fq_g}\ar[r]_-{F{R}g}\rtwocell<\omit>{<-3>\hole\gamma_g}&FY
  \enddiagram
\end{equation}
We will soon use the fact that the equality $\mu_g\cdot{L}g=1$
translates to $\gamma_g\cdot
F{L}g=1$.
Transposing~(\ref{eq:146}) under $F\dashv U$ yields the equivalent
equation
\begin{equation}
  \label{eq:139}
  \gamma_g=F{R}g\cdot\alpha_{{L}g}
\end{equation}
that we prove by considering the pasting diagram
depicted below
\begin{equation}
  \label{eq:162}
  \xymatrixrowsep{.5cm}
  \xymatrixcolsep{2cm}
  \diagram
  &FX\ar[dr]^{F{L}g}&&FX\ar[dr]^{Fg}&\\
  FKg\ar[ur]^{r_{Lg}}\ar@{=}[rr]
  \rrtwocell<\omit>{<-2>\hole\hole\alpha_{{L}g} } &&
  FKg\ar[rr]_-{F{R}g} \rrtwocell<\omit>{<-2>\hole\gamma_g}
  \ar[ur]^{r_{Lg}}&&
  FY
  \enddiagram
\end{equation}
and calculating
\begin{equation}
  \label{eq:166}
  F{R}g\cdot\alpha_{Lg}=
  \bigl(
  \gamma_g\cdot F{L}g\cdot r_{Lg}
  \bigr)
  \bigl(
  F{R}g\cdot\alpha_{Lg}
  \bigr)
  =
  \bigl(
  Fg\cdot r_{Lg}\cdot\alpha_{{L}g}
  \bigr)
  \gamma_g
  =\gamma_g
\end{equation}
where we have used $\gamma_g\cdot F{L}g=1$ and
$r_{Lg}\cdot\alpha_{{L}g}=1$. This completes this part of the
proof.

(\ref{item:4})$\Rightarrow$(\ref{item:2}) is clear.

(\ref{item:2})$\Rightarrow$(\ref{item:1}) Suppose that
$\mathsf{L}=(L,\Phi,\Sigma)$ is a 2-comonad with category of coalgebras and
strict maps isomorphic to $F\text-\mathbf{Emb}$ over
$(L,\Phi)\text-\mathrm{Coalg}$.  The comultiplication has components
$\Sigma_f=(1,\sigma_f)\colon Lf\to L^2f$. Since $\Sigma_f$ is a
$(L,\Phi)$-coalgebra structure on $Lf$, and the commutativity of the top square
in~\eqref{eq:109}, the morphism $\sigma_f\colon Kf\to K{L}f$ and the right
adjoint $r_{Lf}$ of $F{L}f$ are related by
\begin{equation}
  \label{eq:152}
  r_{Lf}=e_{FA}\cdot Fq_{{L}f}\cdot F\sigma_f\colon FK{L}f\longrightarrow FA.
\end{equation}
We have to show that $r_{Lf}=e_{FA}\cdot Fq_f$, for which it will suffice to
show that
\begin{equation}
  \label{eq:153}
  q_f=q_{{L}f}\cdot\sigma_f.
\end{equation}
We make use of the counit axiom $1=L\Phi_f\cdot\Sigma_f$, whose codomain
component is $1_{Kf}=K(1,{R}f)\cdot \sigma_f$. Composing the latter with $q_f$
we obtain the required equality: $q_f=q_f\cdot K(1,{R}f)\cdot\sigma_f=
q_{{L}f}\cdot\sigma_f$.
\end{proof}

\begin{thm}
  \label{thm:1}
  If the 2-adjunction $F\dashv U\colon \mathscr A\to \K$ is simple, where
  $\mathscr A$ has comma objects and $\K$ has pullbacks, then there exists an
  {\normalfont\textsl{\textsc{awfs}}} $(\mathsf{L},\mathsf{R})$ on $\K$ that
  extends the 2-functorial factorisation $f={R}f\cdot{L}f$. Furthermore,
  \begin{itemize}
  \item this {\normalfont\textsl{\textsc{awfs}}} is lax idempotent; and,
  \item the 2-category $\mathsf{L}\text-\mathrm{Coalg}_s$ is isomorphic over
    $\K^\two$ to the 2-category $F\text-\mathbf{Emb}$ of Definition~\ref{df:11}.
  \item the 2-monad $\mathsf{R}$ is the codomain-preserving reflection of
    $U^\two\mathsf{R}'\mathsf{F}^\two$, where $\mathsf{R}'$ is the free split
    opfibration 2-monad on $\mathscr A^\two$.
\end{itemize}

\end{thm}
\begin{proof}
  The category $F\text-\mathbf{Emb}$ has an obvious
  composition, mentioned in Remark~\ref{rmk:10}, so $\mathsf{L}\text-\mathrm{Coalg}_s$ does too. This already
  generates a 2-monad $\mathsf{R}=(R,\Lambda,\Phi)$ with
  $\Lambda_f=({L}f,1)$. We shall show that $\mathsf{R}$ is the codomain
  preserving coreflection of the monad $U^\two\mathsf{R}'\mathsf{F}^\two$
  considered in Section~\ref{sec:extens-transf-a.f}, where $\mathsf{R}'$ is the
  free split opfibration 2-monad on $\mathscr A$. Both this coreflection, that
  we denote by $\bar {\mathsf{R}}$, and $\mathsf{R}$ have the same underlying
  2-functor and unit; this is because the description of the respective
  2-functorial factorisations in Remark~\ref{rmk:8} and just before
  Definition~\ref{df:11} coincide.
  It remains to prove that they share the same multiplication. Denote by
  $\bar\Pi$ the multiplication of $\bar{\mathsf{R}}$ and by $\Pi$ that of
  $\mathsf{R}$. We know from Section~\ref{sec:from-double-categ}, or, more
  precisely, by a dual case to that explained in that section, that $\Pi$ can be
  described in terms of the composition of $\mathsf{L}$-coalgebras as the unique
  morphism of $\mathsf{L}$-coalgebras
  \begin{equation}
    \label{eq:154}
    \xymatrixrowsep{.3cm}
    \diagram
    A\ar[d]_{{L}f}\ar@{=}[r]&A\ar[dd]^{Lf}\\
    Kf\ar[d]_{L{{R}f}}&\\
    K{R}f\ar[r]^-{\pi_f}&Kf
    \enddiagram
  \end{equation}
  that composed with the counit $(1,{R}f)\colon {L}f\to f$ equals the
  morphism $(1,R^2f)$ in $\K^\two$; the
  $\mathsf{L}$-coalgebra structure of
  ${L}({{R}f})\cdot Lf$ is that given by the composition of coalgebras,
  ie the composition of $F\text-\mathbf{Emb}$. To show $\Pi=\bar \Pi$ we only
  need to show that $(1,\bar\pi_f)$ satisfies the same properties.
  We certainly know that
  ${R}f\cdot\bar\pi_f=R^2f$, so it remains to show that $(1,\bar\pi_f)$ is a strict
  morphism of $\mathsf{L}$-coalgebras, or, what is the same, that it is a
  morphism in $F\text-\mathbf{Emb}$. We know that there are coretract
  adjunctions $F{L}f\dashv r_{Lf}$ and $F{L}({{R}f})\dashv
  r_{L{R}f}$ in $\mathscr A$. To say that $(1,\bar\pi_f)$ is a
  morphism in $F\text-\mathbf{Emb}$ is to say that
  \begin{equation}
    \label{eq:74}
    r_{Lf}\cdot F\bar\pi_f= r_{Lf}\cdot r_{L{R}f},
  \end{equation}
  or, substituting the right adjoints $r$ by their expressions given by the
  simplicity of $F\dashv U$,
  \begin{multline}
    \label{eq:75}
    e_{FA}\cdot Fq_f\cdot F\bar\pi_f= e_{FA}\cdot Fq_f\cdot e_{Kf}\cdot
    Fq_{{L}f}=\\
    = e_{FA}\cdot e_{FUFA}\cdot FU Fq_f\cdot Fq_{{L}f}
    = e_{FA}\cdot FUe_{FA}\cdot FU Fq_f\cdot Fq_{{L}f}.
  \end{multline}
  Taking the transpose of each side under $F\dashv U$, the equality is
  equivalent to
  \begin{equation}
    \label{eq:83}
    q_f\cdot\bar\pi_f= Ue_{FA}\cdot UF q_f\cdot q_{{L}f}
  \end{equation}
  which is precisely the equality satisfied by $\bar\pi_f$ as mentioned in
  Remark~\ref{rmk:8}. This completes the proof.
\end{proof}
% The \textsc{awfs} constructed in the theorem above factors a morphism $f\colon A\to
% B$ as $f=\rho_f\cdot{}\lambda_f$, where $\rho_f$ is given by the comma displayed
% below --~the same as in~\eqref{eq:159}~-- and $\lambda_f$ is the unique morphism such that
% $\mu_f\cdot{}\lambda_f=1$.
%   \begin{equation}
%     \label{eq:159}
%     \xymatrixrowsep{.6cm}
%     \diagram
%     A\ar[dr]|{\lambda_f}\ar@/^/[drr]^{i_A}\ar@/_/[ddr]_f&&\\
%     &Kf\ar[r]^{t_f}\ar[d]^{\rho_f}\drtwocell<\omit>{\hole \mu_f}&UFA\ar[d]^{UFf}\\
%     &B\ar[r]_{i_B}&UFB
%     \enddiagram
%   \end{equation}

We now look at the fibrant replacement 2-monad associated to the \textsc{awfs} constructed.
\begin{cor}
  \label{cor:9}
  Suppose that in Theorem~\ref{thm:4} the 2-category $\K$ has a terminal object $1$,
  and that $i_1\colon 1\to UF1$ is a right adjoint of $UF1\to 1$. Then the
  restriction of $\mathsf{R}$ to $\K/1\cong \K$ --~the fibrant replacement
  2-monad of $(\mathsf{L},\mathsf{R})$~-- is isomorphic to $UF$.
\end{cor}
\begin{proof}
  Let us denote by $f\colon A\to 1$ the unique morphism into the terminal
  object, and by $\mathsf{R}_1$ the restriction of $\mathsf{R}$ to $\K/1$. We
  shall show that in
  the comma object
  \begin{equation}
    \label{eq:77}
    \xymatrixrowsep{.5cm}
    \diagram
    Kf\ar[r]^-{q_A}\ar[d]_{{R}f}\drtwocell<\omit>{}&
    UFA\ar[d]^{UFf}\\
    1\ar[r]^-{i_1}&
    UF1
    \enddiagram
  \end{equation}
  the projection $q_A$ is an isomorphism. For any morphism $x\colon X\to UFA$,
  there exists a unique 2-cell $UFf\cdot{}x\Rightarrow i_1\cdot{}!$, as these are in
  bijection, by mateship along $i_1\dashv f$, with endo-2-cells of $X\to 1$, of
  which there is only one. Hence $\K(X,q_A)$ is an
  isomorphism, for each $X$, and thus $q_A$ is an isomorphism.
  Since $q_A\cdot{}{L}f=i_A$, and the compatibility of
  $q$ with the multiplication of $\mathsf{R}$ and $UF$ exhibited
  by~\eqref{eq:68}, namely
  \begin{equation}
    \label{eq:84}
    q_A\cdot\pi_f=Ue_{FA}\cdot UFq_f\cdot q_{Kf},
  \end{equation}
  we have that $q$ is a 2-monad isomorphism $q\colon R_1\to
  UF$.
\end{proof}

We conclude the section with the following lemma, which will be of use in later
sections.
Corollary~\ref{cor:9} says that for any morphism $b\colon 1\to B$ from the
terminal object of \K\ the fibre $A_b$ of any
$\mathsf{R}$-algebra $g\colon A\to B$ --~ie the pullback of $g$ along
$b$~-- has a structure of a $\mathsf T$-algebra, for the monad $\mathsf{T}$
induced by $F\dashv U$.
\begin{equation}
  \label{eq:172}
  \xymatrixrowsep{.5cm}
  \diagram
  A_b\ar[r]^{z_b}\ar[d]_{!}&A\ar[d]^g\\
  1\ar[r]^-b&B
  \enddiagram
\end{equation}
Furthermore, $(z_b,b)$ is a morphism of $\mathsf{R}$-algebras.
\begin{lemma}
  \label{l:27}
  Assume the conditions of Corollary~\ref{cor:9}, and denote by $\mathsf{T}$ the
  monad generated by $F\dashv U$. Given $g\colon A\to B$ and $b\colon 1\to B$,
  the morphism
  \begin{equation}
    \label{eq:173}
    (Kg)_b\xrightarrow{z_b}Kg\xrightarrow{q_g}TA
  \end{equation}
  is a morphism of $\mathsf{T}$-algebras.
\end{lemma}
\begin{proof}
  Denote by $a\colon T(Kg)_b\to (Kg)_b$ the $\mathsf{T}$-algebra structure given
  by Corollary~\ref{cor:9}. We are to show that the following rectangle commutes.
  \begin{equation}
    \label{eq:174}
    \xymatrixrowsep{.5cm}
    \diagram
    T(Kg)_b\ar[r]^-{Tz_b}\ar[d]_a&TKg\ar[r]^-{Tq_b}&T^2A\ar[d]^{m_A}\\
    (Kg)_b\ar[r]^-{z_b}&Kg\ar[r]^-{q_b}&TA
    \enddiagram
  \end{equation}
  In order to do so, consider the string of equalities displayed below,
  the first of which holds since $(z_b,b)$ is a morphism of
  $\mathsf{R}$-algebras; the second holds by definition of $\pi_g$, depicted in~\eqref{eq:68}; the next equality reflects the definition of $K(z_b,b)$ and
  the fact that the restriction of $\mathsf{R}$ to $\K/1$ is $\mathsf{T}$
  --~Corollary~\ref{cor:9}.
    \begin{equation}
      \label{eq:175}
      \xymatrixrowsep{.5cm}
      \xymatrixcolsep{.4cm}
      \diagram
      T(Kg)_b\ar[r]^-a\ar[d]&
      (Kg)_b\ar[d]\ar[r]^-{z_b}&
      Kg\ar[d]_{{R}g}\ar[r]^-{q_g}\drtwocell<\omit>{}&
      TA\ar[d]^{Tg}\\
      1\ar@{=}[r]&
      1\ar[r]_-b&
      B\ar[r]_{i_B}
      &TB
      \enddiagram
      =
      \xymatrixcolsep{.8cm}\diagram
      T(Kg)_b\ar[r]^-{K(z_b,b)}\ar[d]&
      K{R}g\ar[d]_{{R}^2{g}}\ar[r]^-{\pi_g}&
      Kg\ar[d]_{{R}g}\ar[r]^-{q_g}\drtwocell<\omit>{}&
      TA\ar[d]^{Tg}\\
      1\ar[r]_-b&
      B\ar@{=}[r]&
      B\ar[r]_-{i_B}&
      TB
      \enddiagram
      =
    \end{equation}
    \begin{equation}
      \label{eq:177}
      =
      \xymatrixrowsep{.5cm}
      \diagram
      T(Kg)_b\ar[r]^-{K(z_b,b)}\ar[d]&
      K{R}{g}\ar[d]_{R^2g}\drtwocell<\omit>{}\ar[r]&
      TKg\ar[r]^-{Tq_g}\drtwocell<\omit>{}\ar[d]|{T{R}g}&
      T^2A\ar[r]^-{m_A}\ar[d]^{T^2g}&
      TA\ar[d]^{Tg}\\
      1\ar[r]_-b&
      B\ar[r]_-{i_B}&
      TB\ar[r]_-{Ti_B}&
      T^2B\ar[r]_-{m_B}&
      TB
      \enddiagram
      =
    \end{equation}
    \begin{equation}
      \label{eq:178}
      =
      \xymatrixrowsep{.5cm}
      \diagram
      T(Kg)_b\ar[d]\ar[r]^1\drtwocell<\omit>{{!}}&
      T(Kg)_b\ar[d]\ar[r]^-{Tz_b}&
      TKg\ar[d]_{T{R}g}\ar[r]^-{Tq_g}\drtwocell<\omit>{}&
      T^2A\ar[r]^-{m_A}\ar[d]^{T^2g}&
      TA\ar[d]^{Tg}\\
      1\ar[r]_-{i_1}&
      T1\ar[r]_-{Tb}&
      TB\ar[r]_-{Ti_B}&
      T^2B\ar[r]_-{m_B}&
      TB
      \enddiagram
    \end{equation}
  It is now clear that \eqref{eq:174}~commutes, completing the proof.
\end{proof}

\section{Simple 2-monads}
\label{sec:simple-2-monads}

A 2-monad $\mathsf{T}$ on a 2-category \K\ with lax limits of morphisms is said
to be \emph{simple} if the usual Eilenberg-Moore adjunction $F\dashv
U\colon\mathsf{T}\text-\mathrm{Alg}_s\to\K$ is simple with respect to the
coreflection--opfibration \textsc{awfs} on $\mathsf{T}\text-\mathrm{Alg}_s$ -- in the
sense of Definition~\ref{df:9}. To make this definition more explicit, consider
the factorisation of a morphism $f\colon A\to B$ in \K\ depicted in~\eqref{eq:159}, and
recall from Proposition~\ref{prop:16} that the simplicity of $F\dashv
U$ amounts to the existence of a certain
coretract adjunction in $\mathsf{T}\text-\mathrm{Alg}_s$; namely
\begin{equation}
  T{L}f \dashv m_{K'Ff}\cdot{}Tq_f\label{eq:180}
\end{equation}
where $m$ is the multiplication
of $\mathsf{T}$ and the rest of the notation is as in the diagram~\eqref{eq:184}. This
adjunction must be an adjunction in $\mathsf{T}\text-\mathrm{Alg}_s$ -- a condition
that is redundant when, for example, $\mathsf{T}$ is lax idempotent, as it will
often be in our examples.
\begin{rmk}
  \label{rmk:20}
  At this point it is useful to consider the meaning of simple 2-monads and the
  previous proposition when the 2-category is locally discrete, ie
  just a category $\mathcal{C}$. In this case comma objects are just pullbacks,
  and the coreflection--opfibration factorisation becomes the orthogonal
  factorisation $(\mathrm{Iso},\mathrm{Mor})$ that factors a morphism $f$ as the
  identity followed by $f$. To say that a monad $\mathsf{T}$ on $\mathcal{C}$ is
  simple is to say that the image of the comparison morphism $\ell$, which goes
  from the naturality square of $i$ to the pullback in \eqref{eq:160}, is sent
  to an isomorphism by the free $\mathsf{T}$-algebra functor. Equivalently, one can
  say that $T\ell$ is an isomorphism. Observe that, when $\mathsf{T}$ is a reflection,
  this gives the definition of \emph{simple reflection} in the sense of
  \cite{MR779198}.
  \begin{equation}
    \label{eq:160}
    \xymatrixrowsep{.5cm}
    \diagram
    A\ar[dr]^\ell\ar@/_/[ddr]_f\ar@/^/[rrd]^{i_A}&&\\
    &B\times_{TB}TA\ar[d]\ar[r]&TA\ar[d]^{Tf}\\
    &B\ar[r]_-{i_B}&TB
    \enddiagram
  \end{equation}
\end{rmk}
We consider in this section some properties that guarantee that a 2-monad is simple,
thus inducing a transferred \textsc{awfs}.
We make the blanket assumption that the 2-category \K\ has pullbacks and cotensor products
with $\mathbf{2}$, and therefore comma objects.

\begin{notation}
\label{not:1}
If $T\colon \mathscr A\to\mathscr B$ is a 2-functor and $\lim D$ the limit of a
2-functor $D$ into $\mathscr A$, there is
a ``comparison'' morphism $T(\lim D)\to \lim TD$. We are interested in the limits
that are comma objects.
Given a cospan $f\colon{}A\to C\leftarrow B:g$, if $\alpha$ and $\beta$ are the
universal 2-cells of the comma objects $f\downarrow g$ and $Tf\downarrow
Tg$, then the comparison
morphism
\begin{equation}
  \label{eq:118}
  k\colon{}T(f\downarrow g)\longrightarrow Tf\downarrow Tg,
\end{equation}
is defined by the equality
\begin{equation}
  \label{eq:157}
  \xymatrixrowsep{.4cm}
  \diagram
  T(f\downarrow g)\ar[r]^-k&
  Tf\downarrow Tg\ar[r]\ar[d]\drtwocell<\omit>{\beta}&
  TA\ar[d]^{Tf}\\
  &TB\ar[r]_-{Tg}&
  TC
  \enddiagram
  =
  \diagram
  T(f\downarrow g)\drtwocell<\omit>{\hole\hole T(\alpha)}\ar[r]\ar[d]&
  TA\ar[d]^{Tf}&\\
  TB\ar[r]_-{Tg}&
  TC
  \enddiagram
\end{equation}
\end{notation}

\begin{prop}[Simplicity criterion]
  \label{prop:9}
  A 2-monad $\mathsf{T}=(T,i,m)$ is simple if it is lax idempotent and satisfies
  the following property:
  given  morphisms $f$, $u$, $v$ as displayed below, composition with the
  comparison morphism $k$ defined in Notation~\ref{not:1} induces a bijection between
  2-cells $\xi$ as on the left and 2-cells $\zeta$ as on the right. In other
  words, for each $\zeta$ there exists a unique $\xi$ such that $k\cdot \xi=\gamma$.
  \begin{equation}
    \label{eq:128}
    \xymatrixrowsep{.5cm}
    \diagram
    A\ar[d]_v\ar@(r,u)[dr]^u
    &\\
    Tf\downarrow i_B\rtwocell<\omit>{<-2.5>\xi}\ar[r]_-i&T(Tf\downarrow i_B)
    \enddiagram
    \quad\quad
    \diagram
    A\ar[d]_v\ar@(r,u)[dr]^{k\cdot{}u}&\\
    Tf\downarrow i_B\rtwocell<\omit>{<-2.5>\hole\zeta}\ar[r]_-{k\cdot{}i}&TTf\downarrow Ti_B
    \enddiagram
  \end{equation}
\end{prop}
\begin{proof}
  As we have done before, we will denote by $\mathsf{R}'$ the free split opfibration 2-monad on
  $\mathsf{T}\text-\mathrm{Alg}_s^\two$. Let $\mathsf{R}$ be the 2-monad
  on $\K^\two$ induced from $\mathsf{R}'$ by the simple adjunction
  $F\dashv U$, according to Theorem~\ref{thm:1}. By the same theorem,
  $\mathsf{R}$ is the
  codomain-preserving coreflection of the 2-monad $U^\two\mathsf{R}'F^\two$. The
  right part ${R}f$ of the factorisation of a morphism $f$ in \K\ induced by
  $\mathsf{R}$ is given by a comma object
  \begin{equation}
    \label{eq:116}
    \xymatrixrowsep{.5cm}
    \diagram
    Kf\ar[r]^-{q_f}\ar[d]_{{R}f}\drtwocell<\omit>{\hole\mu_f}&
    TA\ar[d]^{Tf}\\
    B\ar[r]_-{i_B}&
    TB
    \enddiagram
  \end{equation}
  and the left part ${L}f\colon A\to Kf$ is the unique morphism such that
  $\mu_f\cdot{}{L}f=1$. As explained at the beginning of the present section, we
  must exhibit a coretract adjunction \eqref{eq:180} in \K; this adjunction is
  automatically an adjunction in $\mathsf{T}\text-\mathrm{Alg}_s$, since
  $\mathrm{T}$ is lax idempotent.

  In order
  to define a counit $\alpha\colon T{L}f\cdot{}m_{K'Ff}\cdot{}Tq_f\Longrightarrow 1$ we can give its
  transpose under the free $\mathsf{T}$-algebra 2-adjunction, which is a
  2-cell $\bar\alpha\colon T{L}f\cdot{}q_f \Longrightarrow i_{Kf}$ in \K.

  The morphism $k$ of~\eqref{eq:118} is the unique such that satisfies the equality
  \begin{equation}
    \label{eq:120}
    \xymatrixrowsep{.5cm}
    \diagram
    TKf\ar[r]^-{Tq_f}\ar[d]_{T{R}f}\drtwocell<\omit>{\hole\hole T\mu_f}&
    T^2A\ar[d]^{T^2f}\\
    TB\ar[r]_-{Ti_B}&
    T^2B
    \enddiagram
    \quad
    =
    \quad
    \xymatrixrowsep{.5cm}
    \diagram
    TKf\ar[r]^-{k}&
    T^2f\downarrow Ti_B\ar[r]^-{d_0}\ar[d]_{d_1}\drtwocell<\omit>{\theta}&
    T^2A\ar[d]^-{T^2f}\\
    &TB\ar[r]_-{Ti_B}& T^2B
    \enddiagram
  \end{equation}
  To give $\bar\alpha$ is to equally give a pair of 2-cells, corresponding to
  $d_0\cdot{}k\cdot{}\bar\alpha$ and $d_1\cdot{}k\cdot{}\bar\alpha$:
  \begin{gather}
    \bar\alpha_1\colon Tq_f\cdot{}T{L}f\cdot{}q_f= Ti_A\cdot{}q_f
    \Longrightarrow i_{TA}\cdot{}q_f= Tq_f\cdot{}i_{Kf}
    \\
    \bar\alpha_2\colon T{R}f\cdot{}T{L}f\cdot{}q_f= Tf\cdot{}q_f \Longrightarrow
    i_B\cdot{}{R}f=T{R}f\cdot{}i_{Kf}
  \end{gather}
  compatible with $\theta$, in the sense that the following two compositions of
  2-cells must be equal.
  \begin{equation}
    \label{eq:76}
    T^2f\cdot Ti_A\cdot q_f\xrightarrow{T^2f\cdot\bar\alpha_1}
    T^2 f\cdot Tq_f\cdot i_{Kf}
    =
    T^2 f\cdot d_0\cdot k \cdot i_{Kf}
    \xrightarrow{\theta\cdot k\cdot i_{Kf}}
    Ti_B\cdot d_1\cdot  k \cdot i_{Kf}
  \end{equation}
  \begin{equation}
    \label{eq:82}
    T^2f\cdot d_0\cdot k\cdot T{L}f\cdot q_f
    \xrightarrow{\theta\cdot k\cdot T{L}f\cdot q_f}
    Ti_B\cdot d_1\cdot k\cdot T{L}f\cdot q_f
    % =
    % Ti_B\cdot T\rho_f \cdot T\lambda_f\cdot q_f
    =
    Ti_B\cdot Tf \cdot q_f
    \xrightarrow{Ti_B\cdot\bar\alpha_2} Ti_B\cdot i_B\cdot{R}f
  \end{equation}

  Set $\bar\alpha_1=\delta_A\cdot{}q_f$ and $\bar\alpha_2=\mu_f$, where $\delta\colon
  Ti\Rightarrow iT$ is the modification given by the lax idempotent structure of
  $\mathsf{T}$.
  % The fact that these 2-cells are compatible with the
  % comma object $\theta$ can easily be verified using the following string of
  % equalities, which hold, respectively, because $\delta$ is a modification,
  % naturality of $i$, functoriality of horizontal composition, the axiom
  % $\delta\cdot{}i=1$, and $\mu_f\cdot\lambda_f=1$.
  We must verify the
  pair of 2-cells displayed above are equal. Using that $\theta\cdot k=T\mu_f$,
  the verification takes the following form, where the first equality is the
  modification property for $\delta$ and the 2-naturality of $i$, the second is
  the interchange law in a 2-category, the third holds since $\delta\cdot i=1$,
  and the last holds since $\mu_f\cdot{L}f=1$.
  \begin{multline}
    \label{eq:123}
    \big(T\mu_f\cdot{}i_{Kf}\big)\big(T^2f\cdot{}\delta_A\cdot{}q_f\big)
    = \big(i_{TB}\cdot \mu_f \big)\big(\delta_B\cdot{}Tf\cdot{}q_f\big)
    = \big(\delta_B\cdot i_B\cdot{R}f\big)\big(Ti_B\cdot\mu_f\big)
    =\\
    = Ti_B\cdot{}\mu_f
    = (Ti_B\cdot{}\mu_f)(T\mu_f\cdot{}T{L}f\cdot{}q_f)
  \end{multline}

  It remains to verify the triangle identities of an adjunction. One of them is
  $m_A\cdot{}Tq_f\cdot{}\alpha=1$, equivalent to
  $m_A\cdot{}Tq_f\cdot{}\bar\alpha=1$, and by definition of $\bar\alpha$, equivalent to
  $m_A\cdot{}\bar\alpha_1=1$. This latter equality clearly holds, since $m_A\cdot{}\delta=1$,

  Up to now we have only used the hypothesis in the case when $v$ is an identity
  morphism. In order to prove the other triangle identity
  $\alpha\cdot{}(T{L}f)=1$ we shall need the hypothesis in its general form, more
  precisely, for $v={L}f$.  The triangle equality is equivalent to
  $\bar\alpha\cdot{}{L}f=1$, which holds since
  $\delta_A\cdot{}q_f\cdot{}{L}f=\delta_A\cdot{}i_A=1$ and $\mu_f\cdot{}{L}f=1$, finishing
  the proof.
\end{proof}
The proposition will be usually used in the following, less powerful form.
\begin{cor}
  \label{cor:2}
  A 2-monad $\mathsf{T}=(T,i,m)$ is simple if it is lax idempotent and composing
  with~\eqref{eq:118} induces a bijection between 2-cells $u\Rightarrow
  i_{f\downarrow g}\cdot{}v$ and $k\cdot{}u\Rightarrow k\cdot{}i_{f\downarrow g}\cdot{}v$, where
  $f\colon A\to B\leftarrow C:g$ are arbitrary morphisms.
\end{cor}
Let $\mu\colon{}h\cdot{}j\Rightarrow g$ be a left extension in a 2-category with comma
objects.  Recall that $\mu$ is a \emph{pointwise left extension} if, whenever
pasted with a comma object as depicted on the left hand side below, the resulting 2-cell is a left
extension.
Recall that if a 2-monad $\mathsf{T}=(T,i,m)$ is lax idempotent then the
identity 2-cell in~\eqref{eq:143} below exhibits $Tf$ as a left extension --~not necessarily a
pointwise extension~-- of $i_B\cdot{}f$ along $i_A$ --~Section~\ref{sec:kz-2-monads}.
\begin{equation}
  \label{eq:143}
  \xymatrixrowsep{.6cm}
  \diagram
  j\downarrow w\ar[r]\ar[d]
  \dtwocell<\omit>{^<-5>}
  &
  W\ar[d]^w
  \\
  X\ar[r]^-j\ar[dr]_g
  &
  Y\duppertwocell<0>^h{^<2>\mu}
  \\
  &
  Z
  \enddiagram
  \qquad
  \diagram
  A\ar[d]_f\ar[r]^{i_A}&
  TA\ar[d]^{Tf}\\
  B\ar[r]^{i_B}&TB
  \enddiagram
\end{equation}

\begin{thm}
  \label{thm:3}
  A lax idempotent 2-monad $\mathsf{T}$ is simple if it satisfies:
  \begin{itemize}
  \item the identity 2-cell on the right hand side of~\eqref{eq:143} exhibits
    $Tf$ as a pointwise left extension of $i_B\cdot{}f$ along $i_A$, for all
    $f$;
  \item and the components of the unit $i\colon{}1\to T$ are fully faithful.
  \end{itemize}
  \end{thm}
\begin{proof}
  We will verify the hypothesis of Corollary~\ref{cor:2}.
  Given a comma
  object $h\downarrow g$
  depicted on the left below, denote by $k\colon
  T(h\downarrow g)\to Th\downarrow Tg$ the comparison morphism. Given a
  morphism $u\colon X\to T(h\downarrow g)$, we consider the diagram
  on the right hand side, where the unlabelled 2-cell is a comma object.
  This pasting exhibits $(Td_n)\cdot{}u$ as a left extension, since $Td_n$ is a
  pointwise left extension.
  \begin{equation}
    \label{eq:117}
    \xymatrixrowsep{.6cm}
    \diagram
    {h\downarrow g}
    \ar[r]^-{d_1}\ar[d]_{d_0}\drtwocell<\omit>{^\gamma}
    &
    {B}
    \ar[d]^{g}
    \\
    {A}
    \ar[r]_-{h}
    &
    {C}
    \enddiagram
    \qquad
    \diagram
    \cdot\ar[r]^{e_1}\ar[d]_{e_0}\drtwocell<\omit>{^}&
    X\ar[d]^u\\
    h\downarrow g\ar[r]_-{i_{h\downarrow g}}\ar[d]_{d_n}&
    T(h\downarrow g)\ar[d]^{Td_n}\\
    \cod(d_n)\ar[r]^-{i}&
    T(\cod(d_n))
    \enddiagram
  \end{equation}
  Given a morphism $v\colon X\to h\downarrow g$, we will show that
  2-cells $\alpha\colon k\cdot{}u\Rightarrow k\cdot{} i_{h\downarrow g}\cdot{}v$ are in bijection
  with 2-cells $u\Rightarrow i_{h\downarrow g}\cdot{}v$.

  We begin by observing that 2-cells $\alpha$ are in bijection with
  pairs of 2-cells
  \begin{equation}
    \alpha_0\colon (Td_0)\cdot{}u\Rightarrow (Td_0)\cdot{}i_{h\downarrow g}\cdot{}v
    \qquad \text{and}\qquad
    \alpha_1:(Td_1)\cdot{}u\Rightarrow (Td_1)\cdot{}i_{h\downarrow g}\cdot{}v
  \end{equation}
  compatible with $T\gamma$ in the sense that
  $(Tg\cdot{}\alpha_1)(T\gamma\cdot{}u)=(T\gamma\cdot{}i_{h\downarrow g}\cdot{}v)(Th\cdot{}\alpha_0)$ holds.

  By the universal property of extensions, $\alpha_n$ is in bijection with
  2-cells $i\cdot{}d_n\cdot{}e_0\Rightarrow (Td_n)\cdot{}i_{h\downarrow g}\cdot{}v\cdot{}e_1=i_{\cod(d_n)}\cdot{}d_n\cdot{}v\cdot{}e_1$, and
  since $i$ has fully faithful components, with 2-cells $\beta_n\colon
  d_n\cdot{}e_0\Rightarrow d_n\cdot{}v\cdot{}e_1$. The compatibility between $\alpha_0$, $\alpha_1$,
  and $T\gamma$ translates into
  $(g\cdot{}\beta_1)(\gamma\cdot{}e_0)=(\gamma\cdot{}v\cdot{}e_1)(h\cdot{}\beta_0)$.
  By the universal property of $\gamma$, the pair $\beta_0$, $\beta_1$ is in
  bijection with 2-cells $e_0\Rightarrow v\cdot e_1$, and thus with 2-cells
  $i_{h\downarrow g}\cdot{}e_0\Rightarrow i_{h\downarrow g}\cdot{}v\cdot{}e_1$. Finally, by the
  description of $u$ as a left extension, these 2-cells are in
  bijection with 2-cells $u\Rightarrow i_{h\downarrow g}\cdot{}v$, as required.
\end{proof}
The theorem can be used to prove that, for a class of $\mathbf{Set}$-colimits,
the 2-monad on \Cat\ whose algebras are categories with chosen colimits of that
class is simple. Section~\ref{sec:exampl-compl-v} proves this fact in another
way, that applies to enriched categories.

\section{Example: completion of \V-categories under a class of colimits}
\label{sec:exampl-compl-v}
This section is divided in four parts. The first compiles the basic facts about
completions under a class of colimits that will be needed to prove. In the
second part, it is shown that the 2-monad
whose algebras are \V-categories with chosen colimits of a class is simple,
therefore inducing a lax orthogonal \textsc{awfs} $(\mathsf{L},\mathsf{R})$ on \VCat. The
third part exhibits the example when the colimits involved are initial objects.
The last part shows that the algebras for $\mathsf{R}$ are, at least when
$\V=\mathbf{Set}$, split opfibrations whose fibres are equipped with chosen
colimits and whose push forward functors strictly preserve
them. Intuition should dictate that
this type of split opfibration should coincide with the
$\mathsf{R}$-algebras; however, in general they do not, as we show at
the end of the section.

\begin{notation}
  In this section we will tend to think of categories enriched in \V\ as
  objects of a 2-category, in this case $\VCat$. Instead of denoting \V-categories
  by calligraphic letters, we opt to maintain the notation we have used for
  2-categories, where objects are denoted by capital roman letters and morphisms
  by lowercase letters. As a result, \V-categories will be denoted by $A$, $B$,
  etc, and \V-functors by $f\colon A\to B$, etc.
\end{notation}

\subsection{Completion under colimits}
\label{sec:compl-under-colim}
Throughout the section \V\ will be a base of enrichment, in our case, a complete
and cocomplete symmetric monoidal closed category. As argued
in~\cite{Kelly:BCECT}, the usual notion of colimit is not well adapted to the
context of enriched categories and must be extended to that of \emph{weighted
  colimit}. A weight is just a \V-functor $\phi\colon J^{\mathrm{op}}\to\V$ with $J$ a small
\V-category. Then a $\phi$-weighted colimit of a functor $G\colon J\to C$ is
expressed by a \V-natural isomorphism
\begin{equation}
  \label{eq:121}
  C(\operatorname{colim}(\phi,G),c)\cong[J^{\mathrm{op}},V](\phi,C(G-,C)).
\end{equation}

The free completion of a \V-category $C$ under small colimits can be
constructed as the \V-category $\mathscr{P}C$ with objects small presheaves --
ie \V-functors $C^\mathrm{op}\to\V$ that are a left Kan extension of
its own restriction to a small subcategory of $C^{\mathrm{op}}$ -- and enriched
homs given by $\mathscr PC(\phi,\psi)=\int_c[\phi c,\psi c]$. This extends to a
pseudomonad on \VCat, whose unit has components the Yoneda embedding $y_C\colon
C\to\mathscr PC$, and whose multiplication we denote by $m^{\mathscr P}$.
A number of properties of $\mathscr{P}C$, in
particular its completeness, are studied in~\cite{Day2007651}.
% In the same paper it is observed that
% $\mathscr Pf\colon\mathscr PC\to\mathscr PD$ has a right adjoint precisely when
% each of the presheaves $D(f-,d)$ are small; this right adjoint is necessarily
% given by $\psi\mapsto\psi f^{\mathrm{op}}$.

A class of colimits is a set of weights $\Phi=\{\phi_i\colon
J^{\mathrm{op}}_i\to\Cat\}_{i\in I}$. The free completion of $C$ under colimits
of the class $\Phi$, or $\Phi$-colimits, can be constructed as the smallest
full sub-\V-category of $\mathscr PC$ that is closed under $\Phi$-colimits and
contains the representable presheaves. We follow the notation
of~\cite{Kelly:BCECT} and denote this \V-category by $\Phi C$. One obtains a
pseudomonad, also denoted by $\Phi$, with unit the corestricted Yoneda embedding
$y_C$ and multiplication $m^\Phi$, together with a pseudomonad morphism
$\Phi\to\mathscr P$ that has fully faithful components.

Given a class of colimits $\Phi$, \cite{Kelly:MonadicityChosenColim}~shows the
existence of a 2-monad on $\VCat$, that we denote by $\mathsf{T}_\Phi$, whose
algebras are the \V-categories with chosen $\Phi$-colimits, and whose strict
morphisms are \V-functors that strictly preserve these. Furthermore, the same
article shows that this is a lax idempotent 2-monad. Earlier, less general,
versions of this monad appeared, for example, in~\cite{MR1474564}.

It is convenient to recall some aspects of the construction of
$\mathsf{T}_\Phi$ from~\cite{Kelly:MonadicityChosenColim}. Part of this
construction is an equivalence $t_A\colon T_\Phi A\to\Phi A$ for each \V-category
$A$, which form a pseudonatural transformation $\mathsf{T}\to\Phi$, and
moreover, a pseudomonad morphism.

\begin{lemma}
  \label{l:1}
  Denote by $\Phi$ a small class of \V-enriched colimits and the associated
  pseudomonad on \VCat. Let $f\colon A\to B$ be a fully faithful \V-functor into a
  $\Phi$-cocomplete \V-category, and denote by $\tilde f\colon\Phi(A)\to B$ a
  left Kan extension of $f$ along the corestricted Yoneda embedding $y_A\colon
  A\to\Phi A$. Then the morphisms
  \begin{equation}
    \label{eq:13}
    \Phi(A)(\phi,y_A(a))\longrightarrow B(\tilde f(\phi),f(a))
  \end{equation}
  induced by $\tilde f$
  are isomorphisms for all $\phi\in\Phi(A)$ and $a\in A$.
\end{lemma}
\begin{proof}
  The morphism~\eqref{eq:13} can be written as the composition of
  $\Phi(A)(\phi,f)$ from $\Phi(A)(\phi,y_A(a))$ to $\Phi(A)(\phi,B(f-,f(a)))$
  and the isomorphism between the latter and $B(\col(\phi,f),f(a))$. The result
  is an isomorphism since $f$ is full and faithful and $\tilde f(\phi)\cong\col(\phi,f)$.
\end{proof}

An explicit description of $\mathsf{T}_\Phi A$ is only possible in particular
instances. In theory, one can give an inductive description of the objects, but
in practice this is not very useful. Instead, we will use $\Phi A$ and its
relationship to $\mathsf{T}_\Phi A$.

\subsection{Simplicity of completion under a class colimits}
\label{sec:simpl-compl-under}
Before proving that the 2-monads $\mathsf{T}_{\Phi}$ are simple, we need the
following easy lemma.
\begin{lemma}
  \label{l:6}
  Suppose given commutative diagrams of\/ \V-functors whose horizontal arrows $u$,
  $v$ and $w$ are full and faithful.
  \begin{equation}
    \label{eq:122}
    \diagram
    A\ar[r]^-u\ar[d]_f&A'\ar[d]^{f'}\\
    C\ar[r]^-v&C'
    \enddiagram
    \qquad
    \diagram
    B\ar[r]^-w\ar[d]_g&B'\ar[d]^{g'}\\
    C\ar[r]^-v&C'
    \enddiagram
  \end{equation}
  Then, the \V-functor $h\colon (f\downarrow g)\longrightarrow (f'\downarrow g')$,
  defined on objects by $(a,b,\alpha)\mapsto(u(a),w(b),v(\alpha))$, is full and
  faithful.
\end{lemma}
\begin{proof}
  Of the routes one may take to prove this result, we choose a fairly direct one.
  The diagram displayed on the left exhibits $f\downarrow g$ as the comma category
  $(v\cdot f)\downarrow(v\cdot g)$; this is a direct consequence of the construction
  of comma categories and the fact that $v$ is full and faithful.
  \begin{equation}
    \label{eq:158}
    \diagram
    f\downarrow g\ar[d]\ar[r]\drtwocell<\omit>{}&A\ar[d]^f&\\
    B\ar[r]_g&C\ar[r]_v&C'
    \enddiagram
    \quad
    \xymatrixrowsep{.5cm}
    \diagram
    (f'\cdot u)\downarrow (g'\cdot w)\ar@{>->}[r]\ar@{>->}[d]
    \ar@{}[dr]|{\mathrm{pb}}&
    \bullet\ar[r]\ar@{>->}[d]\ar@{}[dr]|{\mathrm{pb}}&
    A\ar@{>->}[d]^{u}\\
    \bullet\ar@{>->}[r]\ar[d]\ar@{}[dr]|{\mathrm{pb}}&
    f'\downarrow g'\ar[r]\ar[d]\drtwocell<\omit>{}&
    A'\ar[d]^{f'}\\
    B\ar@{>->}[r]_-w&
    B'\ar[r]_{g'}&
    C'
    \enddiagram
  \end{equation}
  Using the commutativity of the diagrams in the statement, we see that
  $f\downarrow g$ can be constructed as $(f'\cdot u)\downarrow (g'\cdot w)$, and
  the latter comma category can be constructed from $f'\downarrow g'$ by taking
  pullbacks, as shown in the diagram on the right. Since full and faithful
  \V-functors are stable under pullback, all the arrows denoted $\rightarrowtail{}$
  are fully faithful functors. The composition of the isomorphism
  $f\downarrow g\cong (f'\cdot u)\downarrow(g'\cdot w)$ with the diagonal
  $(f'\cdot u)\downarrow(g'\cdot w)\longrightarrow f'\downarrow g'$ is precisely
  the \V-functor $f\downarrow g\longrightarrow f'\downarrow g'$ of the
  statement, which is, therefore, full and faithful.
\end{proof}

Let $\Phi$ be a small class of colimits, and $\mathsf{T}_\Phi$ the 2-monad on
\VCat\ whose algebras are small \V-categories with chosen colimits of the class
$\Phi$.
\begin{thm}
  \label{thm:6}
  The 2-monads $\mathsf{T}_\Phi$ are simple -- in the sense of Section
  \ref{sec:simple-2-monads} -- therefore inducing a lax orthogonal {\normalfont\textsl{\textsc{awfs}}}
  $(\mathsf{L}_\Phi,\mathsf{R}_\Phi)$ on \VCat.
\end{thm}
\begin{proof}%[Proof of Theorem~\ref{thm:6}]
  Let the 2-monad $\mathsf{T}$ in Corollary~\ref{cor:2} be $\mathsf{T}_\Phi$,
  and assume that we
  are given \V-functors $f$ and $g$ as in the statement of Corollary~\ref{cor:2}. The
  \V-category $Tf\downarrow Tg$ is $\Phi$-cocomplete, as the forgetful 2-functor
  from $\mathsf{T}$-algebras creates comma objects. The comparison morphism of the
  corollary is the left Kan extension of the \V-functor $h\colon f\downarrow
  g \to Tf\downarrow Tg$ induced by $i_{A}$ and $i_B$. Since $h$ is full
  and faithful, then $k$ is full and faithful on homs of the form
  $T(f\downarrow g)(u,i_{f\downarrow g}(v))$ by Lemma~\ref{l:1}, so we have
  indeed the bijection of 2-cells required in Corollary~\ref{cor:2}.
\end{proof}

\subsection{Completion under initial objects}
\label{sec:compl-under-init}

  Suppose that the base of enrichment is the category $\mathbf{Set}$ of sets, so
  we work with locally small categories, and that the class of colimits has only
  one member, $\Phi=\{\emptyset\to\mathbf{Set}\}$. Then, $\Phi$-colimits are
  initial objects, and the 2-monad $\mathsf{T}_\Phi$ can be described as having
  endo-2-functor $T_\Phi$ that sends a category $X$ to the category constructed
  by adding to $X$ an object $0$ and adding one arrow $0\to x$ for all $x\in
  X$. Then the morphism ${R}f\colon Kf\to B$, the right part of the
  factorisation of $f$, can be described as the split opfibration with fibre
  over $b\in B$ equal to $T_\Phi(f/b)$ and with push-forward functor
  $T_\Phi(f/b)\to T_\Phi(f/b')$ induced by a morphism $\beta\colon b\to b'$
  equal to $T_\Phi(\beta_*)$ where $\beta_*\colon f/b\to f/b'$ is the
  push-forward functor of the free split opfibration $f/B$.
  \begin{equation}
    \label{eq:161}
    \xymatrixrowsep{.5cm}
    \diagram
    Kf\ar[d]_{{R}f}\ar[r]\drtwocell<\omit>{}&T_\Phi A\ar[d]^{T_\Phi f}\\
    B\ar[r]_-{i_B}&T_\Phi B
    \enddiagram
  \end{equation}
  An $\mathsf{R}$-algebra is a split opfibration $A\to B$ whose fibres $A_b$,
  for $b\in B$, are categories equipped with a chosen initial object, and whose
  push-forward functors $\beta_*\colon A_b\to A_{b'}$, for $\beta\colon b\to b'$
  in $B$, strictly preserve the initial objects.

\subsection{Split opfibrations with fibrewise chosen
  $\Phi$-colimits }
\label{sec:monad-split-opfibr}
Given a small class of $\mathbf{Set}$-colimits $\Phi$, denote by
$\mathbf{Op}\mathbf{Fib}_s \text-\Phi\text-\mathbf{Colim}_s$ the 2-category with objects split
opfibrations in \Cat\ whose fibres are small categories with chosen colimits
of the class $\Phi$ and whose push-forward functors strictly preserve
these. Morphisms from $p\colon E\to B$ to $p'\colon E'\to B'$ are strict
morphisms $(h,k)\colon p\to p'$ of split fibrations --~indicated by the first
\emph{s} in the notation~-- such that the restriction of $h$
to fibres strictly preserves the chosen $\Phi$-colimits --~indicated by the
second \emph{s} in the notation. The 2-cells are those
of $\Cat^\two$.

\begin{thm}
  \label{thm:7}
  Let $\Phi$ be a class of $\mathbf{Set}$-colimits and $(\mathsf{L},\mathsf{R})$ be
  the lax orthogonal {\normalfont\textsl{\textsc{awfs}}} induced by the
  completion under $\Phi$-colimits.
  There is a 2-functor over $\Cat^\two$
  \begin{equation}
    \label{eq:163}
    \mathsf{R}\text-\mathrm{Alg}_s \longrightarrow
    \mathbf{Op}\mathbf{Fib}_s{\text-}\Phi{\text-}\mathbf{Colim}_s.
  \end{equation}
\end{thm}
\begin{proof}
  We have shown in Theorem~\ref{thm:6} that the 2-monad $\mathsf{T}_{\Phi}$,
  whose algebras are categories with chosen $\Phi$-colimits, is simple; this
  means that the free algebra adjunction $F_\Phi\dashv U_\Phi\colon
  \mathsf{T}_\Phi\text-\mathrm{Alg}\to\Cat^\two$ is simple. The
  2-monad $\mathsf{R}$ is, by construction, the codomain-preserving coreflection
  of $U_\Phi R' F_\Phi$, where $R'$ is the split opfibration 2-monad on
  $\mathsf{T}_\Phi\text-\mathrm{Alg}_s$. We are, thus, in a position of applying
  Theorem~\ref{thm:15} to deduce the existence of a 2-functor from
  $\mathsf{R}\text-\mathrm{Alg}_s$ to the 2-category $\mathbf{Op}\mathbf{Fib}_s$ of split
  opfibrations; this means that each $\mathsf{R}$-algebra is a split opfibration
  and each morphism of $\mathsf{R}$-algebras is a morphism of split
  opfibrations. It remains to prove that:
  \begin{enumerate}[label=(\alph*)]
  \item \label{item:5} the fibres of any $\mathsf{R}$-algebra are equipped with
    chosen $\Phi$-colimits;
  \item \label{item:13} the push-forward functors between fibres strictly
    preserve them;
  \item \label{item:14} any morphism of $\mathsf{R}$-algebras strictly preserves
    them.
  \end{enumerate}
  The fibre of
  $g\colon A\to B$ over $b\in B$ is the pullback of $g$ along
  $b\colon\mathbf{1}\to B$; thus, this fibre is an $\mathsf{R}$-algebra
  $A_b\to\mathbf{1}$. We know from Corollary~\ref{cor:9} that the restriction of
  $\mathsf{R}$ to $\Cat/\mathbf{1}\cong \Cat$ is isomorphic to
  $\mathsf{T}_{\Phi}$, so the fibres of $\mathsf{R}$-algebras are
  $\mathsf{T}_\Phi$-algebras, and the restriction of any morphism of
  $\mathsf{R}$-algebras to fibres is a morphism of
  $\mathsf{T}_\Phi$-algebras. This verifies the literals~\ref{item:5}
  and~\ref{item:14}.

  It remains to prove \ref{item:13}, ie that for any morphism $\beta\colon b\to
  b'$ in $B$ the push-forward functor $A_b\to A_{b'}$ preserves the chosen
  colimits. The strategy we follow to prove this claim is an usual one: it
  suffices to prove it for free
  $\mathsf{R}$-algebras and use that any $\mathsf{R}$-algebra is a canonical
  coequaliser of free ones.

  There is a split coequaliser in $\Cat/B$
  \begin{equation}
    \label{eq:285}
    \diagram
    K{R}{g}
    \ar@<12pt>[r]^-{\pi_{g}}\ar@<-3pt>[r]^-{K(p,1)}
    \ar@<-9pt>@{<-}[r]_{{L}({{R}g})}&
    Kg\ar[r]^-{p}\ar@<-5pt>@{<-}[r]_{{L}g}&
    A
    \enddiagram
  \end{equation}
  --~where $p$ denotes the $\mathsf{R}$-algebra structure of $g$~-- which then
  lifts to a (non-split) coequaliser in the 2-category of split
  opfibrations. Taking the fibre over $b\in B$ of this split coequaliser, we
  obtain a coequaliser in $\mathsf{T}\text-\mathrm{Alg}_s$ that splits in
  \Cat. In particular, for any functor $d$ into $A_b$
  \begin{equation}
    \label{eq:287}
    \col(\phi,d)=p_b(\col(\phi,({L}g)_b\cdot{}d))
  \end{equation}
  because $p$ strictly preserves the chosen colimits.
  Taking fibres over the domain $b$ and the codomain $b'$ of the morphism
  $\beta$ in $B$, we have a commutative square in
  \Cat\, where all the categories have chosen $\Phi$-colimits and the horizontal
  functors strictly preserve them.
  \begin{equation}
    \label{eq:286}
    \xymatrixrowsep{.4cm}
    \diagram
    (Kg)_b\ar[r]^-{p_b}\ar[d]_{\beta_*}&A_b\ar[d]^{\beta_*}\\
    (Kg)_{b'}\ar[r]^-{p_{b'}}&A_b
    \enddiagram
  \end{equation}
  Therefore, if the push-forward functors of $Rg$ preserve the chosen
  $\Phi$-colimits, then so do the ones of $g$, as shown by the following string
  of equalities. By commutativity of the square,
  $\beta_*\cdot{}p_b$ strictly preserves $\Phi$-colimits, and
  \begin{multline}
    \label{eq:288}
    \beta_*(\col(\phi,d))=\beta_*p_b(\col(\phi,({L}g)_b\cdot{}d))=
    \\=
    \col(\phi,\beta_*\cdot{}p_b\cdot{}({L}g)_b\cdot{}d)
    = \col(\phi,\beta_*\cdot{}d).
  \end{multline}
  The first equality holds by~\eqref{eq:287}, the second equality holds because
  the diagram~\eqref{eq:286} shows that $\beta_*\cdot p_b$ strictly preserves
  $\Phi$-colimits, and the last equality is a consequence of
  $p\cdot{L}g=1$.

  We now prove that the push-forward functors of a free $\mathsf{R}$-algebra
  $Rg$ strictly preserve chosen $\Phi$-colimits. By the description of $Kg$ as a
  comma object \eqref{eq:116}, its objects are triples $(x,b,\xi)$ where
  $x\in TA$, $b\in B$ and $\xi\colon (Tg)(x)\to i_B(b)$ is a morphism in
  $TB$. If we denote by $z_b\colon (Kg)_b\to Kg$ the inclusion of the fibre over
  $b\in B$ and
  $q_g\colon Kg\to TA$ the projection of the comma
  object, we showed in Lemma~\ref{l:27} that $q_g\cdot{}z_b\colon (Kg)_b\to TA$
  strictly preserves $\Phi$-colimits. It is clear that the triangle on the left
  hand side commutes, since $\beta_*(x,b,\xi)=(x,b',i_B(\beta)\cdot{}\xi)$.
  \begin{equation}
    \label{eq:289}
    \xymatrixrowsep{0.4cm}
    \diagram
    (Kg)_b\ar[d]_{\beta_*}\ar[r]^-{q_g\cdot{}z_b}&TA\\
    (Kg)_{b'}\ar[ur]_{q_g\cdot{}z_{g'}}&
    \enddiagram
    \qquad
    \diagram
    Tg/i_B(b)\ar[r]^-{\mathrm{pr}_b}\ar[d]_{\beta_*}&TA\\
    Tg/i_B(b')\ar[ur]_{\mathrm{pr}_{b'}}&
    \enddiagram
  \end{equation}
  But $(Kg)_b$ is the slice category $Tg/i_B(b)$, and $q_g\cdot{}z_b$ is the
  projection into $TA$, so we have a commutative triangle as in the right hand
  side. We can now apply the Lemma~\ref{l:16}, that follows the present proof,
  to deduce that $\mathrm{pr}_{b}=\beta_*\cdot\mathrm{pr}_{b'}$ preserves and
  creates chosen $\Phi$-colimits, so $\beta_*$ strictly preserves chosen
  $\Phi$-colimits.

  This concludes the proof, since the 2-cells are
  automatically taken care of because the two 2-categories
  of the statement are locally full and faithful over $\Cat^\two$.
\end{proof}

\begin{lemma}
  \label{l:16}
  \begin{enumerate}
  \item Let $C$ be a category with $\Phi$-colimits, $H\colon C\to E$ a
    $\Phi$-cocontinuous functor, $e$ an object of $E$ and $Q\colon H/e\to C$ the
    projection. For any functor $D\colon J\to H/e$ and any colimiting cylinder
    $\eta\colon\phi\Rightarrow C(QD-,c)$ with $\phi\in\Phi$, there exists a
    unique $\epsilon \colon Hc\to e$ in $E$ and a unique colimiting cylinder
    $\nu\colon \phi\Rightarrow H/e(D-,(c,\epsilon))$ such that $Q(\nu)=\eta$.

  \item Moreover, if $C$ is equipped with chosen $\Phi$-colimits, then there
    exists a unique choice of $\Phi$-colimits on $H/e$ that is strictly
    preserved by $Q$.

  \item Suppose $S\colon A\to H/e$ is a functor, where $A$ has chosen
    $\Phi$-colimits. Then $S$ strictly preserves $\Phi$-colimits if and only if
    $QS\colon A\to C$ does so.
  \end{enumerate}
\end{lemma}
\begin{proof}
  Since $H$ preserves colimits, the top horizontal natural transformation in the
  diagram is a
  colimiting cylinder. The functor $D\colon J\to H/e$ can be given by a natural
  transformation $\delta$ from the constant functor on the terminal set to
  $D(HQD-,e)$. By the universal property of the colimit $Hc$, there exists a
  unique morphism $\epsilon\colon Hc\to e$ that makes the diagram commute.
  \begin{equation}
    \label{eq:161}
    \xymatrixrowsep{.5cm}
    \diagram
    \phi\ar@{=>}[r]^-{\eta} \ar@{=>}[dr]_{!}&
    C(QD-,c)\ar@{=>}[r]^-{H}&
    E(HQD-,Hc)\ar@{..>}[d]^{E(1,\epsilon)}\\
    &\Delta 1\ar@{=>}[r]^-{\delta}&
    E(HQD-,e)
    \enddiagram
  \end{equation}
  The functor $H/e(D-,(c,\epsilon))$ is the equaliser of the natural
  transformations
  \begin{gather}
    C(QD-,c)\xRightarrow{H}E(HQD-,Hc)\xRightarrow{E(1,\epsilon)} E(HQD-,e)
    \\
    C(QD-,c)\Longrightarrow \Delta 1\xRightarrow{\delta} E(HQD-,e)
  \end{gather}
  from where it follows that $\eta$ factors uniquely through a certain natural
  transformation $\nu\colon \phi\Rightarrow H/e(D-,(c,\epsilon))$. This
  transformation can easily be shown to have the universal property of a
  colimiting cylinder, a verification that we leave to the reader. In
  particular, $(c,\epsilon)$ is a colimit of $D$ weighted by $\phi$.

  To prove the second part of the statement, if $\eta$ exhibits $c$ as
  $\col(\phi,DQ)$, then we can choose the colimit $\col(\phi,D)$ as
  $(c,\epsilon)$, and this is the unique possible choice that makes $Q$ preserve
  this colimit in a strict way, by the argument of the previous paragraph.
  The last part of the statement easily follows from the second part.
\end{proof}

In many instances, the 2-functor of Theorem~\ref{thm:7} is an isomorphism. For example,
it is not hard to verify this when $\Phi$ is the class for initial objects
$\{\emptyset\to\mathbf{Set}\}$; see Section~\ref{sec:compl-under-init}.

\begin{prop}
  \label{prop:6}
  The 2-functor of Theorem~\ref{thm:7} is not always an isomorphism.
\end{prop}
\begin{proof}
  To save space, let us write $\mathscr F$ instead of
  $\mathbf{Op}\mathbf{Fib}_s\text-\Phi\text-\mathbf{Colim}_s$. One can without much
  problem show that the forgetful 2-functor $\mathscr F\to\Cat^\two$ is monadic,
  but instead we will consider the free object of $\mathscr F$ over a functor
  $g\colon A\to \mathbf{1}+\mathbf{1}$; the codomain is the discrete category
  with two objects that we denote by $\ast$ and $\bullet$. Another way of
  describing $g$ is as a pair of categories $A_\ast$ over $\{\ast\}$ and
  $A_\bullet$ over $\{\bullet\}$. It is not hard to see that the free object of
  $\mathscr F$ on $g$ is $\tilde g\colon \tilde A\to \mathbf{1}+\mathbf{1}$,
  with $\tilde A_\ast=T_{\Phi}(A_\ast)$ and $\tilde
  A_\bullet=T_{\Phi}(A_\bullet)$; this is due to the fact that a split
  opfibration over a discrete category amounts to just a functor. The details of
  this point are left to the reader. The universal property of $\tilde g$ implies
  the existence of a unique morphism $(h,1)\colon \tilde g\to Rg$ such that
  \begin{equation}
    \label{eq:164}
    \xymatrixrowsep{.5cm}
    \diagram
    A\ar[r]^{{L}g}\ar[d]_g&Kg\ar[d]_{Rg}\\
    \mathbf{1}+\mathbf{1}\ar@{=}[r]&\mathbf{1}+\mathbf{1}
    \enddiagram
    =
    \diagram
    A_\ast+A_\bullet\ar[r]^-{i_{A_\ast}+i_{A_\bullet}} \ar[d]&
    T_\Phi(A_\ast)+T_{\Phi}(A_\bullet)\ar[r]^-h\ar[d]&
    Kg\ar[d]^{Rg}\\
    \mathbf{1}+\mathbf{1}\ar@{=}[r]&
    \mathbf{1}+\mathbf{1}\ar@{=}[r]&
    \mathbf{1}+\mathbf{1}
    \enddiagram
  \end{equation}
  If the 2-functor of Theorem~\ref{thm:7} were an isomorphism, the morphism $h$
  would be an isomorphism, as both $\tilde g$ and $Rg$ would be the free object
  on $g$. To complete the proof we must give an example where $h$ is not an
  isomorphism.

  Consider the class of colimits with one sole element $\Phi=\{\Delta\emptyset\colon \mathbf
  1\to\mathbf{Set}\}$ consisting of the functor that picks out the empty set. The colimit of a functor $v\colon1\to \mathbf{Set}$ that
  picks out a set $v$, weighted by $\Delta\emptyset$ --~known as the \emph{tensor product} of $v$
  by $\emptyset$~-- is $\col(\Delta\emptyset,v)=\emptyset$. The completion
  of a small category $A$ under these colimits consists of the full subcategory
  $\Phi A\subseteq[A^{\mathrm{op}},\mathbf{Set}]$ defined by the
  representables together with the initial object. A choice of $\Phi$-colimits
  on a category $A$ amounts to an assignment of an initial object $0(a)\in
  A$ for each object $a\in A$.

  We can explicitly describe the 2-monad
  $\mathsf{T}$ associated to $\Phi$. If $A$ is a category, let $T_\Phi(A)$ have
  objects of the form $(a,n)\in\operatorname{ob}A\times\mathbb{N}$, and have
  morphisms defined by the following two clauses: there is a full and faithful
  functor $i_A\colon A\to T_\Phi(A)$ given on objects by $a\mapsto (a,0)$; and,
  each object $(a,n)$ for $n>0$ is an initial object. We equip $T_\Phi(A)$ with
  the chosen $\Phi$-colimits given by $0(a,n)=(a,n+1)$.

  The fibre $(Kg)_\ast$ of $Rg$ over $\ast$ is $T_{\Phi}(A)\downarrow
  i_{\mathbf{1}+\mathbf{1}}(\ast)$. In particular some of its objects are of the
  form $((a,n),\xi)$ where $a\in \operatorname{ob}A$, $n>0$ and $\xi\colon T_\Phi(g)(a,n)=(g(a),n)\to
  i_{\mathbf{1}\mathbf{1}}(\ast)$ is a morphism in
  $T_\Phi(\mathbf{1}+\mathbf{1})$. The domain of $\xi$ is an initial
  object, so $\xi$ carries no information at all. On the other hand, the
  restriction of the morphism $h$ to fibres $h_\ast\colon T_\Phi(A_\ast)\to
  (Kg)_\ast$ does not reach objects of form $((a,n),\xi)\in (Kg)_\ast$ unless $a\in
  A_\ast$. Therefore, $h_\ast$, and thus $h$, is not surjective on objects,
  completing the proof.
\end{proof}

\section{Further work and examples}
\label{sec:furth-work-exampl}
There are a number of examples and theoretical questions that have been left out
of the present article and will benefit from a fuller explanation in forthcoming
companion articles. Examples pertaining to the world of topological spaces will
be treated in \emph{Lax orthogonal factorisations in topology}. As a way of
illustration, we mention the lax orthogonal \textsc{awfs} arising from the
filter monad on the category of $T_0$ topological spaces; this \textsc{awfs} has an
underlying \textsc{wfs} that was mentioned in our Introduction and studied in~\cite{MR2927175}, where information
about the filter monad can be found. This factorisation $f=Rf\cdot Lf$ has the
property that $Lf$ is a topological embedding and $Rf$ is ``fibrewise
a continuous lattice'' in the appropriate sense.

The main theoretical aspect of lax orthogonal \textsc{awfs}s left out from the
present article is the cofibrant generation thereof. This will have a full
treatment in the forthcoming paper \emph{Cofibrantly}
\textsl{\textsc{kz}}-\emph{generated algebraic weak factorisation systems}.

\bibliographystyle{abbrv}
\bibliography{kznwfs.bib}

\begin{thebibliography}{10}

\bibitem{BKP}
R.~Blackwell, G.~M. Kelly, and A.~J. Power.
\newblock Two-dimensional monad theory.
\newblock {\em J. Pure Appl. Algebra}, 59(1):1--41, 1989.

\bibitem{MR3393453}
J.~Bourke and R.~Garner.
\newblock Algebraic weak factorisation systems {I}: {A}ccessible {AWFS}.
\newblock {\em J. Pure Appl. Algebra}, 220(1):108--147, 2016.

\bibitem{MR2927175}
F.~Cagliari, M.~M. Clementino, and S.~Mantovani.
\newblock Fibrewise injectivity and {K}ock-{Z}\"oberlein monads.
\newblock {\em J. Pure Appl. Algebra}, 216(11):2411--2424, 2012.

\bibitem{MR779198}
C.~Cassidy, M.~H{\'e}bert, and G.~M. Kelly.
\newblock Reflective subcategories, localizations and factorization systems.
\newblock {\em J. Austral. Math. Soc. Ser. A}, 38(3):287--329, 1985.

\bibitem{MR0367013}
A.~Day.
\newblock Filter monads, continuous lattices and closure systems.
\newblock {\em Canad. J. Math.}, 27:50--59, 1975.

\bibitem{Day2007651}
B.~J. Day and S.~Lack.
\newblock Limits of small functors.
\newblock {\em Journal of Pure and Applied Algebra}, 210(3):651 -- 663, 2007.

\bibitem{MR1641443}
M.~H. Escard{\'o}.
\newblock Properly injective spaces and function spaces.
\newblock {\em Topology Appl.}, 89(1-2):75--120, 1998.
\newblock Domain theory.

\bibitem{MR2506256}
R.~Garner.
\newblock Understanding the small object argument.
\newblock {\em Appl. Categ. Structures}, 17(3):247--285, 2009.

\bibitem{MR2283020}
M.~Grandis and W.~Tholen.
\newblock Natural weak factorization systems.
\newblock {\em Arch. Math. (Brno)}, 42(4):397--408, 2006.

\bibitem{Kelly:BCECT}
G.~M. Kelly.
\newblock {\em Basic concepts of enriched category theory}, volume~64 of {\em
  London Mathematical Society Lecture Note Series}.
\newblock Cambridge University Press, Cambridge, 1982.

\bibitem{Kelly:Prop-like}
G.~M. Kelly and S.~Lack.
\newblock On property-like structures.
\newblock {\em Theory Appl. Categ.}, 3:No. 9, 213--250 (electronic), 1997.

\bibitem{Kelly:MonadicityChosenColim}
G.~M. Kelly and S.~Lack.
\newblock On the monadicity of categories with chosen colimits.
\newblock {\em Theory Appl. Categ.}, 7:No. 7, 148--170 (electronic), 2000.

\bibitem{MR1474564}
G.~M. Kelly and I.~J. Le~Creurer.
\newblock On the monadicity over graphs of categories with limits.
\newblock {\em Cahiers Topologie G\'eom. Diff\'erentielle Cat\'eg.},
  38(3):179--191, 1997.

\bibitem{Kock:KZmonads}
A.~Kock.
\newblock Monads for which structures are adjoint to units.
\newblock {\em J. Pure Appl. Algebra}, 104(1):41--59, 1995.

\bibitem{MR2854177}
S.~Lack and M.~Shulman.
\newblock Enhanced 2-categories and limits for lax morphisms.
\newblock {\em Adv. Math.}, 229(1):294--356, 2012.

\bibitem{MR0223432}
D.~G. Quillen.
\newblock {\em Homotopical algebra}.
\newblock Lecture Notes in Mathematics, No. 43. Springer-Verlag, Berlin-New
  York, 1967.

\bibitem{MR0404073}
D.~Scott.
\newblock Continuous lattices.
\newblock In {\em Toposes, algebraic geometry and logic ({C}onf., {D}alhousie
  {U}niv., {H}alifax, {N}. {S}., 1971)}, pages 97--136. Lecture Notes in Math.,
  Vol. 274. Springer, Berlin, 1972.

\bibitem{MR0424896}
V.~Z{\"o}berlein.
\newblock Doctrines on {$2$}-categories.
\newblock {\em Math. Z.}, 148(3):267--279, 1976.

\end{thebibliography}
\end{document}